\documentclass[a4paper]{article}

\usepackage[all]{xy}
\usepackage{amsmath,amssymb}
\usepackage{mathrsfs}

\usepackage{amsthm}
\newtheorem{theorem}{Theorem}[section]
\newtheorem{prop}[theorem]{Proposition}
\newtheorem{lemma}[theorem]{Lemma}
\newtheorem{conj}[theorem]{Conjecture}

\newtheorem{mainThm}{Theorem}

\newtheorem*{claim}{Claim}

\newtheorem{cor}[theorem]{Corollary}

\theoremstyle{definition}
\newtheorem{defini}[theorem]{Definition}
\newtheorem{assum}[theorem]{Assumption}
\newtheorem{example}[theorem]{Example}
\newtheorem{remark}[theorem]{Remark}

\def\calA{\mathscr{A}}

\def\calE{\mathscr{E}}
\def\calF{\mathscr{F}}

\def\calH{\mathscr{H}}
\def\calI{\mathscr{I}}

\def\calK{\mathscr{K}}

\def\calS{\mathscr{S}}

\def\Im{\mathrm{Im}\,}
\def\deg{\mathrm{deg}}

\def\Conf{\mathrm{Conf}}
\def\bConf{\overline{\Conf}}
\def\bC{\overline{C}}
\def\dR{\mathrm{dR}}

\def\Z{\mathbb{Z}}
\def\Q{\mathbb{Q}}
\def\C{\mathbb{C}}
\def\R{\mathbb{R}}
\def\F{\mathbb{F}}
\def\End{\mathrm{End}}
\def\Hom{\mathrm{Hom}}
\def\Tr{\mathrm{Tr}}
\def\ve{\varepsilon}

\def\odd{\mathrm{odd}}
\def\Sym{\mathrm{Sym}}

\def\tbigwedge{\textstyle\bigwedge}
\def\Map{\mathrm{Map}}
\def\pr{\mathrm{pr}}

\usepackage{url}
\usepackage{graphicx} 
\usepackage[setpagesize=false]{hyperref} 
\usepackage{mathtools} 
\mathtoolsset{showonlyrefs=true} 
\newcommand{\floor}[1]{\left\lfloor #1 \right\rfloor}
\allowdisplaybreaks

\begin{document}
\normalsize

\title{Theta-invariants of $\Z\pi$-homology equivalences to spherical 3-manifolds}
\author{Hisatoshi Kodani and Tadayuki Watanabe}
\date{\small\it Dedicated to Professor Tomotada Ohtsuki on the occasion of his 60th birthday.}

\maketitle

\begin{abstract}
We propose a method to study $\Z\pi$-homology equivalences from
 3-manifolds to a fixed spherical 3-manifold by means of a formulation of the $\Theta$-invariant of 3-manifolds by Axelrod--Singer, Kontsevich, and Bott--Cattaneo for the group ring of $\pi$.
 The $\Theta$-invariants originate from perturbative Chern--Simons theory, and are defined by integrals over configuration spaces of two points with local systems and by choosing some invariant tensors. We compute upper bounds of the dimensions of the space spanned by the $\Theta$-invariants and of that spanned by Garoufalidis and Levine's finite type invariants of type 2. The computation is based on representation theory of finite groups. 
\end{abstract}

\noindent
\section{Introduction}

\newcommand{\fig}[1]
        {\raisebox{-0.5\height}
                 {\includegraphics{#1}}
        }
\def\Span{\mathrm{span}}

We fix a connected, closed, oriented smooth 3-manifold $N$. Let $\pi=\pi_1N$ and let $\widetilde{N}$ denote the universal cover of $N$. For a commutative ring $R$, let $R\pi$ denote the group algebra associated to $\pi$.
\begin{defini}[$\Z\pi$-homology equivalence]
\begin{enumerate}
\item A degree 1 map $f\colon M\to N$ from another closed oriented 3-manifold $M$ is said to be a {\it $\Z\pi$-homology equivalence} if it induces an isomorphism $f_*\colon H_*(\widetilde{M})\to H_*(\widetilde{N})$, where $\widetilde{M}$ is the $\pi$-covering of $M$ obtained by pulling back $\widetilde{N}$ by $f$. 

\item Two $\Z\pi$-homology equivalences $f\colon M\to N$ and $f'\colon M'\to N$ are {\it diffeomorphically equivalent} if there exists a diffeomorphism $g\colon M\to M'$ such that $f$ is homotopic to $f'\circ g$. 

\item Let $\calH(N)$ be the set of all diffeomorphism equivalence classes of $\Z\pi$-homology equivalences $f\colon M\to N$.
\end{enumerate}
\end{defini}

Understanding the set $\calH(N)$ is a fundamental problem for $\Z\pi$-homology equivalences, and it was extensively studied by Garoufalidis and Levine in \cite{GL} (see \cite[\S{2.3}]{GL} for the motivation from surgery theory). We will mention in Appendix~\ref{s:3-mfd-over-Kpi1} that a $\Z\pi$-homology equivalence can be considered as a special case of a 3-manifold over $K(\pi,1)$. 

In this paper, we propose a method to study $\calH(N)$ based on the filtration(s) and diagrams given in \cite{GL}. 
The first result of this paper is to give a formulation of an invariant of $\Z\pi$-homology equivalences to spherical 3-manifolds which unifies the ``$\Theta$-invariants'' for acyclic local systems  (\cite{AS,Kon,BC,CS}) induced from those on $N$ by $\Z\pi$-homology equivalences. 
By an {\it acyclic local system $A$ on $N$}, we mean a $\C\pi$-module $A$ such that $H_*(N;A):=H(S_*(\widetilde{N};\C)\otimes_{\C\pi} A)=0$, where $S_*(-;\C)$ is the $\C$-complex of singular chains. In \cite[Theorem~2.5]{BC}, Bott and Cattaneo defined an invariant $I_{(\Theta,\rho_1,\rho_2)}$ of a framed 3-manifold associated to an acyclic local system, in the spirit of the Chern--Simons perturbation theory of \cite{AS,Kon} at flat connections, but in a purely topological, and more general setting (see \S\ref{ss:Bott-Cattaneo} for more detail). It is a topological invariant of a framed 3-manifold $(M,\tau)$ that may depend on the choice of the isomorphism class of an acyclic local system $A$ on $M$ and $\pi_1M$-equivariant linear maps $\rho_1\colon \C\to A^{\otimes 3}$, $\rho_2\colon A^{\otimes 3}\to \C$. There may be infinitely many choices of $(A,\rho_1,\rho_2)$ for each $M$. For $\pi=\pi_1N$, let 
\[ \begin{split}
  \calS_\Theta^\odd(\C\pi):=\bigl(\Sym^3\,\C\pi\bigr)_{\pi\times \pi},
\end{split} \]
where we consider $\C\pi$ as a left $\pi\times \pi$-module by $(g,h)\cdot x\mapsto gxh^{-1}$ and we take the coinvariants, namely, the quotient by the action (\S\ref{ss:rep}). 
We also consider the space $\calS_\Theta^{\odd}(\mathrm{Ker}\,\ve)$ or $\calS_\Theta^{\odd}(W)$ defined similarly as above by replacing $\C\pi$ with the $\pi\times \pi$-submodule $\mathrm{Ker}\,\ve$ for the augmentation map $\ve\colon \C\pi\to \C$, which is defined by $\ve(\sum_{g\in \pi}a_g g)=\sum_{g\in \pi}a_g$, or a $\pi\times\pi$-submodule $W$ of $\mathrm{Ker}\,{\ve}$.

In this paper, we consider the $\Theta$-invariant $J_{(\Theta,\rho_1,\rho_2)}$ as in \cite{CS}, defined in a style similar to Kontsevich and Bott--Cattaneo (see \S\ref{ss:Bott-Cattaneo}).
\begin{mainThm}[Theorem~\ref{thm:unframed-inv}, \ref{thm:recover-I}, Corollary~\ref{cor:inv-framed}]\label{thm:Z_universal}
Let $N$ be a spherical 3-manifold, and let $\pi=\pi_1N$. Then, configuration space integrals for local systems give a well-defined invariant 
\[ \widehat{Z}_\Theta^\odd\colon \calH(N)\to \calS_\Theta^\odd(\mathrm{Ker}\,\ve) \]
of $\Z\pi$-homology equivalences to $N$. It is universal among the $\Theta$-invariant $J_{(\Theta,\rho_1,\rho_2)}(f,A)$ for finite dimensional acyclic local systems $A$ on $N$ and $\pi$-equivariant linear maps $\rho_1\colon \C\to A^{\otimes 3}$, $\rho_2\colon A^{\otimes 3}\to \C$.
\end{mainThm}

Note that we consider here local systems on 3-manifolds which factor through those on a spherical 3-manifold $N$. We emphasize that the definition and the argument for the topological invariance of the configuration space integral are the same as those of (\cite{Kon,BC,KuTh,Les,CS}) and are not new (more recent works on $\Theta$-invariants for acyclic local systems are found in \cite{CMW,Sh,Wer,KL}). What is new in this paper is to use an algebraic operation to unify the $\Theta$-invariants for all finite dimensional local systems on $N$.

We will see later that the image of $\widehat{Z}_\Theta^\odd$ is included in a subspace smaller than $\calS_\Theta^\odd(\mathrm{Ker}\,\ve)$. Let
\[ \begin{split}
  \calA_\Theta^\odd(\C\pi):=\Bigl(\bigl(\Sym^3\,\C\pi\bigr)_{\pi\times \pi}\Bigr)^{\Z_2},
\end{split} \]
the subspace of $\calS_\Theta^\odd(\C\pi)$ of $\Z_2$-invariants (\S\ref{ss:rep}).
The action of $\Z_2$ on $\bigl(\Sym^3\,\C\pi\bigr)_{\pi\times \pi}$ is that induced by the involution $x\otimes y\otimes z\mapsto x^{-1}\otimes y^{-1}\otimes z^{-1}$ ($x,y,z\in\pi$) and we take the invariant. We also consider the space $\calA_\Theta^{\odd}(W)$ defined similarly as above by replacing $\C\pi$ with the $\pi\times \pi$-submodule $W$ of $\C\pi$ that is invariant under the $\Z_2$-action. We will see in Proposition~\ref{prop:graph-inv} that the space $\calA_\Theta^\odd(\C\pi)$ is isomorphic to the space of two-loop $\pi$-decorated graphs, which is relevant to finite type invariants of $\Z\pi$-homology equivalences (\cite{GL}). The following theorem gives an upper bound of the dimension of the span of the $\Theta$-invariants.

\begin{mainThm}[Proposition~\ref{prop:Z2-invariance}, Theorem~\ref{thm:upper-bound2}]\label{thm:upper-bound}
The image of $\widehat{Z}_\Theta^\odd$ is included in $\calA_\Theta^\odd(\mathrm{Ker}\,\ve)$. Furthermore, let $\calI_\Theta(N)$ denote the subspace of $\Map(\calH(N),\C)$ spanned by the $\Theta$-invariants $J_{(\Theta,\rho_1,\rho_2)}(-,A)$ for finite dimensional acyclic local systems $A$ on $N$ and $\pi$-equivariant linear maps $\rho_1\colon \C\to A^{\otimes 3}$, $\rho_2\colon A^{\otimes 3}\to \C$. Then there is a canonical epimorphism 
\[ \Hom_\C(\calA_\Theta^\odd(\mathrm{Ker}\,\ve),\C)\to \calI_\Theta(N). \]
\end{mainThm}

The next result of this paper is the computation of the dimensions of the spaces $\calA_\Theta^\odd(\C\pi)$ and $\calA_\Theta^\odd(\mathrm{Ker}\,\ve)$ for the fundamental groups $\pi$ of spherical 3-manifolds. It is known that the fundamental groups of spherical 3-manifolds have been classified.
\begin{theorem}[e.g. \cite{Sav02}]\label{thm:classification}
Let $\pi$ be the fundamental group of a spherical 3-manifold. Then $\pi$ is isomorphic to one of the following.
\begin{itemize}
\item[\rm (a)] $\Z_n$, the cyclic group of order $n$, $n\geq 1$.
\item[\rm (b)] $\Z_m\times D_{4p}^*$, where $D_{4p}^*=\langle x,y\mid x^2=(xy)^2=y^p\rangle$ is the binary dihedral group of order $4p$, $m,p>0$, and $(m,2p)=1$. 
\item[\rm (c)] $\Z_m\times D_{2^{k+2}p}'$, where $D_{2^{k+2}p}'=\langle x,y\mid x^{2^{k+2}}=1,y^p=1,xy^{-1}=yx\rangle$, $m>0$, $k\geq 0$, $p\geq 3$ odd, and $(m,2p)=1$.
\item[\rm (d)] $\Z_m\times T^*$, where $T^*=\langle a,b\mid (ab)^2=a^3=b^3\rangle$ is the binary tetrahedral group of order 24, $m>0$, and $(m,6)=1$. 
\item[\rm (e)] $\Z_m\times T_{8\cdot 3^k}'$, where $T_{8\cdot 3^k}'=\langle x,y,z\mid x^2=(xy)^2=y^2, zxz^{-1}=y,zyz^{-1}=xy,z^{3^k}=1\rangle$, $m>0$, $(m,6)=1$, and $k\geq 1$.
\item[\rm (f)] $\Z_m\times O^*$, where $O^*=\langle a,b\mid (ab)^2=a^3=b^4\rangle$ is the binary octahedral group of order 48, $m>0$, and $(m,6)=1$.
\item[\rm (g)] $\Z_m\times I^*$, where $I^*=\langle a,b\mid (ab)^2=a^3=b^5\rangle$ is the binary icosahedral group of order 120, $m>0$, and $(m,30)=1$.
\end{itemize}
\end{theorem}
It may be useful to make a historical remark on the classification of the fundamental groups of spherical 3-manifolds. The above list is, in fact, the classification of subgroups of $\mathrm{SO}(4)$ which act on $3$-sphere $S^3 \subset \R^4$ fixed point free given in \cite[Section 7.5]{Wol11}. In \cite[Section 1.2]{Sav02}, in addition to the above list, it is mentioned that other candidates of finite groups which may act on $S^3$ fixed point free are considered by Milnor (\cite{Mil}) such as $Q(8n,k,\ell)$. However, now that Thurston's geometrization conjecture has been proved by Perelman (\cite{pere}), it shows that these candidates are excluded from the list and the groups listed in Theorem~\ref{thm:classification} are the only ones which are fundamental groups of spherical $3$-manifolds.

We obtain the following result by using the character theory of finite groups. 
\begin{mainThm}[Proof in \S\ref{s:computation_dim}] \label{thm:dim-formula}
We have a complete list of the values of $\dim\calA_\Theta^\odd(\C\pi)$ and $\dim\calA_\Theta^\odd(\mathrm{Ker}\,\ve)$ for all the groups $\pi$ of Theorem~\ref{thm:classification}. More concretely, the values are given as follows.
	\begin{enumerate}
	\item[\rm (a)] When $\pi = \Z_n$, 
	\begin{equation}
	\begin{split}
    &\dim \calA_\Theta^\odd(\C\pi)= p_3(n) = \begin{cases}
        \frac{1}{12} n^2 + \frac{1}{2}n + 1 & (n \equiv 0 \bmod 2, n \equiv 0 \bmod 3),\\
        \frac{1}{12} n^2 + \frac{1}{2}n + \frac{2}{3} & (n \equiv 0 \bmod 2, n \not \equiv 0 \bmod 3),\\
         \frac{1}{12} n^2 + \frac{1}{2}n + \frac{3}{4} & (n \not \equiv 0 \bmod 2, n \equiv 0 \bmod 3),\\
          \frac{1}{12} n^2 + \frac{1}{2}n + \frac{5}{12} & (n \not \equiv 0 \bmod 2, n \not \equiv 0 \bmod 3),
    \end{cases}\\
	&\dim \calA_\Theta^\odd(\mathrm{Ker}\,\ve)= p_3(n-3)= \begin{cases}
        \frac{1}{12} n^2  & (n \equiv 0 \bmod 2, n \equiv 0 \bmod 3),\\
        \frac{1}{12} n^2 - \frac{1}{3} & (n \equiv 0 \bmod 2, n \not \equiv 0 \bmod 3),\\
         \frac{1}{12} n^2 + \frac{1}{4} & (n \not \equiv 0 \bmod 2, n \equiv 0 \bmod 3),\\
          \frac{1}{12} n^2 - \frac{1}{12} & (n \not \equiv 0 \bmod 2, n \not \equiv 0 \bmod 3).
    \end{cases}
    \end{split}
\end{equation}
Here, for an integer $m \geq 0$, $p_3(m)$ denotes the number of partitions of $m$ into at most three parts, namely, the number of integer solutions of the equation $x + y +z = m$ $(0 \leq x \leq y \leq z )$, and we set $p_3(m) =0$ for $m<0$.
	\item[\rm (b)] $(1)$ When $\pi = \Z_m \times D_{4p}^{\ast}$ where $m \geq 1$, $p>0$ even, and $(m, 2p)=1$,	
	\begin{equation}
	\begin{split}
	&\dim \calA_\Theta^\odd(\C\pi) \\
	&= \begin{cases}
        \frac{1}{6} \, m^{2} p^{2} + \frac{1}{2} \, m^{2} p + \frac{2}{3} \, m^{2} + \frac{3}{2} \, m p + \frac{1}{6} \, p^{2} + m + \frac{1}{2} \, p + 1 & (p \not \equiv 0 \bmod 3, m \not \equiv 0 \bmod 3),\\
        \frac{1}{6} \, m^{2} p^{2} + \frac{1}{2} \, m^{2} p + \frac{2}{3} \, m^{2} + \frac{3}{2} \, m p + \frac{1}{6} \, p^{2} + m + \frac{1}{2} \, p + \frac{4}{3} & (\text{otherwise}),\\
    \end{cases}\\
    & \dim \calA_\Theta^\odd(\mathrm{Ker}\,\ve) \\
    &= \begin{cases}
        \frac{1}{6} \, m^{2} p^{2} + \frac{1}{2} \, m^{2} p + \frac{2}{3} \, m^{2} + m p + \frac{1}{6} \, p^{2} - \frac{1}{2} \, m - \frac{1}{2}  & (p \not \equiv 0 \bmod 3, m \not \equiv 0 \bmod 3),\\
        \frac{1}{6} \, m^{2} p^{2} + \frac{1}{2} \, m^{2} p + \frac{2}{3} \, m^{2} + m p + \frac{1}{6} \, p^{2} - \frac{1}{2} \, m - \frac{1}{6} & (\text{otherwise}).\\
    \end{cases}
    \end{split}
\end{equation}

$(2)$ When $\pi = \Z_m \times D_{4p}^{\ast}$ where $m \geq 1$, $p>0$ odd, and $(m, 2p)=1$,	
	\begin{equation}
	\begin{split}
	&\dim \calA_\Theta^\odd(\C\pi) \\
	&= \begin{cases}
        \frac{1}{6} \, m^{2} p^{2} + \frac{1}{2} \, m^{2} p + \frac{2}{3} \, m^{2} + \frac{3}{2} \, m p + \frac{1}{6} \, p^{2} + \frac{1}{2} \, m + \frac{1}{2}  & (p \not \equiv 0 \bmod 3, m \not \equiv 0 \bmod 3),\\
        \frac{1}{6} \, m^{2} p^{2} + \frac{1}{2} \, m^{2} p + \frac{2}{3} \, m^{2} + \frac{3}{2} \, m p + \frac{1}{6} \, p^{2} + \frac{1}{2} \, m + \frac{5}{6} & (\text{otherwise}),\\
    \end{cases}\\
    & \dim \calA_\Theta^\odd(\mathrm{Ker}\,\ve) \\
    &= \begin{cases}
         \frac{1}{6} \, m^{2} p^{2} + \frac{1}{2} \, m^{2} p + \frac{2}{3} \, m^{2} + m p + \frac{1}{6} \, p^{2} - m - \frac{1}{2} \, p  & (p \not \equiv 0 \bmod 3, m \not \equiv 0 \bmod 3),\\
         \frac{1}{6} \, m^{2} p^{2} + \frac{1}{2} \, m^{2} p + \frac{2}{3} \, m^{2} + m p + \frac{1}{6} \, p^{2} - m - \frac{1}{2} \, p + \frac{1}{3} & (\text{otherwise}).\\
    \end{cases}
    \end{split}
\end{equation}
(The values of the dimensions for $m=1, p\leq 15$ are shown in Table~\ref{tab:val_dims_d4p}.)
\item[\rm (c)] When $\pi = \Z_m \times D_{2^{k+2}p}'$ where $m\geq 1$, $k\geq 0$, $p \geq 3$ odd, and $(m, 2p)=1$,
    \begin{equation}
    \begin{split}
    	&\dim \calA_\Theta^\odd(\C\pi)\\
    	&= \begin{cases}
        \frac{1}{6}2^{2k}m^2 p^2  + \frac{1}{2}2^{2k}m^2 p + \frac{2}{3} 2^{2k}m^2 + \frac{3}{2} \cdot 2^{k} mp   + \frac{1}{6} \, p^{2} + \frac{1}{2} 2^k m  + \frac{1}{2} &  (p \not \equiv 0 \bmod 3, m \not \equiv 0 \bmod 3),\\
        \frac{1}{6}2^{2k}m^2 p^2  + \frac{1}{2}2^{2k}m^2 p + \frac{2}{3} 2^{2k}m^2 + \frac{3}{2} \cdot 2^{k} mp   + \frac{1}{6} \, p^{2} + \frac{1}{2} 2^k m  + \frac{5}{6} & (\text{otherwise}),\\
       \end{cases}\\
        & \dim \calA_\Theta^\odd(\mathrm{Ker}\,\ve)\\
        &=\begin{cases}
      \frac{1}{6}  2^{2k} m^{2} p^{2} + \frac{1}{2}  2^{2k} m^{2} p + \frac{2}{3} 2^{2k} m^{2} + 2^{k} m p - 2^{k} m + \frac{1}{6} \, p^{2} - \frac{1}{2} \, p &  (p \not \equiv 0 \bmod 3, m \not \equiv 0 \bmod 3),\\
        \frac{1}{6}  2^{2k} m^{2} p^{2} + \frac{1}{2}  2^{2k} m^{2} p + \frac{2}{3}  2^{2k} m^{2} + 2^{k} m p - 2^{k} m + \frac{1}{6} \, p^{2} - \frac{1}{2} \, p + \frac{1}{3} & (\text{otherwise}).\\
        \end{cases}
    \end{split}
\end{equation}

\item[\rm (d)] When $\pi = \Z_m \times T^{\ast}$ where $m\geq 1$ and $(m, 6)=1$,
	\begin{equation}
	\dim \calA_\Theta^\odd(\C\pi) = \frac{19}{3} \, m^{2} + 6 \, m + \frac{8}{3},\quad  \dim \calA_\Theta^\odd(\mathrm{Ker}\,\ve)=  \frac{19}{3} \, m^{2} + \frac{5}{2} \, m + \frac{7}{6}.
\end{equation}
\item[\rm (e)] When $\pi = \Z_m \times T_{8\cdot 3^k}'$ where $m\geq 1$, $(m, 6)=1$, and $k \geq 2$,
\begin{equation}
\dim \calA_\Theta^\odd(\C\pi) = 19 \cdot 3^{2k-3}m^2 + 2 \cdot 3^k m + 3,\quad \dim \calA_\Theta^\odd(\mathrm{Ker}\,\ve)= 19 \cdot 3^{2k - 3} m^{2} + \frac{5}{6} \cdot 3^{k} m + \frac{3}{2}.
\end{equation}
(The values of the dimensions for $m=1, k\leq 9$ are shown in Table~\ref{tab:val_dims_t83k}.)
\item[\rm (f)]  When $\pi=\Z_m \times O^{\ast}$ where $m\geq 1$ and $(m,6)=1$, 
\begin{equation}
\dim \calA_\Theta^\odd(\C\pi)=\frac{34}{3} m^2 + 12m +  \frac{35}{3},\quad \dim \calA_\Theta^\odd(\mathrm{Ker}\,\ve) = \frac{34}{3} m^2 + 8m +  \frac{23}{3}.
\end{equation}
\item[\rm (g)]  When $\pi = \Z_m \times I^{\ast}$ where $m\geq 1$ and $(m,30)=1$,
\begin{equation}
\dim \calA_\Theta^\odd(\C\pi)=\frac{74}{3} \, m^{2} + 19 \, m + \frac{64}{3},\quad \dim \calA_\Theta^\odd(\mathrm{Ker}\,\ve)= \frac{74}{3}m^2 + \frac{29}{2}m + \frac{101}{6}.
\end{equation}
	\end{enumerate}

\end{mainThm}

\begin{remark}
We have double checked the result of the computation by using {\tt Sage} for many groups in the list\footnote{See the code at \url{https://gist.github.com/ht-kodani/84fa8620a8544a109bbc789acc029c58} for details.}.
It would be notable that the leading terms of $\dim \calA_\Theta^\odd(\C\pi)$ and $\dim \calA_\Theta^\odd(\mathrm{Ker}\,\ve)$ agree (see also Conjecture~\ref{conj:lower-bound}-2 below). 
\end{remark}

\begin{remark}
Case (c), where $\pi = \Z_m \times D_{2^{k+2}p}'$,  with $k = 0$ coincides with the odd-$p$ subfamily of Case (b), where $\pi = \Z_m \times D_{4p}^{\ast}$,  since $D
_{4p}' = D^{\ast}_{4p}$. Similarly, Case (e), where $\pi = \Z_m \times T_{8\cdot 3^k}'$,  with $k = 1$ coincides with case (d), where $\pi = \Z_m \times T^{\ast}$, since $T_{24}'= T^{\ast}$, whereas the formula obtained by setting $k=1$ in the formula for $T_{8\cdot 3^k}'$ is different from the corresponding formula for $T^{\ast}$. For this reason, we  treats $T_{8\cdot 3^k}'$ separately only for $k \geq  2$ in Theorem~\ref{thm:dim-formula}.
\end{remark}

\begin{example}
If $N$ is the lens space $L(n,q)$ (resp. the Poincar\'{e} homology sphere $\Sigma(2,3,5)$), then $\pi=\Z_n$ (resp. $\pi=I^*$), and the dimension of the relevant space $\calA_\Theta^\odd(\mathrm{Ker}\,\ve)$ is $p_3(n-3)$ (resp. $\frac{74}{3}+\frac{29}{2}+\frac{101}{6}=56$) by Theorem~\ref{thm:dim-formula}. By Theorem~\ref{thm:upper-bound}, we have
\[ 
   \dim\calI_\Theta(L(n,q))\leq p_3(n-3),\quad
    \dim\calI_\Theta(\Sigma(2,3,5))\leq 56. 
\]
If $N$ is the prism manifold $S^3/\pi$ with $\pi=D_{4p}^*$, $p\equiv 3$ mod 6, then the dimension of $\calA_\Theta^\odd(\mathrm{Ker}\,\ve)$ is $\frac{1}{3}p^2+p$ by Theorem~\ref{thm:dim-formula}. By Theorem~\ref{thm:upper-bound}, we have, for $p\equiv 3$ mod 6, 
\[ \dim\calI_\Theta(S^3/D_{4p}^*)\leq \frac{1}{3}p^2+p. \]
\end{example}

\subsection{Background: Finite type invariants of $\Z\pi$-homology equivalences}
Theorem~\ref{thm:dim-formula} would also be applied to finite type invariants of $\Z\pi$-homology equivalences.
We briefly review Garoufalidis and Levine's work on finite type invariants of $\Z\pi$-homology equivalence in \cite{GL}. In the following, we assume for simplicity that $\pi_2N=0$. In \cite[Lemma~2.4]{GL}, it is shown that a degree 1 map $f\colon M\to N$ can be represented by a null-homotopic framed link in $N$. When $\pi_2N=0$ the diffeomorphism equivalence class of $f$ is uniquely determined by such a framed link\footnote{In \cite[Proposition~2.5]{GL}, the indeterminacy for the case $\pi_2N\neq 0$ is determined.}. Moreover, in \cite[\S{2.2}]{GL}, a ``surgery obstruction map''
\[ \Phi\colon \calH(N)\to B(\pi) \]
to certain semi-group $B(\pi)$ was defined by assigning a matrix of $\Z\pi$-valued linking numbers (or $\pi$-equivariant linking number) of the framed link in $N$. 

Let $\calK(N)=\mathrm{Ker}\,\Phi$ and let $\calF(N)$ be the vector space over $\Q$ spanned by the set $\calK(N)$. In \cite{GL}, Garoufalidis and Levine introduced a descending filtration on $\calF(N)$ by $\Q$-subspaces: 
\[ \calF(N)=\calF_0^Y(N)\supset \calF_1^Y(N)\supset \calF_2^Y(N)\supset\cdots \]
using surgeries of $\Z\pi$-homology equivalences along ``graph claspers" or ``Y-links" (\cite{Hab,GGP}), which replaces a $\Z\pi$-homology equivalence $f\colon M\to N$ with another one $f^{G}\colon M^G\to N$ obtained by modifying $f$ in a neighborhood of an embedded graph $G$ in $M$. This filtration is an analogue of Ohtsuki's filtration of homology 3-spheres (\cite{Ohts}). {\it Finite type invariants of type $n$} are defined as linear maps $\lambda\colon \calK(N)\to A$ to a vector space $A$ such that $\lambda(\calF_{n+1}^Y(N))=0$. It is proved in \cite{GL} that $\calF^Y_0(N)/\calF^Y_1(N)=0$ and $\calF^Y_1(N)/\calF^Y_2(N)=0$. Hence $\calF(N)=\calF_0^Y(N)=\calF_1^Y(N)=\calF_2^Y(N)$.

According to \cite[Theorem~2]{GL}, the space $\calF_n^Y(N)/\calF_{n+1}^Y(N)$ for $n$ even is generated by $\pi$-decorated graphs. 
A {\it $\pi$-decorated graph} is a pair $(\Gamma,\alpha)$ of an abstract, vertex-oriented, edge-oriented trivalent graph $\Gamma$ and a map $\alpha\colon \mathrm{Edges}(\Gamma)\to \Q\pi$. Let $\calA_n(\pi)$ be the vector space over $\Q$ spanned by $(\Gamma,\alpha)$ with $n$ vertices quotiented by the relations AS, IHX, Orientation Reversal, Linearity, and Holonomy (Figure~\ref{fig:relations}). %
\begin{figure}[h]%
\centering
\includegraphics[height=35mm]{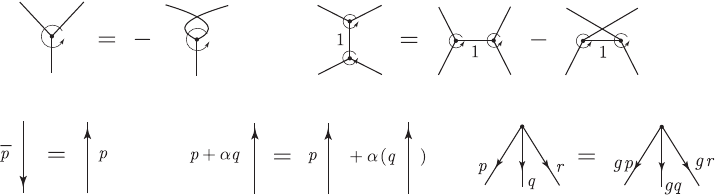}
\caption{The relations AS, IHX, Orientation reversal, Linearity, and Holonomy. Here $\bar{p}$ is the involution of $\Q\pi$ given by $\bar{g}=g^{-1}$, $p,q,r\in \Q\pi$, $\alpha\in \Q$, and $g\in \pi$.}\label{fig:relations}
\end{figure}%

When $\mathrm{Im}\,\alpha\subset \pi$, a $\pi$-decorated graph $(\Gamma,\alpha)$ determines uniquely a homotopy class $\bar{\alpha}$ of a map $\Gamma\to N$. Surgery of $\Z\pi$-homology equivalences on Y-links associated to graphs embedded in $N$ defines a linear map
\[ \psi_n\colon\calA_n(\pi)\to \calF_n^Y(N)/\calF_{n+1}^Y(N). \]
\begin{theorem}[{Garoufalidis--Levine \cite[Theorem~2]{GL}}]\label{thm:GL}
For all integers $n\geq 1$, $\psi_n$ is surjective.
\end{theorem}
In \cite{GL}, this was proved for $\Z[\frac{1}{2}]$-coefficients. When $N=S^3$, it is known that $\psi_n$ is an isomorphism over $\Q$ for all $n$ (\cite{LMO,Le,BGRT,BGRT2,BGRT3,KuTh}, see also \cite[Remark~4.2]{GL}). Moreover, when $N=S^3$, the inverse of $\psi_n$ can be given by Kontsevich's configuration space integrals (\cite{KuTh}, see also \cite{Les2,Les3}). With this in mind, it would be natural to expect a similar result for general $N$. In particular, the most relevant application of our result to the study of $\Z\pi$-homology equivalences would be the second and third items of the following conjecture. 

\begin{conj}\label{conj:lower-bound}
Let $N$ be a spherical 3-manifold with nontrivial fundamental group $\pi$.
\begin{enumerate}
\item $\widehat{Z}_\Theta^\odd\colon \calF(N)\to \calA_\Theta^\odd(\mathrm{Ker}\,\ve)$ descends to a linear map $\overline{Z}_\Theta^\odd\colon\calF(N)/\calF_3^Y(N)\to \calA_\Theta^\odd(\mathrm{Ker}\,\ve)$. In other words, $\widehat{Z}_\Theta^\odd$ is a finite type invariant of $\Z\pi$-homology equivalences of type 2.
\item The composition 
\[ \overline{Z}_\Theta^\odd\circ \psi_2\colon \calA_\Theta^\odd(\C\pi)\to \calF_2^Y(N)/\calF_3^Y(N)\otimes\C=\calF(N)/\calF_3^Y(N)\otimes\C\to \calA_\Theta^\odd(\mathrm{Ker}\,\ve) \]
agrees with the projection (see Proposition~\ref{prop:graph-inv} below for the reason that $\psi_2$ can be defined on $\calA_\Theta^\odd(\C\pi)$). Hence we have
\[ \dim\calA_\Theta^\odd(\mathrm{Ker}\,\ve) \leq \dim \calF_2^Y(N)/\calF_3^Y(N)\otimes\C \leq \dim \calA_\Theta^\odd(\C\pi).\]
\item The canonical map $\Hom_\C(\calA_\Theta^\odd(\mathrm{Ker}\,\ve),\C)\to \calI_\Theta(N)$ of Theorem~\ref{thm:upper-bound} is an isomorphism.
\end{enumerate}
\end{conj}

A resolution of the second or the third item of Conjecture~\ref{conj:lower-bound} implies that the whole of $\calH(N)$ cannot be detected by surgery theory, in the sense of \cite[\S{2.3}]{GL}. Namely, according to \cite[11.3A Theorem]{FQ} and \cite{KiTa}, there is an exact sequence
\[ H^1(N;\Z/2)\to L_0^h(\pi)/L_0^h(1)\to \calH_0^{\mathrm{top}}(N)\to H^2(N;\Z/2)\to L_3^h(\pi), \]
where $\calH_0^{\mathrm{top}}(N)$ is the topological version of the structure set of $\Z\pi$-homology bordism classes of $\Z\pi$-homology equivalences $f\colon M\to N$, and $L_i^h(\pi)$ are the Wall surgery obstruction groups (\cite{Wal}). The group $L_0^h(\pi)/L_0^h(1)$ is closely related to the codomain $B(\pi)$ of $\Phi$ (see \cite[\S{2.1} and \S{2.3}]{GL}). Then Conjecture~\ref{conj:lower-bound} implies that $\calH(N)$ has more information than those detected by $L_0^h(\pi)/L_0^h(1)$ or $H^2(N;\Z/2)$.

There are also analogues of Conjecture~\ref{conj:lower-bound} for $\calF_{2n}^Y(N)/\calF_{2n+1}^Y(N)$  and for more general trivalent graphs for $n>2$. At present, we work only on the 2-loop part for simplicity.

\begin{prop}\label{prop:graph-inv}
For a finite group $\pi$, let $\hat{\pi}$ denote the set of conjugacy classes of $\pi$, and let $\C\hat{\pi}$ denote the vector space over $\C$ spanned by the set $\hat{\pi}$. We have the following.
\begin{enumerate}
\item $\calA_\Theta^\odd(\C\pi)\cong\calA_2(\pi)\otimes\C$. 
\item $\calA_\Theta^\odd(\C\pi)=\calA_\Theta^\odd(\mathrm{Ker}\,\ve)\oplus (\C\hat\pi)_{\Z_2}$ as a vector space over $\C$, where the $\Z_2$-action on $\C\hat{\pi}$ is induced by the inversion $g\mapsto g^{-1}$ on $\pi$.
\end{enumerate}
\end{prop}
\begin{proof}
For $a,b,c\in \C\pi$, let $\Theta(a,b,c)$ denote the theta-graph, i.e. the graph with two vertices connected by three parallel edges, with three parallel arrows colored by $a,b,c$ respectively.
It can be shown by using the graph relations that the space $\calA_2(\pi)\otimes\C$ is spanned by $\Theta(a,b,c)$ for $a,b,c\in \pi$. Then an isomorphism $\calA_2(\pi)\otimes\C\to \calA_\Theta^\odd(\C\pi)$ of the assertion 1 is defined by the correspondence 
$\Theta(a,b,c) \leftrightarrow [a\cdot b\cdot c]$.
This is well-defined since the graph relations in $\calA_2(\pi)\otimes\C$ correspond to the invariance condition in $\calA_\Theta^\odd(\C\pi)$. The assertion 2 has been proved in \cite[Proposition~4.3]{OW}.
\end{proof}

In principle, it is possible to compute the dimension of $\calA_2(\pi)\otimes \C$ for each fixed finite group $\pi$ by using the presentation of $\calA_2(\pi)\otimes \C$ by the set $\{\Theta(1,b,c)\mid b,c\in \pi\}$ of generators and the relations induced from those of Figure~\ref{fig:relations}. However, it requires to enumerate all the elements in $\pi$ and to solve a system of linear equations in $|\pi|^2$ variables, and it would be intractable in such an approach to determine the dimensions for groups with large orders $|\pi|$ and for general infinite sequences of groups. We instead use character theory of finite groups to make the general computation possible.

\subsection*{Acknowledgements}
The authors thank Tatsuro Shimizu for fruitful discussions and important comments throughout this work. The authors thank Minoru Hirose for giving them useful information and comments about SageMath, and, in particular, improving the initial trial SageMath code for the dimension computation. The authors thank Hiroyuki Ochiai for several useful discussions on integer sequences. The authors thank Yuta Nozaki and the referees for valuable comments on earlier versions of this paper. 
HK was partially supported by JSPS Grant-in-Aid for Scientific Research 25K06954. TW was partially supported by JSPS Grant-in-Aid for Scientific Research 21K03225, 20K03594, and by RIMS, Kyoto University. 

\section{Preliminaries}

\subsection{Notations and facts about representations of finite groups}\label{ss:rep}

In this paper, we consider linear representation $\rho\colon \C\pi\to \End_\C(V)$, where $V$ is a finite-dimensional vector space over $\C$, for a finite group $\pi$. For such a $\pi$-module $V$, the subspace of {\it invariants} $V^\pi$ and the space of {\it coinvariants} $V_\pi$ are defined as follows.
\[\begin{split}
& V^\pi:=\{v\in V\mid gv=v\,\,(\forall g\in \pi)\},\qquad
 V_\pi:=V/{\Span_\C\{gv-v\mid g\in\pi,\,v\in V\}}.
\end{split} \]
The map $V_\pi\to V^\pi$ induced from the map $V\to V^\pi$; $v\mapsto \frac{1}{|\pi|}\sum_{g\in\pi}gv$ is an isomorphism, whose inverse is given by the composition $V^\pi\stackrel{\subset}{\to} V\twoheadrightarrow V_\pi$. We will often identify $V_\pi$ and $V^\pi$ by this isomorphism. The dimension of $V^\pi$ can be computed by the following formula (\cite[(2.9)]{FH}):
\begin{equation}\label{eq:dim_invariant_part}
    \dim V^\pi=\frac{1}{|\pi|}\sum_{g\in \pi}\chi_V(g),
\end{equation}
where $\chi_V$ is the character of $V$. 

For a $\pi$-module $V$, the dual $V^*=\Hom_\C(V,\C)$ is also a $\pi$-module by $g\varphi(v)=\varphi(g^{-1}v)=((g^{-1})^*\varphi)(v)$. The $\C$-linear map $(\,\cdot\,)^*\colon \End(V)\to \End(V^*)$ assigning to an endomorphism $\phi$ its dual $\phi^*$ is defined. We identify $\End(V)$ with $V\boxtimes V^*$ by the following isomorphism of $\pi\times \pi$-modules
\begin{equation}\label{eq:lambda}
\lambda_V\colon V\boxtimes V^*\to \End(V);\quad (v,a)\mapsto (u\mapsto a(u)v). 
\end{equation}
Let $\{A_i\}_{i=1}^r$ be the collection of all distinct irreducible $\pi$-modules. The $\pi\times \pi$-module $\C\pi$ and its submodule $\mathrm{Ker}\,\ve$ are decomposed into irreducible $\pi\times\pi$-modules as
\begin{equation} \label{eq:AW-isom}
  \C\pi=\bigoplus_{i=1}^r\End(A_i)\cong \bigoplus_{i=1}^r(A_i\boxtimes A_i^*),\qquad 
  \mathrm{Ker}\,\ve\cong\bigoplus_{i=2}^r(A_i\boxtimes A_i^*),
\end{equation}
where $A_1$ is the trivial irreducible $\pi$-module (\cite[Proposition~3.29]{FH}).

\subsection{Homology of configuration spaces of two points}

Suppose that $N$ is a spherical 3-manifold with nontrivial fundamental group $\pi$, which is finite. 
Let $A$ be a nontrivial irreducible $\C\pi$-module. 
The homology $H_*(N;A)$ for the local system $A$ is defined by
\[ H_*(N;A):=H(S_*(\widetilde{N})\otimes_{\C\pi}A), \]
where $\widetilde{N}$ is the universal cover of $N$, $S_*(\cdot)$ is the $\C$-complex of singular chains. We collect some facts on the homology of the configuration space of two points on a spherical 3-manifold from \cite{OW} with proofs.

\begin{lemma}[\cite{OW}]\label{lem:M-acyclic}
Let $N$ be as above and let $\pi=\pi_1N$. Let $A$ be a nontrivial irreducible $\C\pi$-module. 
\begin{enumerate}
\item We have $H_*(N;A)=0$ and $H_*(N;A^*)=0$. 
\item For a $\Z\pi$-homology equivalence $f\colon M\to N$, we have $H_*(M;A)=0$ and $H_*(M;A^*)=0$,
where we also consider $A$ as a $\C\pi_1M$-module through $f_*\colon\pi_1M\to \pi$. 
\end{enumerate}
\end{lemma}
\begin{proof}
For 1, since $\pi$ is finite, $\C\pi$ is semisimple in the sense of \cite[\S{I.4}]{CE} by Maschke's theorem and we have $H^1(\pi;A)=0$ (Theorem~VI.16.6 and Lemma~VI.16.7 of \cite{HS}). By the universal coefficient theorem, which is valid if the ring is hereditary (e.g., $\C\pi$ for $\pi$ finite), the sequence
\[ 0\to \mathrm{Ext}_{\C\pi}^1(H_{i-1}(C),A)\to H^i(\mathrm{Hom}_{\C\pi}(C,A))\to \mathrm{Hom}_{\C\pi}(H_i(C),A)\to 0 \]
is exact for $C=S_*(S^3;\C)$ (as a $\C\pi$-module) and any $\C\pi$-module $A$ (e.g., \cite[Theorem~VI.3.3]{CE}). Hence we have
$H^3(N;A)\cong \mathrm{Hom}_{\C\pi}(\C,A)\cong A^{\pi}$,
$H^2(N;A)=0$, 
$H^1(N;A)\cong\mathrm{Ext}_{\C\pi}^1(\C,A)=H^1(\pi;A)=0$,
$H^0(N;A)\cong\mathrm{Hom}_{\C\pi}(\C,A)\cong A^{\pi}$.
Since $A$ is a nontrivial irreducible $\C\pi$-module, we have $A^\pi=0$. 
Then by Poincar\'{e} duality, we also have $H_*(N;A)=0$. The dual $\pi$-module $A^*$ is nontrivial irreducible too, and we also have $H_*(N;A^*)=0$.

For 2, let $\widetilde{M}$ be the $\pi$-covering obtained by pulling back $\widetilde{N}$ by $f$. The argument in the previous paragraph works also for $C=S_*(\widetilde{M};\C)$, whose homology is isomorphic to that of $S^3$ since $f$ is a $\Z\pi$-homology equivalence. 
\end{proof}

\begin{lemma}\label{lem:M=QHS}
Let $N$ be as above and let $\pi=\pi_1N$. Then a $\Z\pi$-homology equivalence $f\colon M\to N$ induces isomorphisms $H_i(M;\Z)\cong H_i(N;\Z)$ for $0\leq i\leq 3$. Hence $M$ is a rational homology 3-sphere.
\end{lemma}
\begin{proof}
Let $\widetilde{M}\to M$ be the $\pi$-covering that is the pullback of the universal covering $\widetilde{N}\to N$. Since this is a principal $\pi$-bundle, it is also given as a pullback of the universal covering $E\pi=\widetilde{K(\pi,1)}\to B\pi=K(\pi,1)$ by the map $\overline{f}\colon M\to K(\pi,1)$ induced by $f$ (see Appendix~\ref{s:3-mfd-over-Kpi1}).
In particular, $\widetilde{M}$ is homotopy equivalent to the homotopy fiber of $\overline{f}$ over the basepoint, and $\widetilde{M}\to M\stackrel{\overline{f}}{\to} K(\pi,1)$ is equivalent to a fibration. Since its homotopy fiber $\simeq\widetilde{M}$ is an integral homology 3-sphere, it is homologically 2-connected, and the Serre exact sequence is as follows:
$H_3(\widetilde{M})\to H_3(M)\to H_3(\pi)\to H_2(\widetilde{M})\to H_2(M)\to H_2(\pi)
    \to  H_1(\widetilde{M})\to H_1(M)\to H_1(\pi)\to 0$.
Since $H_2(\widetilde{M})=H_1(\widetilde{M})=0$, the maps $H_i(M)\to H_i(\pi)$ are isomorphisms for $i=1,2$. By considering a similar exact sequence for $\overline{\mathrm{id}}\colon N\to K(\pi,1)$, we have that the maps $H_i(N)\to H_i(\pi)$ are isomorphisms for $i=1,2$. This proves the assertion for $i=1,2$. 
Moreover, both $H_3(M)$ and $H_3(N)$ are isomorphic to $\Z$, and the map between them is an isomorphism since $f$ is degree one.
\end{proof}

Let $M\to N$ be a $\Z\pi$-homology equivalence as above. Let $\Delta_M$ be the diagonal of $M\times M$. The configuration space of two points of $M$ is
$\Conf_2(M)=M\times M-\Delta_M$, and we consider the real Fulton--MacPherson compactification of $\Conf_2(M)$:
\[ \bConf_2(M)=B\ell_{\Delta_M}(M\times M), \]
which has the same homotopy type as $\Conf_2(M)$. Roughly, this is obtained by replacing the diagonal $\Delta_M$ in $M\times M$ with its normal sphere bundle $SN\Delta_M$.
A $\pi\times \pi$-module $W$ gives a local coefficient system on $M\times M$. 
We make the following assumption.
\begin{defini}\label{def:2-acyclic}
We say that a finite dimensional $\pi\times \pi$-module $W$ is {\it acyclic} with respect to $M\times M$ if it satisfies 
\[ H_*(M\times M;W)=0. \]
\end{defini}
For a non-trivial irreducible $\pi$-module $A$, let $A\boxtimes A^*$ denote the pullback of the local coefficient system $A\boxtimes_\C A^*$ on $M\times M$ (as a $\pi\times\pi$-module) to $\bConf_2(M)$. 
\begin{lemma}[\cite{OW}]\label{lem:W}
The following $\pi\times\pi$-modules $W$ are acyclic with respect to $M\times M$.
\begin{enumerate}
\item $W=\bigoplus_i (A_i\boxtimes A_i^*)^{\oplus m_i}$ for the collection $\{A_i\}_{i=1}^r$ of all distinct nontrivial irreducible $\pi$-modules $A_i$.
\item $W=\mathrm{Ker}\,\ve$, where $\ve\colon \C\pi\to \C$ is the augmentation map.
\end{enumerate}
\end{lemma} 
\begin{proof}
That 1 is acyclic follows from Lemma~\ref{lem:M-acyclic}, the K\"{u}nneth formula for $\C\pi$-modules, and $H_*(X;\bigoplus_i V_i)=\bigoplus_i H_*(X;V_i)$ for a space $X$ and a direct sum $\bigoplus_i V_i$ of $\pi_1X$-modules. 
The acyclicity for 2 follows from the $\pi\times\pi$-module decomposition (\ref{eq:AW-isom}) of $\mathrm{Ker}\,\ve$.
\end{proof}

\begin{lemma}\label{lem:H(diagonal)}
For a finite dimensional $\pi\times\pi$-module $W$, let $W_\Delta$ be $W$ with the $\pi$-module structure induced by the diagonal map $\pi\to \pi\times \pi$; $g\mapsto g\times g$. Then we have $H_*(\Delta_M;W_\Delta)\cong H_*(M;\C)\otimes_\C(W_\Delta)^\pi$
\end{lemma}
\begin{proof}
If $W_\Delta=(W_\Delta)^\pi\oplus \bigoplus_j V_j$ for non-trivial irreducible $\pi$-modules $V_j$, we have $H_*(\Delta_M;W_\Delta)=H_*(\Delta_M;(W_\Delta)^\pi)\oplus \bigoplus_j H_*(\Delta_M;V_j)=H_*(\Delta_M;(W_\Delta)^\pi)$ by Lemma~\ref{lem:M-acyclic}.
\end{proof}

A framing $\tau_0$ on $M$ gives a trivialization of the normal bundle of $\Delta_M$ in $M\times M$. This induces the identification $\partial\bConf_2(M)=SN\Delta_{M}\cong \Delta_M\times S^2= M\times S^2$ of the normal sphere bundle of $\Delta_M\subset M\times M$. We denote by $ST(M)$ the $S^2$-bundle over $M$ of unit tangent vectors in $TM$, which is canonically identified with $SN\Delta_M$. For a submanifold $\sigma$ of $M$, we denote by $ST(\sigma)$ the preimage of $\sigma$ under the projection $ST(M)\to M$.
\begin{prop}[\cite{OW}]\label{prop:H(Conf)}
Let $W$ be a $\pi\times\pi$-module that is acyclic with respect to $M\times M$. Then we have
\[ H_p(\bConf_2(M);W)\cong\left\{\begin{array}{ll}
\Span_\C\{[ST(*)]\}\otimes_\C (W_\Delta)^\pi & (p=2),\\
\Span_\C\{[ST(M)]\}\otimes_\C (W_\Delta)^\pi & (p=5),\\
0 & (\mbox{otherwise}).
\end{array}\right. \]
\end{prop}
\begin{proof}
We consider the homology exact sequence of the pair $(M^{\times 2},\Conf_2(M))$:
\[
 H_{i+1}(M^{\times 2};W)\to H_{i+1}(M^{\times 2},\Conf_2(M);W)\to H_i(\Conf_2(M);W)\to H_i(M^{\times 2};W),
\]
where $H_*(M^{\times 2};W)=0$ by the acyclicity of $W$. 
Let $N\Delta_M$ be a closed tubular neighborhood of $\Delta_M$. We have
$H_{i+1}(M^{\times 2},\Conf_2(M);W) \cong H_{i+1}(N\Delta_M,\partial N\Delta_M;W)
    \cong H_{i-2}(\Delta_M;W_\Delta)\cong H_{i-2}(M;\C)\otimes_\C(W_\Delta)^\pi$
by excision, the Thom isomorphism and Lemma~\ref{lem:H(diagonal)}. 
Here, by Lemma~\ref{lem:M=QHS}, we have that $H_{i-2}(M;\C)$ is rank 1 for $i-2=0,3$, and its generator is $*,M$, respectively.
\end{proof}

\begin{lemma}[\cite{OW}]\label{lem:H(dC)}
Let $W$ be a $\pi\times\pi$-module that is acyclic with respect to $M\times M$. Let $\gamma_{\tau_0}\colon M\to ST(M)$ be the section induced from $M\to M\times\{(1,0,0)\}\subset M\times S^2$ by the identification $ST(M)=M\times S^2$ induced by $\tau_0$. Then we have
\[ \begin{split}
	H_3(\partial \bConf_2(M);W_\Delta)=\Span_\C\{[\gamma_{\tau_0}(M)]\}\otimes_\C (W_\Delta)^\pi.
\end{split} \]
\end{lemma}
\begin{proof}
This follows from the trivialization $\partial \bConf_2(M)\cong S^2\times M$ induced by $\tau_0$, Lemma~\ref{lem:H(diagonal)}, Lemma~\ref{lem:M=QHS} and the K\"{u}nneth formula for $\C$-modules.
\end{proof}

\subsection{Twisted de Rham cohomology}

A local system $A$ on a compact smooth manifold $X$ determines a flat vector bundle $\calE\to X$ with fiber $A$ uniquely up to isomorphism. The flat connection on $\calE$ gives a twisted differential $d_A$ on $\Omega_\dR^*(X;\calE)$ (\cite[I-\S{7}]{BT}),
which we denote also by $\Omega_\dR^*(X;A)$. By the de Rham theorem, we have the isomorphism
\[ H_\dR^*(X;A)\cong H^*(X;A), \]
where the RHS is the singular cohomology with local coefficient system $A$. 

\section{The invariant $\widehat{Z}_\Theta^\odd$}

In this section, we prove Theorems~\ref{thm:Z_universal} and \ref{thm:upper-bound}.

\subsection{The invariant  $Z_\Theta^\odd$ of $\Z\pi$-homology equivalences with framings}

Let $f\colon M\to N$ be a $\Z\pi$-homology equivalence as in the previous section. We assume that $f$ is smooth. Let $W$ be a finite dimensional $\pi\times \pi$-module $W$ as in Lemma~\ref{lem:W}.
\begin{itemize}
\item Let $\calE_W\to N\times N$ be the flat $W$-bundle given by the Borel construction $\widetilde{N\times N}\times_{\pi\times\pi}W$, where $q\colon \widetilde{N\times N}\to N\times N$ is the universal covering of $N\times N$. This is equipped with the flat connection induced from the trivial connection $d$ on the trivial $W$-bundle $\widetilde{N\times N}\times W$, which is $\pi\times\pi$-equivariant. The irreducible decomposition of $W$ induces the Whitney sum decomposition of $\calE_W$ into flat subbundles of irreducible $\pi\times\pi$-modules.
\item Let $\calE_W^f\to \bConf_2(M)$ be the flat $W$-bundle induced from the pullback bundle $(f\times f)^*\calE_W\to M\times M$ by the projection $\bConf_2(M)\to M\times M$. 
\item Let $\calE_{W_\Delta}=\calE_W|_{\Delta_N}$ and let $\calE_{W_\Delta}^f=\calE_W^f|_{\partial\bConf_2(M)}$. Since the holonomy group of the restriction of the flat connection on the $\pi\times \pi$-bundle $\widetilde{N\times N}$ to the fibers over $\Delta_N$ is the diagonal subgroup $\Delta_\pi\subset \pi\times\pi=\Delta_\pi\ltimes \pi$, we have $q^{-1}(\Delta_N)=\widetilde{\Delta}_N\times\pi$, where $\widetilde{\Delta}_N$ is the universal cover of $\Delta_N$, and 
$\calE_{W_\Delta}=q^{-1}(\Delta_N)\times_{\pi\times\pi}W=(\widetilde{\Delta}_N\times\pi)\times_{\Delta_\pi\ltimes\pi}W
=\widetilde{\Delta}_N\times_{\Delta_\pi}(\pi\times_\pi W)
=\widetilde{\Delta}_N\times_{\Delta_\pi} W_\Delta$ as flat bundles.

\item Let $(\calE_{W_\Delta})^\pi$ denote the trivial subbundle of $\calE_{W_\Delta}$ corresponding to the $\pi$-invariant part $(W_\Delta)^\pi$ of $W_\Delta$. Note that this subbundle can be obtained by taking the image of the fiberwise projection induced by the $\pi$-equivariant projection $W_\Delta\to (W_\Delta)^\pi$, $w\mapsto \frac{1}{|\pi|}\sum_{g\in \pi}gw$. Also, 
\[ (\calE_{W_\Delta})^\pi=\widetilde{\Delta}_N\times_\pi(W_\Delta)^\pi=\widetilde{\Delta}_N/\pi\times (W_\Delta)^\pi=\Delta_N\times (W_\Delta)^\pi\subset \widetilde{\Delta}_N\times_\pi W_\Delta. \]
This gives the canonical trivialization of $(\calE_{W_\Delta})^\pi$. The trivial connection on $\widetilde{\Delta}_N\times(W_\Delta)^\pi$, which is compatible with that on $q^{-1}(\Delta_N)\times W$, induces the trivial connection on $(\calE_{W_\Delta})^\pi$. 
\item Let $(\calE_{W_\Delta}^f)^\pi$ denote the pullback from $(\calE_{W_\Delta})^\pi$ in $\calE_{W_\Delta}^f$.  
\end{itemize}
Let $\mathbf{1}_W\in W_\Delta$ be the canonical element corresponding to the direct sum of the identity morphisms in $\End(A_i)$ for the irreducible factors $A_i\boxtimes A_i^*$ in $W$. The element $\mathbf{1}_W\in W_\Delta$ gives rise to sections of $(\calE_{W_\Delta})^\pi$ and $(\calE_{W_\Delta}^f)^\pi$, which we also denote by $\mathbf{1}_W$. 
Let $\tau\colon TM\stackrel{\cong}{\longrightarrow} M\times \R^3$ be a framing. Then $\tau$ induces a smooth map $\phi_\tau\colon \partial\bConf_2(M)\to S^2$ by the composition $\partial\bConf_2(M)=ST(M)\cong M\times S^2\stackrel{\text{proj}}{\to} S^2$. 
The following is an analogue of \cite[p.107]{Kon} and \cite[Lemma~1.2]{BC} for our setting.
\begin{lemma}[Propagator]\label{lem:propagator} Let $W$ be a $\pi\times\pi$-module as in Lemma~\ref{lem:W}. Let $d_W^f$ denote the twisted differential on $\Omega_\dR^*(\bConf_2(M);\calE_W^f)$ induced from the flat connection on $\calE_W^f$ pulled back from $\calE_W$. 
\begin{enumerate}
\item The $d_W^f$-closed 2-form $\phi_\tau^*\mathrm{Vol}_{S^2}\, \mathbf{1}_W\in \Omega_\dR^2(\partial\bConf_2(M); \calE_{W_\Delta}^f)$, where $\mathrm{Vol}_{S^2}$ is the $SO(3)$-invariant unit volume form $(1/4\pi)(x_1\,dx_2\wedge dx_3-x_2\,dx_1\wedge dx_3+x_3\,dx_1\wedge dx_2)$ on $S^2$, can be extended to a $d_W^f$-closed form 
\[ \omega\in \Omega_\dR^2(\bConf_2(M); \calE_W^f). \]

\item For a fixed framing $\tau$, the extension $\omega$ is unique in the sense that for two such extensions $\omega$ and $\omega'{}$, there is a 1-form $\eta\in\Omega_\dR^1(\bConf_2(M);\calE_W^f)$ such that
$\omega'{} - \omega = d_W^f \eta$ and $\eta=0$ on a neighborhood of $\partial\bConf_2(M)$.
\end{enumerate}
We call such an $\omega$ a \emph{$\calE_W^f$-propagator} for $\tau$.
\end{lemma}
\begin{proof}
The assertion 1 follows by examining the long exact sequence
\[ 
  H^2(\bC;W)\to H^2(\partial\bC;W_\Delta)\to H^3(\bC,\partial \bC;W),
\]
where we abbreviate as $\bC=\bConf_2(M)$. We have $H^p(\bC,\partial\bC;W)\cong H_{6-p}(\bC;W)=0$ for $p=2,3$ by Poincar\'{e}--Lefschetz duality and by Proposition~\ref{prop:H(Conf)}. This shows that the restriction map $H^2(\bC;W)\to H^2(\partial\bC;W_\Delta)$ is surjective. Hence there is a $d_W^f$-closed form $\omega_0\in \Omega_\dR^2(\bC;\calE_W^f)$ such that $[\omega_0|_{\partial \bC}]=[\phi_\tau^*\mathrm{Vol}_{S^2}\,\mathbf{1}_W]$ in $H^2(\partial\bC;W_\Delta)$. By Lemma~\ref{lem:collar} below, there exists a $d_W^f$-closed 2-form $\omega\in \Omega_\dR^2(\bC;\calE_W^f)$ such that $\omega|_{\partial \bC}=\phi_\tau^*\mathrm{Vol}_{S^2}\,\mathbf{1}_W$. 

The assertion 2 follows since the difference $\omega'-\omega$ vanishes on $\partial\bC$ and represents 0 in the twisted de Rham cohomology $H^2(\bC,\partial\bC;W)=0$.
\end{proof}

\begin{lemma}\label{lem:collar}
Let $X$ be a compact smooth manifold with corners and let $\calE\to X$ be a flat vector bundle. Let $d_A$ denote the twisted differential on $\Omega_\dR^*(X;\calE)$. Let $\alpha_0\in\Omega_\dR^k(X;\calE)$ and $\beta\in\Omega_\dR^k(\partial X;\calE|_{\partial X})$ be $d_A$-closed forms such that $[\alpha_0|_{\partial X}]=[\beta]$ in $H^k(\partial X;\calE|_{\partial X})$. Then there exists a $d_A$-closed form $\alpha\in \Omega_\dR^k(X;\calE)$ such that $\alpha|_{\partial X}=\beta$ and $[\alpha]=[\alpha_0]$ in $H^k(X;\calE)$.
\end{lemma}
\begin{proof}
The $d_A$-closed form $\beta$ can be extended to a $d_A$-closed form $\beta'$ on a collar neighborhood $\partial X\times [0,\ve]$ by the pullback from $\partial X=\partial X\times\{0\}$. It follows from the assumption $[\alpha_0|_{\partial X}]=[\beta]$ that $\beta'=\alpha_0|_{\partial X\times[0,\ve]}+d_A\mu$ for some $\mu\in\Omega_\dR^{k-1}(\partial X\times[0,\ve];\calE|_{\partial X\times[0,\ve]})$. We take a smooth function $\chi\colon X\to [0,1]$ such that $\chi=1$ on a neighborhood of $\partial X$ in $X$ and its support is included in $\partial X\times [0,\frac{\ve}{2}]$. Then the form $\chi\mu$ can be considered as defined over $X$, and the $d_A$-closed form
\[ \alpha=\alpha_0+d_A(\chi\mu) \]
in $\Omega_\dR^k(X;\calE)$ satisfies the desired properties.
\end{proof}

\begin{defini}
Let $W$ be a $\pi\times\pi$-module as in Lemma~\ref{lem:W}. We take $\calE_W^f$-propagators $\omega_1,\omega_2,\omega_3$ for $\tau$ and define 
\[ Z_\Theta^\odd(\omega_1,\omega_2,\omega_3)
=\frac{1}{6}\int_{\bConf_2(M)}\pr_\Theta(\omega_1\wedge \omega_2\wedge \omega_3)\in \calS_\Theta^\odd(W), \]
where $\pr_\Theta\colon W^{\otimes 3}\to \calS_\Theta^\odd(W)=(\mathrm{Sym}^3\, W)_{\pi\times\pi}$ is the projection, and the integration is defined with respect to the canonical trivialization of $(\mathrm{Sym}^3\,\calE_W^f)^{\pi\times\pi}=(\mathrm{Sym}^3\,\calE_W^f)_{\pi\times\pi}$.
\end{defini}

\begin{remark}\label{rem:propagator-BC}
\begin{enumerate}
\item In \cite[Lemma~1.2]{BC}, Bott and Cattaneo work with an orthogonal acyclic local system and construct an antisymmetric propagator under more restrictive hypothesis on representations, but could be applied to more general 3-manifolds. 
\item It is observed in \cite{CS} that when $M$, $W$ are such that $H^2(\Delta_M;W_\Delta)\neq 0$, a correction term $\xi$ from $H^2(\Delta_M;W_\Delta)$ should be added to $\phi_\tau^*\mathrm{Vol}_{S^2}\,\mathbf{1}_W$ to obtain a propagator, filling a gap in \cite{BC}. Such a term leads to the dumbbell correction term. Examples of $M,W$ with $H^2(\Delta_M;W_\Delta)\neq 0$ can be found in \cite[Proposition~3.5.5]{KL}.
In \cite{BC} and \cite{CS}, $H^3(\bC,\partial\bC;W)$ is identified with $H^3(M^2,\Delta_M;W)$ by excision, and with $H^2(\Delta_M;W_\Delta)$ by using $H^*(M^2;W)=0$. This identification was used in \cite{CS} to characterize the correction term $\xi$. 
\end{enumerate}
\end{remark}

The following theorem is a reinterpretation of a result due to Kontsevich \cite{Kon}, Bott--Cattaneo \cite{BC} (see also Cattaneo--Shimizu \cite{CS}) in terms of our setting.
\begin{theorem}\label{thm:inv-framed}
$Z_\Theta^\odd$ is an invariant of $(f,W,\tau)$. Namely, $Z_\Theta^\odd(\omega_1,\omega_2,\omega_3)$ does not depend on the choices of $\calE_W^f$-propagators $\omega_1,\omega_2,\omega_3$ for $\tau$. We will also denote $Z_\Theta^\odd(\omega_1,\omega_2,\omega_3)$ by $Z_\Theta^\odd(f,W,\tau)$.
\end{theorem}   
\begin{proof}
Let $\omega_\ell$, $\omega_\ell'$ ($\ell=1,2,3$) be $\calE_W^f$-propagators for $\tau$. By Lemma~\ref{lem:propagator}-2, there is a 1-form $\eta\in \Omega_\dR^1(\bConf_2(M);\calE_W^f)$ such that $\omega_\ell'=\omega_\ell+d_W^f\eta$. Let $\rho\colon I\times \partial\bConf_2(M)\to [0,1]$ be a smooth function such that $\rho=0$ on a neighborhood of $\{0\}\times \bConf_2(M)$ and $\rho=1$ on a neighborhood of $\{1\}\times \bConf_2(M)$. Let $\pr_2\colon I\times\bConf_2(M)\to \bConf_2(M)$ be the projection, and denote by $\pr_2^*\calE_W^f$ the pullback of the flat bundle $\calE_W^f$ on $\bConf_2(M)$ by $\pr_2$. We also use $d_W^f$ to denote the induced twisted differential on $\Omega_\dR^*(I\times\bConf_2(M);\pr_2^*\calE_W^f)$. 
Let $\widetilde{\omega}_\ell$ be a $d_W^f$-closed form in $\Omega_\dR^2(I\times \bConf_2(M);\pr_2^*\calE_W^f)$ defined by
\[ \widetilde{\omega}_\ell=\pr_2^*\,\omega_\ell+d_W^f(\rho\eta).
\]
This satisfies $\widetilde{\omega}_\ell|_{\{0\}\times \bConf_2(M)}=\omega_\ell$, $\widetilde{\omega}_\ell|_{\{1\}\times \bConf_2(M)}=\omega_\ell'$, and
\[ \widetilde{\omega}_\ell|_{I\times \partial\bConf_2(M)}=\widetilde{\phi}_\tau^*\mathrm{Vol}_{S^2}\, \mathbf{1}_W, \]
where $\widetilde{\phi}_\tau=\phi_\tau\circ\pr_2\colon I\times\partial\bConf_2(M)\to S^2$. 
We claim that the following diagram commutes.
\[ \xymatrix{
    \Omega_\dR^6(I\times\bConf_2(M);(\pr_2^*\calE_W^f)^{\otimes 3}) \ar[r]^-{\pr_\Theta} \ar[d]_-{d_W^f} 
    & \Omega_\dR^6(I\times\bConf_2(M);(\Sym^3(\pr_2^*\calE_W^f))_{\pi\times\pi}) \ar[d]^-{d} \\
    \Omega_\dR^7(I\times\bConf_2(M);(\pr_2^*\calE_W^f)^{\otimes 3}) \ar[r]^-{\pr_\Theta}
    & \Omega_\dR^7(I\times\bConf_2(M);(\Sym^3(\pr_2^*\calE_W^f))_{\pi\times\pi})
}
\]
Indeed, the component of the left vertical twisted differential $d_W^f$ of the trivial subbundle factor $(\Sym^3(\pr_2^*\calE_W^f))^{\pi\times\pi}$ is $d\colon \Omega_\dR^6(I\times\bConf_2(M);(\Sym^3(\pr_2^*\calE_W^f))^{\pi\times\pi})\to \Omega_\dR^7(I\times\bConf_2(M);(\Sym^3(\pr_2^*\calE_W^f))^{\pi\times\pi})$. By identifying the $\pi\times\pi$-equivariant projection $(\pr_2^*\calE_W^f)^{\otimes 3}\to (\Sym^3(\pr_2^*\calE_W^f))^{\pi\times\pi}$ with $\pr_\Theta\colon (\pr_2^*\calE_W^f)^{\otimes 3}\to (\Sym^3(\pr_2^*\calE_W^f))_{\pi\times\pi}$, we get the commutativity of the above diagram.

By this and by the Stokes theorem, we have
\[ \begin{split}
0&=\int_{I\times\bConf_2(M)}\pr_\Theta\, d_W^f(\widetilde\omega_1\wedge\widetilde\omega_2\wedge\widetilde\omega_3)
=\int_{I\times\bConf_2(M)}d\pr_\Theta(\widetilde\omega_1\wedge\widetilde\omega_2\wedge\widetilde\omega_3)\\
&=\int_{\bConf_2(M)}\pr_\Theta(\omega_1'\wedge\omega_2'\wedge\omega_3')
-\int_{\bConf_2(M)}\pr_\Theta(\omega_1\wedge\omega_2\wedge\omega_3)
\pm\int_{I\times\partial\bConf_2(M)}\pr_\Theta(\widetilde\omega_1\wedge\widetilde\omega_2\wedge\widetilde\omega_3)\\
&=\int_{\bConf_2(M)}\pr_\Theta(\omega_1'\wedge\omega_2'\wedge\omega_3')
-\int_{\bConf_2(M)}\pr_\Theta(\omega_1\wedge\omega_2\wedge\omega_3)
\pm\int_{I\times\partial\bConf_2(M)}\widetilde{\phi}_\tau^*(\mathrm{Vol}_{S^2})^{\wedge 3}\,\pr_\Theta(\mathbf{1}_W^{\otimes 3})\\
&=\int_{\bConf_2(M)}\pr_\Theta(\omega_1'\wedge\omega_2'\wedge\omega_3')
-\int_{\bConf_2(M)}\pr_\Theta(\omega_1\wedge\omega_2\wedge\omega_3).
\end{split} \]
This completes the proof.
\end{proof}

\begin{cor}\label{cor:inv-homotopy}
Let $f\colon M\to N$ and $f'\colon M\to N$ be two smooth $\Z\pi$-homology equivalences that are homotopic. Let $\tau\colon TM\to M\times \R^3$ be a framing. Then we have
\[ Z^\odd_\Theta(f,W,\tau)=Z^\odd_\Theta(f',W,\tau).\]
\end{cor}
\begin{proof}
Since $f$ and $f'$ are homotopic, there is a smooth homotopy $H\colon I\times M\to N$ between them. We define $H'\colon I\times M\times M\to N\times N$ by $H'(x,y,t)=(H(x,t),H(y,t))$. Let $\calE_W^H$ be the flat $W$-bundle over $I\times \bConf_2(M)$ induced from $H'^*\calE_W$. The pullback $\calE_{W_\Delta}^H\to I\times \partial\bConf_2(M)$ of $\calE_{W_\Delta}$, the invariant part $(\calE_{W_\Delta}^H)^\pi$, and its section $\mathbf{1}_W$ can also be considered. Let $d_W^H$ denote the twisted differential on $\Omega_\dR^*(I\times\bConf_2(M);\calE_W^H)$. Since the projection $I\times\bConf_2(M)\to \bConf_2(M)$ is a homotopy equivalence, the $d_W^H$-closed 2-form $\widetilde{\phi}_\tau^*\mathrm{Vol}_{S^2}\,\mathbf{1}_W$ in $\Omega^2_\dR(I\times \partial\bConf_2(M);\calE_{W_\Delta}^H)$ can be extended to a $d_W^H$-closed 2-form $\widetilde{\omega}$ by Lemma~\ref{lem:propagator}. By the same argument as in the proof of Theorem~\ref{thm:inv-framed} and by Theorem~\ref{thm:inv-framed}, we have $Z^\odd_\Theta(f,W,\tau)=Z^\odd_\Theta(f',W,\tau)$.
\end{proof}

\begin{cor}\label{cor:inv-framed}
Let $f\colon M\to N$ and $f'\colon M'\to N$ be two smooth $\Z\pi$-homology equivalences that are diffeomorphically equivalent by a diffeomorphism $g\colon M\to M'$. For a framing $\tau\colon TM\to M\times \R^3$, let $g_*\tau\colon TM'\to M'\times\R^3$ denote the framing $(g\times \mathrm{id})\circ\tau\circ Dg^{-1}$ on $M'$. Then we have
\[ Z^\odd_\Theta(f,W,\tau)=Z^\odd_\Theta(f',W,g_*\tau).\]
\end{cor}
\begin{proof}
The assumption on $f,f',g$ gives a smooth homotopy $H\colon I\times M\to N$ such that $H(0,x)=f(x)$ and $H(1,x)=f'(g(x))$. The flat bundle $\calE_W$ over $N\times N$ is pulled back by $H'\colon I\times M\times M\to N\times N$ to $I\times \bConf_2(M)$, whose restriction to $\{1\}\times\bConf_2(M)$ gives $Z^\odd_\Theta(f'\circ g,W,\tau)=Z_\Theta^\odd(f,W,\tau)$ by Corollary~\ref{cor:inv-homotopy}. Since the $\calE_W^{f'\circ g}$-propagators to define $Z^\odd_\Theta(f'\circ g,W,\tau)$ are pullbacks of $\calE_W^{f'}$-propagators on $M'$, the integral to define $Z^\odd_\Theta(f'\circ g,W,\tau)$ can be rewritten as that over $\bConf_2(M')$, which gives $Z_\Theta^\odd(f',W,g_*\tau)$. This completes the proof.
\end{proof}
\subsection{Framing correction}

The invariant $Z_\Theta^\odd$ may depend on the choice of the framing $\tau$. We follow the known technique (\cite{BC}, see also \cite{CS}) to make it independent of the choice of $\tau$. Namely, we choose a compact oriented 4-manifold $X$ with $\partial X=M$. The {\it Hirzebruch signature defect} is defined by
\[ \delta(\tau):=\frac{1}{4}p_1(TX;\tau)[X,\partial X]-\frac{3}{4}\mathrm{sign}\,X, \]
where $p_1(TX;\tau)\in H^4(X,\partial X;\Q)$ is the relative Pontrjagin class with respect to the framing $\tau$ on $\partial X$, and $\mathrm{sign}\,X$ is the signature of $X$ (\cite{At, KM}). The proof of the following theorem is parallel to \cite[Theorem~5.1]{CS} and we omit the proof.

\begin{theorem}\label{thm:unframed-inv}
The element $\widehat{Z}_\Theta^\odd(f)\in\calS_\Theta^\odd(\mathrm{Ker}\,\ve)$ defined by
\[ \widehat{Z}_\Theta^\odd(f):=Z_\Theta^\odd(f,\mathrm{Ker}\,\ve,\tau)-\frac{1}{6}\pr_\Theta(\mathbf{1}_{\mathrm{Ker}\,\ve}^{\otimes 3})\,\delta(\tau) \]
is independent of the choice of $\tau$. Hence $\widehat{Z}_\Theta^\odd(f)$ is an invariant of $\Z\pi$-homology equivalence.
\end{theorem}

\subsection{The $\Theta$-invariants}\label{ss:Bott-Cattaneo}

We shall prove that $\widehat{Z}_\Theta^\odd(f)$ is universal among the $\Theta$-invariants.
Bott and Cattaneo interpreted in \cite{BC} the $\Theta$-invariant for {\it orthogonal} representations of $\pi_1M$, which can be straightforwardly generalized to arbitrary finite dimensional acyclic local systems as follows. Here we only consider acyclic local systems on $M$ induced from that on $N$ by a $\Z\pi$-homology equivalence $f\colon M\to N$. Namely, let $E$ be a finite dimensional $\pi$-module such that $H^*(N;E)=0$. Then we have $H^*(N;E^*)=0$ since the dual of nontrivial irreducible modules are nontrivial irreducible too. One can take $\calE_{E\boxtimes E^*}^f$-propagators $\omega_1,\omega_2,\omega_3\in \Omega^2_{\mathrm{dR}}(\bConf_2(M);E\boxtimes E^*)$, and obtains
\[ \omega_1\wedge\omega_2\wedge\omega_3\in \Omega_{\mathrm{dR}}^6(\bConf_2(M); (E\boxtimes E^*)^{\otimes 3}).\]
We choose $\pi$-equivariant linear maps $\rho_1\colon \C\to E\otimes E\otimes E$ and $\rho_2\colon E\otimes E\otimes E\to \C$. We define a linear map $\Tr_{\rho_1,\rho_2}\colon (E\boxtimes E^*)^{\otimes 3}\to \C$ by
\[ \Tr_{\rho_1,\rho_2}(\phi_1\otimes \phi_2\otimes \phi_3)
=\rho_2\bigl(\langle \rho_1(1),\phi_1\otimes \phi_2\otimes \phi_3\rangle\bigr),\]
where $\phi_1,\phi_2,\phi_3\in E\boxtimes E^*\cong \End(E)$, and $\langle -,-\rangle$ is the evaluation $E^{\otimes 3}\otimes \End(E)^{\otimes 3}\to E^{\otimes 3}$. The $\Theta$-invariant $J_{(\Theta,\rho_1,\rho_2)}$ for a $\Z\pi$-homology equivalence $f\colon M\to N$ and an acyclic local system $E$ on $N$ is defined by\footnote{In \cite{BC}, Bott and Cattaneo used propagators which have explicit formula on the boundary of the compactified configuration space, incorpolating some data from Riemannian metrics. We use propagators with the boundary behaviours as those of Kontsevich (\cite{Kon}), which use framings.}
\[ J_{(\Theta,\rho_1,\rho_2)}(f,E)=\frac{1}{6}\int_{\bConf_2(M)}\Tr_{\rho_1,\rho_2}(\omega_1\wedge \omega_2\wedge \omega_3)-\frac{1}{6}\Tr_{\rho_1,\rho_2}(\mathbf{1}_{E\boxtimes E^*}^{\otimes 3})\,\delta(\tau) \in \C. \]
It follows from Theorem~\ref{thm:inv-framed} that replacing $\omega_1\wedge\omega_2\wedge\omega_3$ in the definition of $J_{(\Theta,\rho_1,\rho_2)}$ with the symmetrization $\frac{1}{3!}\sum_{\sigma\in \mathfrak{S}_3}\omega_{\sigma(1)}\wedge\omega_{\sigma(2)}\wedge\omega_{\sigma(3)}\in \Omega_{\mathrm{dR}}^6(\bConf_2(M); \Sym^3(E\boxtimes E^*))$ gives the same result.

\begin{lemma}\label{lem:Tr-commute}
Let $S^3\rho_\pi\colon \Sym^3(\mathrm{Ker}\,\ve)\to \Sym^3(E\boxtimes E^*)$ be the linear map induced by the representation $\rho_\pi\colon \C\pi\to \End(E)\cong E\boxtimes E^*$, where we consider $\mathrm{Ker}\,\ve$ as a subspace of $\C\pi$. Then $S^3\rho_\pi$ is $\pi\times \pi$-equivariant, and it induces a linear map $\varpi_{E,\rho_1,\rho_2}\colon \bigl(\Sym^3(\mathrm{Ker}\,\ve)\bigr)_{\pi\times \pi}\to \C$ which makes the following diagram commutative.
\begin{equation}\label{eq:Tr-Tr}
 \xymatrix{
  \Sym^3(\mathrm{Ker}\,\ve) \ar[r]^-{S^3\rho_\pi} \ar[d]_-{\pr_\Theta} & \Sym^3(E\boxtimes E^*) \ar[d]^-{\Tr_{\rho_1,\rho_2}} \\
  \bigl(\Sym^3(\mathrm{Ker}\,\ve)\bigr)_{\pi\times \pi} \ar[r]^-{\varpi_{E,\rho_1,\rho_2}} & \C
} 
\end{equation}
\end{lemma}
\begin{proof}
The $\pi\times \pi$-equivariance of $S^3\rho_\pi$ is induced from that of $\rho_\pi\colon \C\pi\to \End(E)$. Then $S^3\rho_\pi$ induces a linear map $\overline{S^3\rho}_\pi\colon (\Sym^3(\mathrm{Ker}\,\ve))_{\pi\times \pi}\to (\Sym^3(E\boxtimes E^*))_{\pi\times \pi}$. Since $\Tr_{\rho_1,\rho_2}$ factors through $(\Sym^3(E\boxtimes E^*))_{\pi\times \pi}$ by the $\pi$-equivariance of $\rho_1$ and $\rho_2$, there is a linear map $\kappa_E\colon (\Sym^3(E\boxtimes E^*))_{\pi\times \pi}\to \C$ such that $\Tr_{\rho_1,\rho_2}=\kappa_E\circ\pr$, where $\pr\colon \Sym^3(E\boxtimes E^*)\to (\Sym^3(E\boxtimes E^*))_{\pi\times \pi}$ is the projection. The linear map $\varpi_{E,\rho_1,\rho_2}=\kappa_E\circ \overline{S^3\rho}_\pi$ is as desired.
\end{proof}

\begin{theorem}\label{thm:recover-I}
We have $J_{(\Theta,\rho_1,\rho_2)}(f,E)=\varpi_{E,\rho_1,\rho_2}(\widehat{Z}_\Theta^\odd(f))$.
Hence $\widehat{Z}_\Theta^\odd$ is universal among the $\Theta$-invariants $J_{(\Theta,\rho_1,\rho_2)}$ for finite dimensional acyclic local systems $E$ on $N$.
\end{theorem}
\begin{proof}
By the commutative diagram of Lemma~\ref{lem:Tr-commute}, we have
\begin{equation}\label{eq:Tr-rho}
\begin{split}
\textstyle \varpi_{E,\rho_1,\rho_2}\circ\pr_\Theta(\frac{1}{3!}\sum_{\sigma\in \mathfrak{S}_3}\omega_{\sigma(1)}\wedge\omega_{\sigma(2)}\wedge\omega_{\sigma(3)})
 &=\textstyle\Tr_{\rho_1,\rho_2}\circ S^3\rho_\pi(\frac{1}{3!}\sum_{\sigma\in \mathfrak{S}_3}\omega_{\sigma(1)}\wedge\omega_{\sigma(2)}\wedge\omega_{\sigma(3)})\\
&=\textstyle\Tr_{\rho_1,\rho_2}\bigl(\frac{1}{3!}\sum_{\sigma\in \mathfrak{S}_3}\rho_\pi(\omega_{\sigma(1)})\wedge\rho_\pi(\omega_{\sigma(2)})\wedge\rho_\pi(\omega_{\sigma(3)})\bigr).
\end{split} 
\end{equation}
The integral of the LHS of \eqref{eq:Tr-rho} gives $\varpi_{E,\rho_1,\rho_2}(Z_\Theta^\odd(f,\mathrm{Ker}\,\ve,\tau))$. Since $\rho_{\pi}(\omega_i)$ is a $\calE_{E\boxtimes E^*}^f$-propagator, the integral of the RHS of \eqref{eq:Tr-rho} gives the integral term in the formula for $J_{(\Theta,\rho_1,\rho_2)}(f,E)$. Since we have
\[ S^3\rho_\pi(\mathbf{1}_{\mathrm{Ker}\,\ve}^{\otimes 3})=\mathbf{1}_{E\boxtimes E^*}^{\otimes 3},\]
the correction terms for $\varpi_{E,\rho_1,\rho_2}(\widehat{Z}_\Theta^\odd(f))$ and $J_{(\Theta,\rho_1,\rho_2)}(f,E)$ agree. This completes the proof.
\end{proof}

\begin{remark}
In contrast to \cite{CS}, the dumbbell correction term is not needed in the definitions of $\widehat{Z}_\Theta^\odd$ and $J_{(\Theta,\rho_1,\rho_2)}(f,E)$ in our setting. This is because $H^2(\Delta_M;W_\Delta)=0$ by Lemmas~\ref{lem:H(diagonal)} and \ref{lem:M=QHS}. When $H^2(\Delta_M;W_\Delta)\neq 0$, the dumbbell correction term may be necessary (see Remark~\ref{rem:propagator-BC}-2 and \cite[Section 3]{CS}).
\end{remark}

\subsection{Antisymmetric propagators and $\Z_2$-invariance}

Let $T\colon \bConf_2(M)\to \bConf_2(M)$ be the smooth involution induced by the flip $M\times M\to M\times M$; $(x,y)\mapsto (y,x)$. Now we make the following assumption.
\begin{assum}\label{assum:symmetric}
A finite dimensional $\pi\times \pi$-module $W$ is such that
 $W$ is invariant under taking the duals $\End(A_i)\to \End(A_i^*)$ on the irreducible factors of $W$ (see Lemma~\ref{lem:flip}). 
\end{assum}

We extend the involution $T$ to an involution of such a finite dimensional local system $W$ on $\bConf_2(M)$ by defining the involution $(\cdot)^*\colon W\to W$, $\phi\mapsto \phi^*$, which is the $\C$-linear operation induced by the involution $\C\pi\to \C\pi$; $g\mapsto g^{-1}$. On the irreducible factor $A_i\boxtimes A_i^*=\End(A_i)$ of $W$, it corresponds to taking the dual morphism and induces a $\C$-linear map $\End(A_i)\to \End(A_i^*)$ between the factors of $W$ (see Lemma~\ref{lem:flip}). Then we define the involution $\widetilde{T}^*\colon \Omega_{\mathrm{dR}}^*(\bConf_2(M);\calE_W^f)\to \Omega_{\mathrm{dR}}^*(\bConf_2(M);\calE_W^f)$ by\footnote{This is a natural pullback operation for the flat bundle automorphism $\calE_W^f\to \calE_W^f$; $(x,y,\phi)\mapsto (y,x,\phi^*)$, which is compatible with the covering automorphism of the fundamental groupoid $\Pi(M)\to M\times M$ induced by the flip $M\times M\to M\times M$.} $\widetilde{T}^*(\alpha\otimes \phi)=T^*\alpha\otimes \phi^*$. It is clear from the definition that $\widetilde{T}^*$ on $\Omega_{\mathrm{dR}}^*(\bConf_2(M);\calE_W^f)$ commutes with taking $(\,\cdot\,)^*$ on the coefficients.
We say that a $\calE_W^f$-propagator $\omega\in\Omega_{\mathrm{dR}}^2(\bConf_2(M);\calE_W^f)$ is {\it antisymmetric} if it satisfies
\[ \widetilde{T}^*\omega=-\omega^*. \]

\begin{lemma}
For a $\pi\times\pi$-module $W$ that is acyclic with respect to $M\times M$ and satisfies Assumption~\ref{assum:symmetric}, there exists an antisymmetric $\calE_W^f$-propagator in $\Omega_{\mathrm{dR}}^2(\bConf_2(M);\calE_W^f)$. 
\end{lemma}
\begin{proof}
For a $\calE_W^f$-propagator $\omega\in\Omega_{\mathrm{dR}}^2(\bConf_2(M);\calE_W^f)$, let $\omega':=\frac{1}{2}\bigl(\omega-(\widetilde{T}^*\omega)^*\bigr)$.
This is a $\calE_W^f$-propagator since $T^*\mathrm{Vol}_{S^2}=-\mathrm{Vol}_{S^2}$ and $(\mathbf{1}_W)^*=\mathbf{1}_W$.
Then we have
  $\widetilde{T}^*\omega'=\frac{1}{2}\bigl(\widetilde{T}^*\omega-\widetilde{T}^*(\widetilde{T}^*\omega)^*\bigr)
  =\frac{1}{2}\bigl(\widetilde{T}^*\omega-\omega^*\bigr)=-(\omega')^*$.
\end{proof}

\begin{prop}\label{prop:Z2-invariance}
The image of $\widehat{Z}_\Theta^\odd$ is included in $\calA_\Theta^\odd(\mathrm{Ker}\,\ve)$.
\end{prop}
\begin{proof}
If we choose $\omega_1,\omega_2,\omega_3$ in the definition of $Z_\Theta^\odd$ as antisymmetric $\calE_{\mathrm{Ker}\,\ve}^f$-propagators, then we have $\widetilde{T}^*\pr_\Theta(\omega_1\wedge\omega_2\wedge\omega_3)=\pr_\Theta\widetilde{T}^*(\omega_1\wedge\omega_2\wedge\omega_3)=(-1)^3\pr_\Theta(\omega_1^*\wedge\omega_2^*\wedge\omega_3^*)=-\pr_\Theta(\omega_1^*\wedge\omega_2^*\wedge\omega_3^*)$.
Since $T$ reverses the orientation of the configuration space $\bConf_2(M)$, we have
\[ \int_{\bConf_2(M)}\pr_\Theta(\omega_1\wedge\omega_2\wedge\omega_3)=-\int_{\bConf_2(M)}\widetilde{T}^*\pr_\Theta(\omega_1\wedge\omega_2\wedge\omega_3)=\int_{\bConf_2(M)}\pr_\Theta(\omega_1^*\wedge\omega_2^*\wedge\omega_3^*). \]
It follows that
\[ \int_{\bConf_2(M)}\pr_\Theta(\omega_1\wedge\omega_2\wedge\omega_3)=\int_{\bConf_2(M)}\frac{1}{2}\Bigl(\pr_\Theta(\omega_1\wedge\omega_2\wedge\omega_3)+\pr_\Theta(\omega_1^*\wedge\omega_2^*\wedge\omega_3^*)\Bigr), \]
and its value is in $((\Sym^3(\mathrm{Ker}\,\ve))_{\pi\times\pi})^{\Z_2}$. We may also see that $\pr_\Theta(\mathbf{1}_{\mathrm{Ker}\,\ve}^{\otimes 3})\delta(\tau)$ is $\Z_2$-invariant since the identity $\mathbf{1}_{\mathrm{Ker}\,\ve}$ is invariant under the inversion.
\end{proof}

\begin{theorem}\label{thm:upper-bound2}
Let $\calI_\Theta(N)$ denote the subspace of $\Map(\calH(N),\C)$ spanned by the $\Theta$-invariants $J_{(\Theta,\rho_1,\rho_2)}(-,A)$ for finite dimensional acyclic local systems $A$ on $N$ and $\pi$-equivariant linear maps $\rho_1\colon \C\to A^{\otimes 3}$, $\rho_2\colon A^{\otimes 3}\to \C$. Then there is a canonical epimorphism 
\[ \Hom_\C(\calA_\Theta^\odd(\mathrm{Ker}\,\ve),\C)\to \calI_\Theta(N). \]
\end{theorem}
\begin{proof}
We consider the canonical map $\Psi\colon\Hom_\C(\calA_\Theta^\odd(\mathrm{Ker}\,\ve),\C)\to \Map(\calH(N),\C)$ defined by $\varpi\mapsto \varpi\circ \widehat{Z}_\Theta^\odd$. By Theorem~\ref{thm:recover-I}, the image of $\Psi$ includes $\calI_\Theta(N)$. Note that $\varpi_{A,\rho_1,\rho_2}\colon \bigl(\Sym^3(\mathrm{Ker}\,\ve)\bigr)_{\pi\times \pi}\to \C$ induces a linear map on the $\Z_2$-invariant part $\bigl(\bigl(\Sym^3(\mathrm{Ker}\,\ve)\bigr)_{\pi\times \pi}\bigr)^{\Z_2}$. 

To see that $\Im\Psi=\calI_\Theta(N)$, it suffices to prove that any linear map $\varpi\colon \bigl(\Sym^3(\mathrm{Ker}\,\ve)\bigr)_{\pi\times \pi}\to \C$ is given by a linear combination of the induced maps $\varpi_{E,\rho_1,\rho_2}$ in the diagram \eqref{eq:Tr-Tr} for some triples $(E,\rho_1,\rho_2)$. We take the universal choice
\[ E=\bigoplus_i A_i, \]
where the direct sum is over all distinct nontrivial irreducible $\pi$-modules. The representation $\rho_\pi\colon \C\pi\to \End(E)$ induces an embedding $\mathrm{Ker}\,\ve=\bigoplus_i\End(A_i)\to \End(E)=\End(\bigoplus_i A_i)$ of $\pi\times\pi$-modules. Moreover, $S^3\rho_\pi\colon \Sym^3(\mathrm{Ker}\,\ve)\to \Sym^3(\End(E))\cong \Sym^3(E\boxtimes E^*)$ is an embedding of $\pi\times \pi$-modules. More explicitly, we have the following decomposition of $\pi\times\pi$-modules:
\[E\boxtimes E^*=\Bigl(\bigoplus_i A_i\Bigr)\boxtimes \Bigl(\bigoplus_i A_i^*\Bigr)=\bigoplus_i(A_i\boxtimes A_i^*)\oplus V=\mathrm{Ker}\,\ve\oplus V\]
for some $\pi\times \pi$-module $V$, where the image of $\mathrm{Ker}\,\ve$ under $\rho_\pi$ is the corresponding term in this decomposition. Hence we have the decomposition
\[ \Sym^3(E\boxtimes E^*)
=\Sym^3(\mathrm{Ker}\,\ve)\oplus \Sym^3 V\oplus (\Sym^2(\mathrm{Ker}\,\ve)\otimes V)\oplus (\mathrm{Ker}\,\ve\otimes\Sym^2 V) \]
of $\pi\times\pi$-modules. The image of $\Sym^3(\mathrm{Ker}\,\ve)$ under the map $S^3\rho_\pi$ is the corresponding term in this decomposition. We define $\kappa_E\colon \bigl(\Sym^3(E\boxtimes E^*)\bigr)_{\pi\times \pi}\to \C$ by $\varpi\oplus 0\oplus 0\oplus 0$. Then we have $\varpi\circ\pr_\Theta=\kappa_E\circ\pr\circ S^3\rho_\pi$. 

The proof will complete if we prove that $\kappa_E\circ \pr\colon \Sym^3(E\boxtimes E^*)\to \C$ is a linear combination of $\pi\times\pi$-equivariant maps of the forms $\Tr_{\rho_1,\rho_2}$ for some pairs $(\rho_1,\rho_2)$. 
We consider the following decomposition of $\pi\times\pi$-modules (\cite[Exercise~6.5]{FH})
\[ E^{\otimes 3}\boxtimes (E^*)^{\otimes 3}=(E\boxtimes E^*)^{\otimes 3}=\Sym^3(E\boxtimes E^*)\oplus \tbigwedge^3(E\boxtimes E^*)\oplus (\mathbb{S}_{(2,1)}(E\boxtimes E^*))^{\oplus 2}, \]
where $\mathbb{S}_\lambda$ is the Schur functor or Weyl module for a partition $\lambda$ of an integer (\cite[\S{6.1}]{FH}). We recall that a finite dimensional irreducible $\pi\times\pi$-module is of the form $A_i\boxtimes A_j^*$ for irreducible $\pi$-modules $A_i,A_j^*$, and that a trivial $\pi\times \pi$-module is isomorphic to a direct sum $(U\boxtimes U)^{\oplus m}$ for the one dimensional trivial $\pi$-module $U$. It follows that $(E^{\otimes 3}\boxtimes (E^*)^{\otimes 3})_{\pi\times \pi}=(E^{\otimes 3})_\pi\boxtimes ((E^*)^{\otimes 3})_\pi$. Hence we have that $\Hom\bigl(\bigl(\Sym^3(E\boxtimes E^*)\bigr)_{\pi\times \pi},\C\bigr)$ is a direct summand of 
$\Hom((E^{\otimes 3})_\pi,\C)\boxtimes \Hom(((E^*)^{\otimes 3})_\pi,\C)$,
an element of which is induced by a linear combination of $\pi\times\pi$-equivariant maps $E^{\otimes 3}\boxtimes (E^*)^{\otimes 3}\to \C$ of the forms $\Tr_{\rho_1,\rho_2}$.
\end{proof}

\begin{remark}
The following formula for $\pi\times\pi$-modules is known (\cite[Exercise~6.11 (b)]{FH}) 
\[ \Sym^d(E\boxtimes E^*)=\bigoplus_\lambda \mathbb{S}_\lambda E\boxtimes \mathbb{S}_\lambda E^*, \]
where the direct sum is over partitions $\lambda$ of $d$ with at most $\dim{E}$ rows. If $\dim{E}\geq 3$, we have
\[ \begin{split}
\Sym^3(E\boxtimes E^*)&=(\mathbb{S}_{(3)}E\boxtimes \mathbb{S}_{(3)}E^*)
\oplus (\mathbb{S}_{(1,1,1)}E\boxtimes \mathbb{S}_{(1,1,1)}E^*)
\oplus (\mathbb{S}_{(2,1)}E\boxtimes \mathbb{S}_{(2,1)}E^*)\\
&=(\Sym^3 E\boxtimes \Sym^3 E^*)
\oplus (\tbigwedge^3 E\boxtimes \tbigwedge^3 E^*)
\oplus (\mathbb{S}_{(2,1)}E\boxtimes \mathbb{S}_{(2,1)}E^*), 
\end{split} \]
and if $\dim{E}< 3$, some terms may be omitted from this formula. In any case, we have
\[ \Hom\bigl(\bigl(\Sym^3(E\boxtimes E^*)\bigr)_{\pi\times \pi},\C\bigr)=\bigoplus_\lambda \Hom((\mathbb{S}_\lambda E)_\pi,\C)\boxtimes \Hom((\mathbb{S}_\lambda E^*)_\pi,\C). \]
The example in \cite{AS,Kon}, \cite[Example~2.2]{BC}, $x\otimes y\otimes z\mapsto B([x,y],z)$ and its dual, which use the Lie bracket and the Killing form $B$ of the Lie algebra of a compact Lie group $G$, corresponds to an element of the component $\Hom((\tbigwedge^3 E)_\pi,\C)\boxtimes \Hom((\tbigwedge^3 E^*)_\pi,\C)$.
\end{remark}

\section{Computation of $\dim\calA_\Theta^\odd(\C\pi)$ and $\dim\calA_\Theta^\odd(\mathrm{Ker}\,\ve)$}\label{s:computation_dim}

This section gives a proof of Theorem \ref{thm:dim-formula}, that is, we completely determine the dimensions of the spaces $\calA_\Theta^\odd(\C\pi)$ and $\calA_\Theta^\odd(\mathrm{Ker}\,\ve)$ associated with the fundamental group $\pi=\pi_1N$ of any spherical 3-manifold $N$ according to the classification of such $\pi$ given in Theorem~\ref{thm:classification}. 

\subsection{Formulas for computation of $\dim \calA_\Theta^\odd(\C\pi)$ and $\dim\calA_\Theta^\odd(\mathrm{Ker}\,\ve)$} \label{section:4.1.1}

We prepare a formula to compute $\dim \calA_\Theta^\odd(\C\pi)$ by using irreducible characters of a finite group $\pi$. 
It generalizes the formula introduced in \cite[\S 4.3]{OW}, where all irreducible representations of a finite group are supposed to have real-valued characters. 
Our formula presented in this section includes irreducible representations with non-real valued characters. 

Suppose that all distinct (up to equivalence) irreducible representations of a finite group $\pi$ over $\C$ consist of $A_1, \ldots, A_{r_1 + 2r_2}$ $(r_1\geq 1, r_2 \geq 0)$ such that $A_i$ ($i=1,2,\ldots, r_1$) has real valued characters and $A_j$ ($j=r_1+1 ,r_1 + 2,\ldots, r_1+2r_2$) has non-real complex valued characters. 
Then, the RHS of isomorphisms \eqref{eq:AW-isom} are further decomposed into the direct sums of irreducible components with real-valued characters and those with non-real-valued characters as follows:
\begin{equation}\label{eq:AW-isom_4.1}
\begin{split}
    & \mathrm{Ker}\,\ve\cong \bigoplus_{i=2}^{r_1} (A_i\boxtimes A_i)\oplus \bigoplus_{j=r_1 + 1}^{r_1 + r_2} (A_j\boxtimes A_{j}^{\ast} \oplus A_{j}^{\ast} \boxtimes A_{j}),\quad \C\pi\cong (A_1\boxtimes A_1)\oplus \mathrm{Ker}\,\ve,
\end{split}
\end{equation}
as $\pi \times \pi$-modules. 
By the isomorphism $\calA_\Theta^\odd(\C \pi) \cong(\Sym^3 \C\pi)^{(\pi\times \pi) \rtimes \Z_2}$ as $(\pi\times \pi) \rtimes \Z_2$-modules (cf. \S 2.1), the dimension $\dim \calA_\Theta^\odd(\C\pi)$ can be computed by applying the formula \eqref{eq:dim_invariant_part} as follows.
\begin{equation}
	\dim \calA_\Theta^\odd(\C\pi)= \frac{1}{2} \left( d_1(\C\pi) + d_2(\C\pi) \right),
\end{equation}
where, denoting $\C\pi$ by $W$,
\begin{equation}
\begin{split}
  d_1(\C\pi) &= \dim\,(\Sym^3 \C\pi )^{\pi\times \pi}=\displaystyle\frac{1}{6|\pi|^2}\sum_{g,h\in\pi}(\chi_W(g,h)^3+3 \chi_W(g^2,h^2) \chi_W(g,h)+2 \chi_W(g^3,h^3)),\\
  d_2(\C\pi) &=\displaystyle\frac{1}{6|\pi|^2}\sum_{g,h\in\pi}(\chi_W(\theta\cdot(g,h))^3 +3 \chi_W((\theta\cdot (g,h))^2) \chi_W(\theta\cdot(g,h))+2 \chi_W((\theta\cdot(g,h))^3),
    \end{split}
\end{equation}
and $\theta$ denotes the generator of $\Z_2 = \langle \theta\rangle$.
The first term $d_1(\C\pi)$ is computed by using $\chi_W(g,h)=\sum_{i=1}^{r_1+2r_2}\chi_{A_i}(g)\overline{\chi_{A_i}(h)}$, $\chi_W(g^2,h^2)=\sum_{i=1}^{r_1+2r_2}\chi_{A_i}(g^2)\overline{\chi_{A_i}(h^2)}$, and so on. To compute the second term $d_2(\C\pi)$, we need analyze a way to compute $\chi_W(\theta\cdot(g,h))$. 
For this purpose, we introduce the following lemma describing the induced action of $\Z_2$ on RHS of \eqref{eq:AW-isom_4.1}.

\begin{lemma}\label{lem:flip}
Let $\rho\colon \C\pi \rightarrow \bigoplus_{i=1}^{r_1 + 2r_2} \End(A_i)$ denote the isomorphism in \eqref{eq:AW-isom_4.1}. We consider $\bigoplus_{i=1}^{r_1 + 2r_2} \End(A_i)$ as a $(\pi \times \pi) \rtimes \Z_2$-module with the $\Z_2$-action induced by taking the dual $\phi \mapsto \phi^{\ast}$ $(\phi \in \End(A_i))$. Then, $\rho$ is an isomorphism of $(\pi \times \pi)\rtimes \Z_2$-modules. Moreover, the factors
\[
A_i\boxtimes A_i \ (i=1,\ldots, r_1), \quad (A_j\boxtimes A_{j}^{\ast} \oplus A_{j}^{\ast} \boxtimes A_{j})\  (j=r_1+1,\ldots, r_1 + r_2)
\]
in $\bigoplus_{i=1}^{r_1 + 2r_2} \End(A_i)$ are invariant under the $\Z_2$-action and the actions are given by 
\[\begin{array}{lcll}
    x \boxtimes y & \mapsto & \nu(A_i) y \boxtimes x & \text{\rm (for $x \boxtimes y \in A_i\boxtimes A_i$, $i=1,\ldots,r_1$), and}\\
    x \boxtimes y \oplus z \boxtimes w & \mapsto & w \boxtimes z \oplus y \boxtimes x & \text{\rm (for $x \boxtimes y \oplus z \boxtimes w \in A_j\boxtimes A_{j}^{\ast} \oplus A_{j}^{\ast} \boxtimes A_{j}$, $j=r_1+1,\ldots,r_1+r_2$)},
\end{array}\]
respectively. Here, $\nu(A_i)$ denotes the Frobenius--Schur indicator of $A_i$ defined by
\[
\nu(A_i) \coloneqq \frac{1}{|\pi|} \sum_{g \in \pi} \chi_{A_i}(g^2).
\]
\end{lemma}
\begin{proof}
    First, we check that $\rho$ is $(\pi \times \pi) \rtimes \Z_2$-linear. We need only to show that $\rho$ is $\Z_2$-linear since we know that $\rho$ is an isomorphism of $\pi\times \pi$-modules. For $\varphi\in A_i^*$, we have
    $\rho_{A_i}(g^{-1})^*(\varphi)=(g^{-1})^*\varphi=\rho_{A_i^*}(g)(\varphi)$
    by the $\pi$-module structure on $A_i^*$. Hence we have
    $\rho_{A_i}(\theta\cdot g)=\rho_{A_i}(g^{-1})=\rho_{A_i^*}(g)^*=\theta\cdot \rho_{A_i^*}(g)$, 
    and $\rho$ is $\Z_2$-linear. 
    Note that for $i=1,\ldots, r_1$, there is an isomorphism $A_i \cong A_i^{\ast}$ as $\pi$-modules since $A_i$ has real-valued character (cf. \cite[Theorem~3.37]{FH}). Hence $\theta\cdot \rho_{A_i^*}(g)=\rho_{A_i}(g)^*$. For $j=r_1 +1,\ldots, r_1 + r_2$, the factor $(\rho_{A_j}(g), \rho_{A_j^{\ast}}(g))\in\End(A_j)\oplus\End(A_j^*)$ in $\rho(g)=(\rho_{A_i}(g))_i$ is mapped to $(\rho_{A_j} (g)^{\ast}, \rho_{A_j^{\ast}}(g)^{\ast})\in\End(A_j^*)\oplus\End(A_j)$ in $\rho(g^{-1})$ under the $\Z_2$-action on $\bigoplus_i\End(A_i)$. 
    
    Next, we describe the $\Z_2$-actions on $A_i \boxtimes A_i$ for $i=1,\ldots,r_1$ concretely. 
    Thus we consider the case that $A_i$ has real-valued character.
    In this case, it is known that $A_i$ admits either a $\pi$-invariant nondegenerate symmetric bilinear form or a $\pi$-invariant nondegenerate skew-symmetric bilinear form,  corresponding respectively to the cases $\nu(A_i)=1$ and $\nu(A_i) =-1$ (\cite[Theorem~3.37, Exercise~3.38]{FH}). When $\nu(A_i)=1$, we may assume that the invariant bilinear form is positive definite, and we can find an orthonormal basis of $A_i$ with respect to the inner product, which is also the dual basis of itself. The orthonormal basis of $A_i$ gives an identification $A_i^*=A_i$, under which the $\Z_2$-action can be represented by the matrix transpose. Thus $\Z_2$ acts on $A_i\boxtimes A_i$ by $x\boxtimes y\mapsto y\boxtimes x$. For the case $\nu(A_i)=-1$, it follows from Lemma~\ref{lem:flip-quaternionic} below that $\Z_2$ acts on $A_i\boxtimes A_i$ by $x\boxtimes y\mapsto -y\boxtimes x$.  

     Finally, we describe the $\Z_2$-actions on $A_j\boxtimes A_{j}^{\ast} \oplus A_{j}^{\ast} \boxtimes A_{j}$ for $j=r_1+1,\ldots,r_1+r_2$ concretely. For this case ($\nu(A_j)=0$), it follows from Lemma~\ref{lem:flip-complex} below that $\Z_2$ acts on $A_j\boxtimes A_j^*\oplus A_j^*\boxtimes A_j$ by $x\boxtimes y\oplus z\boxtimes w\mapsto w\boxtimes z\oplus y\boxtimes z$.  
\end{proof}

\begin{lemma}\label{lem:flip-quaternionic}
Let $V$ be a complex vector space of dimension $2n$ with a symplectic structure $\omega\colon V\otimes V\to \C$. Let $\iota_\omega\colon V\to V^*$ be the isomorphism defined by $\iota_\omega(x)=\omega(x,-)$. We choose a symplectic basis $e_1,\ldots,e_{2n}$ of $V$. 
Let $\theta\colon V\otimes V\to V\otimes V$ be the linear map defined by $\theta(x\otimes y)=-y\otimes x$. Let $\widetilde{\theta}\colon \End(V)\to \End(V^*)$ be the linear map defined by $\widetilde{\theta}(\varphi)=\varphi^*$. Let $\mu\colon \End(V^*)\to \End(V)$ be the map defined by $\mu(\psi)=\iota_\omega^{-1}\psi\iota_\omega$. 
Then the following diagram commutes (see \S\ref{ss:rep} for $\lambda_V$).
\[ \xymatrix{
    V\otimes V \ar[rr]^-{\mathbf{1}\otimes \iota_\omega}_{\cong} \ar[d]_-{\theta} && V\otimes V^* \ar[rrr]^-{\lambda_V}_{\cong} &&& \End(V) \ar[d]^-{\widetilde{\theta}} \ar[d]\\
    V\otimes V \ar[rr]^-{\mathbf{1}\otimes\iota_\omega}_\cong && V\otimes V^* \ar[rr]^-{\lambda_V}_\cong && \End(V) & \End(V^*) \ar[l]_-{\mu}^-{\cong}
} \]
\end{lemma}
\begin{proof}
Let $e_1^*,\ldots,e_{2n}^*\in V^*$ be the dual basis of $e_1,\ldots,e_{2n}$. Also, $\iota_\omega(e_1),\ldots,\iota_\omega(e_{2n})$ form another basis of $V^*$.
Then the two bases of $V^*$ are related by
\[ (e_1^*,\ldots,e_{2n}^*)=(\iota_\omega(e_1),\ldots,\iota_\omega(e_{2n}))J\quad \text{or}\quad e_j^*=\sum_k J_{kj}\iota_\omega(e_k),\quad \text{ where }\quad J=\begin{pmatrix}
    0 & \mathbf{1}\\
    -\mathbf{1} & 0
\end{pmatrix}. \]
We identify $\End(V)$ with the set of all $2n\times 2n$ matrices by taking the matrix representation with respect to the symplectic basis. Then we have
$\mu\widetilde{\theta}(M)=J M^T J^{-1}$. 
Let $E(i,j)\in \End(V)$ be the matrix whose $(i, j)$ entry is 1 and all other entries are 0. It suffices to prove the following identity. 
\begin{equation}\label{eq:tilde-theta}
\mu\widetilde{\theta}(E(i,j))=\lambda_V(\mathbf{1}\otimes\iota_\omega)\theta (\mathbf{1}\otimes\iota_\omega)^{-1}\lambda_V^{-1}(E(i,j)). 
\end{equation}
The $(k,\ell)$ entry of the left hand side of \eqref{eq:tilde-theta} can be rewritten as follows.
\[ \mu\widetilde{\theta}(E(i,j))_{k\ell}=(JE(i,j)^T J^{-1})_{k\ell}
=\sum_{p,q}J_{kp}(E(j,i))_{pq}(J^{-1})_{q\ell}
=J_{kj}(J^{-1})_{i\ell}. \]
The right hand side of \eqref{eq:tilde-theta} can be rewritten as follows.
\[ \begin{split}
&\lambda_V(\mathbf{1}\otimes\iota_\omega)\theta (\mathbf{1}\otimes\iota_\omega)^{-1}(e_i\otimes e_j^*)=\lambda_V(\mathbf{1}\otimes\iota_\omega)\theta (\mathbf{1}\otimes\iota_\omega)^{-1}\Bigl(\sum_k J_{kj}e_i\otimes \iota_\omega(e_k)\Bigr) \\
& =\lambda_V(\mathbf{1}\otimes\iota_\omega)\Bigl(\sum_k (-J_{kj})e_k\otimes e_i\Bigr)=\lambda_V\Bigl(\sum_k (-J_{kj})e_k\otimes \iota_\omega(e_i)\Bigr)
\\
&=\lambda_V\Bigl(\sum_{k,\ell} (-J_{kj})(J^{-1})_{\ell i}e_k\otimes e_\ell^*\Bigr)=\sum_{k,\ell}(-J_{kj})(J^{-1})_{\ell i}E(k,\ell).
\end{split} \]
By $J^T=J^{-1}=-J$, we have
$(-J_{kj})(J^{-1})_{\ell i}=(-J_{kj})(-J_{\ell i})=J_{kj}J_{\ell i}=J_{kj}(J^{-1})_{i\ell}$. 
This completes the proof of \eqref{eq:tilde-theta}.
\end{proof}

\begin{lemma}\label{lem:flip-complex}
Let $V$ be a complex vector space of dimension $n$. 
Let $\theta\colon V\otimes V^*\to V^*\otimes V$ be the linear map defined by $\theta(x\otimes y)=y\otimes x$. Let $\widetilde{\theta}\colon \End(V)\to \End(V^*)$ be the linear map defined by
$\widetilde{\theta}(\varphi)=\varphi^*$.
Then the following diagram commutes (see \S\ref{ss:rep} for $\lambda_V$).
\[ \xymatrix{
     V\otimes V^* \ar[rr]^-{\lambda_V}_{\cong} \ar[d]_-{\theta} && \End(V) \ar[d]^-{\widetilde{\theta}} \\
     V^*\otimes V \ar[rr]^-{\lambda_{V^*}}_\cong && \End(V^*)
} \]
\end{lemma}
\begin{proof}
We take a basis $e_1,\ldots,e_n$ of $V$ and its dual basis $e_1^*,\ldots,e_n^*$. These bases allow us to identify $\End(V)$ and $\End(V^*)$ with the set of all $n\times n$ matrices. Under these identifications, the map $\widetilde{\theta}$ is represented by the matrix transpose. Hence we have
$\lambda_{V^*}\theta\lambda_V^{-1}(E(i,j))
    =\lambda_{V^*}\theta (e_i\otimes e_j^*)
    =\lambda_{V^*} (e_j^*\otimes e_i)=E(j,i)=\widetilde{\theta}(E(i,j))$,
where $E(i,j)\in \End(V)$ is the matrix whose $(i, j)$ entry is 1 and all other entries are 0. This completes the proof.
\end{proof}

\begin{lemma}\label{lem:chi-theta-g-h}
For $i=1, \ldots, r_1$ and  $g, h \in \pi$, the following identity holds.
\[
\chi_{A_i\boxtimes A_i}(\theta\cdot (g,h))=\chi'_{A_i}(hg)
\]
where $\chi'_{A_i}$ is defined by
\[
\chi_{A_i}'(g) \coloneqq \nu(A_i) \cdot \chi_{A_i}(g).
\]
\end{lemma}
\begin{proof}
    Let $e_1, \ldots, e_{\dim A_i}$ be any basis of $A_i$ and $e_1^*, \ldots, e_{\dim A_i}^*$ be the corresponding Kronecker dual basis of $A_i^{\ast}$ so that $\langle e_a^*, e_b\rangle = \delta_{ab}$ where $\langle -, -\rangle$ denotes the canonical pairing between $A_i^{\ast}$ and $A_i$. Then by Lemma~\ref{lem:flip}, we have
    \[
    \begin{split}
    &\chi_{A_i\boxtimes A_i}(\theta\cdot (g,h)) = \sum_{a,b} \langle e_a^*\boxtimes e_b^*, (\theta \cdot (g,h))(e_a \boxtimes e_b) \rangle =  \sum_{a, b} \langle e_a^*\boxtimes e_b^*, \theta (\rho_{A_i}(g) e_a \boxtimes \rho_{A_i}(h) e_b)\rangle\\
    &=\nu(A_i) \sum_{a, b} \langle e_a^*\boxtimes e_b^*, \rho_{A_i}(h) e_b \boxtimes \rho_{A_i}(g) e_a \rangle = \nu(A_i) \sum_{a, b} \langle e_a^*, \rho_{A_i}(h) e_b  \rangle  \langle  e_b^*,  \rho_{A_i}(g) e_a \rangle \\
    &= \nu(A_i)\sum_{a} \langle e_a^*, \rho_{A_i}(h) \rho_{A_i}(g) e_a \rangle = \nu(A_i) \chi_{A_i}(hg) = \chi'_{A_i}(hg)
    \end{split}
    \]
    as desired. 
\end{proof}

On the other hand, $\theta$  acts on $A_j\boxtimes A_{j}^{\ast} \oplus A_{j}^{\ast}\boxtimes A_{j}$ ($j=r_1+1,r_1 + 2,\ldots, r_1 + r_2$) by $x \boxtimes y \oplus z \boxtimes w \mapsto w \boxtimes z \oplus y \boxtimes x$. 
It means that $\theta$ induces a non trivial permutation on $A_j\boxtimes A_{j}^{\ast} \oplus A_{j}^{\ast}\boxtimes A_{j}$  ($j=r_1+1,r_1 + 2,\ldots, r_1 + r_2$) and hence the trace of the action of $\theta \cdot (g,h)$ on $A_j\boxtimes A_{j}^{\ast} \oplus A_{j}^{\ast}\boxtimes A_{j}$  is zero. Therefore, we conclude that only  irreducible representations with real-valued characters contribute to $\chi_W(\theta\cdot(g,h))$, that is, by Lemma~\ref{lem:chi-theta-g-h}, 
\begin{equation}
    \chi_W(\theta\cdot(g,h))=\sum_{i=1}^{r_1} \chi_{A_i \boxtimes A_i}(\theta\cdot (g,h))=\sum_{i=1}^{r_1} \chi_{A_i}'(hg).
\end{equation}
Since $\nu(A_j) =0$ $(j=r_1+1, \ldots, r_1 +r_2)$,  the rightmost term in the above equation can also be written as a sum over all indices. Keeping these in mind, the computation of $d_2(\C)$ is simplified as follows.
\[ \begin{split}
\sum_{g,h}\chi_W(\theta\cdot(g,h))^3&=\sum_{g,h}\Bigl(\sum_{i=1}^{r_1 + 2r_2} \chi_{A_i\boxtimes A_i^*}(\theta\cdot(g,h))\Bigr)^3
 =\sum_h\sum_g\Bigl(\sum_{i=1}^{r_1} \chi_{A_i}'(hg)\Bigr)^3\\
&=\sum_h\sum_g\Bigl(\sum_{i=1}^{r_1} \chi_{A_i}'(g)\Bigr)^3=|\pi|\sum_g\Bigl(\sum_{i=1}^{r_1}\chi_{A_i}'(g)\Bigr)^3,\\
\sum_{g,h}\chi_W((\theta\cdot(g,h))^2)\chi_W(\theta\cdot(g,h))
&=\sum_{g,h}\Bigl(\sum_{i=1}^{r_1 + 2r_2}\chi_{A_i\boxtimes A_i^*}(hg,gh)\Bigr)\Bigl(\sum_{j=1}^{r_1}\chi_{A_j}'(hg)\Bigr)\\
&=\sum_{g,h}\Bigl(\sum_{i=1}^{r_1 + 2r_2} |\chi_{A_i}(hg)|^2\Bigr)\Bigl(\sum_{j=1}^{r_1}\chi_{A_j}'(hg)\Bigr)\\
&=|\pi|\sum_g\Bigl(\sum_{i=1}^{r_1 + 2r_2} |\chi_{A_i}(g)|^2\Bigr)\Bigl(\sum_{j=1}^{r_1} \chi_{A_j}'(g)\Bigr),\\
\sum_{g,h}\chi_W((\theta\cdot(g,h))^3)&=\sum_{g,h}\sum_{i=1}^{r_1 + 2r_2} \chi_{A_i\boxtimes A_i^*}((\theta\cdot(g,h))^3)=\sum_{g,h}\sum_{i=1}^{r_1}\chi_{A_i}'(hghghg)\\
&=|\pi|\sum_g\sum_{i=1}^{r_1}\chi_{A_i}'(g^3).
\end{split} \]
Substituting these into the formula of $d_2(\C\pi)$, we obtain
\begin{equation}\label{eq:ow_formula_tau}
    \begin{split}
 d_2(\C \pi) = \frac{1}{6|\pi|}\sum_g\Bigl\{\Bigl(
\sum_{i=1}^{r_1}\chi_{A_i}'(g)\Bigr)^3+3\Bigl(\sum_{i=1}^{r_1 + 2r_2}|\chi_{A_i}(g)|^2\Bigr)\Bigl(\sum_{j=1}^{r_1}\chi_{A_j}'(g)\Bigr)+2\sum_{i=1}^{r_1}\chi_{A_i}'(g^3)\Bigr\}.\\
\end{split}
\end{equation} 
The corresponding formula in earlier versions of \cite[\S4.3]{OW} mistakenly uses $\chi_{A_j}(g)$ instead of $\chi_{A_j}'(g)$. The formula (\ref{eq:ow_formula_tau}) corrects this error.

To simplify these formulas further, we recall the orthogonality of irreducible characters (cf. \cite[Proposition~7]{Se77}, \cite[Exercise 2.21]{FH}).

\begin{lemma} \label{lem:OIC}
Notations being as above, we have the following.
\begin{equation}
 \chi_W(g, h) =  \sum_{i} \chi_{A_i}(g) \overline{\chi_{A_i}(h)} = \begin{cases}
 \frac{|\pi|}{|C(g)|} & (C(g)=C(h) ), \\
 0 & (C(g) \neq C(h) ),
 \end{cases}
\end{equation}
where $C(g)$ denotes the conjugacy class of $\pi$ represented by $g$.
\end{lemma}
\begin{prop}\label{prop:dim_formulas}
Let $\hat{\pi}$ denote the set of conjugacy classes of $\pi$. We set $\Delta_{\hat{\pi}}^{(3)}\coloneqq \{(C(g),C(h))\in \hat{\pi}\times \hat{\pi}\mid C(g^3)=C(h^3)\}$. Then, the value of $\dim\,\calA_\Theta^\odd(\C\pi)=\frac{1}{2}(d_1(\C\pi)+d_2(\C\pi))$ is computed as
\[ \begin{split}
&d_1(\C\pi)=\frac{1}{6|\pi|}  \biggl(\sum_{C(g) \in \hat{\pi}} \left(\frac{1}{ |C(g)|} |\pi|^2+3 \frac{|C(g)|}{|C(g^2)|} |\pi|\right) + \sum_{(C(g),C(h)) \in \Delta_{\hat{\pi}}^{(3)}} 2 \frac{|C(g)| |C(h)|}{|C(g^3)|}\biggr),\\
&d_2(\C\pi)= \frac{1}{6 |\pi|} \sum_{C(g) \in \hat{\pi}} |C(g)| \Bigl\{\Bigl(
\sum_{i=1}^{r_1}\chi_{A_i}'(g)\Bigr)^3+3\frac{|\pi|}{|C(g)|} \Bigl(\sum_{j=1}^{r_1}\chi_{A_j}'(g)\Bigr)+2\sum_{i=1}^{r_1}\chi_{A_i}'(g^3)\Bigr\}
\end{split} \]
where $\chi_{A_i}'$ denotes $\nu(A_i) \cdot \chi_{A_i}$ for $i=1,\ldots, r_1$. 
\end{prop}
\begin{proof}
Since the orthogonality of irreducible characters imply that $\chi_W(g,h) =0$ for $C(g) \neq C(h)$ and the characters are class functions,  we have
\begin{equation}
    \begin{split}
         d_1(\C \pi) 
         &= \displaystyle\frac{1}{6|\pi|^2}\sum_{C(g) \in \hat{\pi}} \biggl(|C(g)|^2 \biggl(\chi_W(g,g)^3+3 \chi_W(g^2,g^2) \chi_W(g,g)\biggr) +\sum_{C(h) \in \hat{\pi}} 2 |C(g)| |C(h)| \chi_W(g^3,h^3)\biggr)\\
  &= \frac{1}{6|\pi|} \sum_{C(g) \in \hat{\pi}} \biggl(\frac{1}{ |C(g)|} |\pi|^2+3 \frac{|C(g)|}{|C(g^2)|} |\pi| + \sum_{C(h) \in \hat{\pi}; C(h^3)=C(g^3) }2 \frac{|C(g)| |C(h)|}{|C(g^3)|}\biggr).
    \end{split}
\end{equation}
where we set $W = \C \pi$. The formula to $d_2(\C \pi)$ is straightforward by \eqref{eq:ow_formula_tau} and Lemma~\ref{lem:OIC}.
\end{proof}
By Proposition~\ref{prop:dim_formulas}, to compute $\dim\,\calA_\Theta^\odd(\C\pi)=\frac{1}{2}(d_1(\C\pi)+d_2(\C\pi))$, in principle, it suffices to know the following data: the order $|\pi|$ of $\pi$, the character table of $\pi$, the number of elements $|C(g)|$ of each conjugacy class $C(g) \in \hat{\pi}$, and the conjugacy classes $C(g^2)$ and $C(g^3)$ for any $C(g) \in \hat{\pi}$. In this paper, as we will see below, we treat finite groups parametrized by integers such as $\pi= \Z_n, D_{4p}^{\ast}, D_{2^{k+2}p}, T_{8\cdot 3^k}$. Notice that these groups require a case-by-case argument when computing 
\begin{equation}
    \sum_{C(g) \in \hat{\pi}} \frac{|C(g)|}{|C(g^2)|},  \quad \sum_{(C(g),C(h)) \in \Delta_{\hat{\pi}}^{(3)}} \frac{|C(g)| |C(h)|}{|C(g^3)|}, \quad \sum_{i=1}^{r_1}\chi_{A_i}'(g^3),
\end{equation}
in the formulas of Proposition~\ref{prop:dim_formulas} since the transformation rule of $C(g^2)$ and $C(g^3)$ changes depending on whether the parameters $m, p$ are divisible by $2$ and $3$ respectively. 

The following lemma will be used in Lemmas \ref{lem:chi_W_g_d4n_even}, \ref{lem:chi_W_g_d4n_odd}, and \ref{lem:chi_W_g_d2k+2p}.
\begin{lemma}\label{lem:sum_rt_unit}
Let $k$ be an integer. Then, the following identity holds.
\begin{equation}\label{eq:sum_rt_unit}
   D_{n-1}\left(\frac{2\pi k}{2n}\right) -1 =   \sum_{\lambda =1}^{n-1}2 \cos\left(\frac{2\pi k\lambda}{2n} \right)  =\begin{cases}
        2n -2 & (k \equiv 0 \bmod 2n ),\\
        -2 & (k\equiv 0 \bmod 2,  k\not \equiv  0 \bmod 2n),\\
        0 & (\text{otherwise}),
    \end{cases}
\end{equation}
where $D_n(x) = \displaystyle\sum_{\lambda=-n}^n e^{i\lambda x}  = \frac{\sin ((n+1/2)x)}{\sin(x/2)}$ (Dirichlet kernel, e.g. \cite[\S{8.13}]{Ru}). Similarly, 
\begin{equation}\label{eq:sum_alt_rt_unit}
   \widetilde{D}_{n-1}\left(\frac{2\pi k}{2n}\right) -1 =   \sum_{\lambda =1}^{n-1} (-1)^{\lambda} 2 \cos\left(\frac{2\pi k\lambda}{2n} \right)  =\begin{cases}
   		2n -2 & (k \equiv n \bmod 2n),\\
        -2 & (k\equiv n \bmod 2, k \not \equiv n \bmod 2n),\\
        0 & (\text{otherwise}),
    \end{cases}
\end{equation}
where $\widetilde{D}_n(x) = D_n(x - \pi) =   \displaystyle\sum_{\lambda = -n}^n (-1)^{\lambda} e^{i \lambda x} = (-1)^{n} \frac{\cos ((n+1/2)x)}{\cos(x/2)}$.
\end{lemma}

\subsection{Cyclic group $\Z_n$}

To begin with, we recall the dimension formula for $\pi=\Z_n$ given in \cite{OW} and then provide an alternative computation to it by using the formula presented in the last section. We will find that this reformulation gives an explicit formula to $p_3(n)$, the number of partitions of an integer $n$ into at most $3$. The computations in this section also serve as a preliminary warm-up for those in the subsequent sections.

Let $n$ be a positive integer. The \textit{cyclic group} $\Z_n$ of order $n$ is given by the finitely presented group $\langle \gamma \mid \gamma^n = 1\rangle$.
When $\pi = \Z_n$, the second author and Ohta give the following dimension formula.
\begin{prop}[{\cite[Proposition~1.3]{OW}}] \label{prop:OW_dim_cyclic}
	When $\pi = \Z_n$, we have the following.
	\begin{equation}
		\begin{split}
		\dim \calA_\Theta^\odd(\C\pi)	 = p_3(n), \quad \dim \calA_\Theta^\odd(\mathrm{Ker}\,\ve) = p_3(n-3).
		\end{split}
	\end{equation}
	Here, for an integer $m \geq 0$, $p_3(m)$ denotes the number of partitions of $m$ into at most three parts, namely, the number of integer solutions of the equation $x + y +z = m$ $(0 \leq x \leq y \leq z )$, and we set $p_3(m) =0$ for $m<0$.
\end{prop}
The proof given in \cite{OW} is based on a combination of a direct examination of $(\Sym^3 \C\pi)^{(\pi \times \pi) \rtimes \Z_2}$ and a geometric argument. In the following, we revisit the computation from the viewpoint of representation theory via the formula in Section \ref{section:4.1.1}.

\subsubsection{Conjugacy classes of $\Z_n$}

Since $\Z_n$ is abelian, the set of conjugacy classes of $\pi = \Z_n$, denoted by $\hat{\pi}$,  consists only of singletons $[\gamma^m]=\{\gamma^m\}$ $(m=0,1,\ldots, n-1)$ and is equal to $\Z_n$ as sets. 
By examining $\Z_2$ action on $\hat{\pi}$, one obtains the following lemma immediately. The dimension $\dim(\C\hat{\pi})_{\Z_2}$ is necessary to compute $\dim \calA_\Theta^\odd(\ker\ve)$ by using Proposition~\ref{prop:graph-inv}-2.
\begin{lemma}\label{lem:conj_z2_zn}
Let $\pi = \Z_n$. When $n$ is odd we have $[x^{-1}]=[x]$ only for $x=e$, and when $n$ is even we have $[x^{-1}]=[x]$ only for $x=e, \gamma^{n/2}$. In particular, the following holds.
\begin{equation}
	\dim(\C\hat{\pi})_{\Z_2}= 1 + \floor{\frac{n}{2}} =  \begin{cases}
 	1 + \frac{n}{2} & ( n \equiv 0 \bmod 2),\\
 	1 + \frac{n-1}{2} & (n \not \equiv 0 \bmod 2),
 \end{cases}
\end{equation}
where $\floor{\frac{n}{2}}$ denotes the greatest integer less than or equal to $\frac{n}{2}$.
\end{lemma}

\begin{remark}
Note that  $\dim(\C\hat{\pi})_{\Z_2} = 1 + \floor{\frac{n}{2}}$ is the number of partitions of $n$ into at most $2$ parts, we denote it by $p_2(n)$ (cf. the sequence \href{https://oeis.org/A008619}{A008619} in OEIS (\cite{oeis})). Therefore, Lemma~\ref{lem:conj_z2_zn} is just a restatement, in a more explicit manner, of the same statement given in the proof of \cite[Proposition~5.3]{OW}.
\end{remark}

We prepare the following lemma for computation of $d_1(\C\pi) =\dim\,(\Sym^3 \C\pi)^{\pi\times \pi}$. For the proof, see Appendix \ref{appendix:A}.

\begin{lemma}\label{lem:zn_dim1_chiw_cubic}
	When $\pi = \Z_n$, the following holds.
    \begin{equation}
       \sum_{(C(g),C(h)) \in \Delta_{\hat{\pi}}^{(3)}} \frac{|C(g)| |C(h)|}{|C(g^3)|} 
        = \begin{cases}
            3n & (n \equiv 0 \bmod 3),\\
            n & (n \not \equiv 0 \bmod 3).
        \end{cases}
    \end{equation}
\end{lemma}

By applying Lemma~\ref{lem:zn_dim1_chiw_cubic} to the formula in Proposition~\ref{prop:dim_formulas}, one can compute $d_1(\C \pi)$ as follows.

\begin{lemma}\label{lem:dim1_zn}
For $\pi = \Z_n$, we have
    $d_1(\C\pi) = \begin{cases}
        \frac{1}{6} n^2 + \frac{1}{2}n + 1 & (n \equiv 0 \bmod 3),\\
        \frac{1}{6} n^2 + \frac{1}{2}n + \frac{1}{3} & (n \not \equiv 0 \bmod 3).
    \end{cases}$
\end{lemma}

\subsubsection{Sum of real characters for $\Z_n$}

Next, we compute $d_2(\C\pi)$. As is well known, the unique irreducible representations $\rho_{V_{\lambda}}$ of $\Z_n$ are given by one dimensional representations
\begin{equation}
	\rho_{V_{\lambda}}(\gamma) = \zeta_n^{\lambda} \quad (\lambda=0,1,2,\ldots, n-1),
\end{equation}
where we set $\zeta_n = \exp\left(\frac{2\pi \sqrt{-1}}{n}\right) \in \C$. The character table of $\Z_n$ is given as Tables~\ref{tab:ch_real_zn_odd} and \ref{tab:ch_real_zn_even}.

\begin{table}[h]
\centering
  \begin{minipage}[t]{.45\textwidth}
    \centering
      \begin{tabular}{|c|c|}  \hline
    $\hat{\pi}$ & $\gamma^m$\\ 
    size & $1$ \\ \hline
    $V_0$ & $1$  \\ 
    $V_{\lambda}$ & $\zeta_n^{\lambda m}$  \\ \hline
  \end{tabular}
    \caption{The characters for $\Z_n$ with odd $n$. Here, $\lambda=1,2,\ldots, n-1$, for which $\chi_{V_\lambda}$ is non-real.}
    \label{tab:ch_real_zn_odd}
  \end{minipage}
\hspace{1cm} 
  \begin{minipage}[t]{.45\textwidth}
    \centering
      \begin{tabular}{|c|c|}  \hline
    $\hat{\pi}$ & $\gamma^m$\\ 
    size & $1$ \\ \hline
    $V_0$ & $1$  \\ 
    $V_{n/2}$ & $(-1)^m$ \\ 
    $V_{\lambda}$ & $\zeta_n^{\lambda m}$  \\ \hline
  \end{tabular}
    \caption{The characters for $\Z_n$ with even $n$. Here, $\lambda=1,2,\ldots,n/2-1,n/2+1,\ldots, n-1$, for which $\chi_{V_\lambda}$ is non-real.}
    \label{tab:ch_real_zn_even}
  \end{minipage}
\end{table}

\begin{lemma}\label{lem:chi_W_g_zn}
\renewcommand{\labelenumi}{$(\arabic{enumi})$}
Let $\pi = \Z_n$. For  $\gamma^m \in \Z_n$, we have
        \begin{equation}
                \sum_{\chi_{A_i} : \text{real}} \chi_{A_i}(\gamma^m) = \sum_{\chi_{A_i}: \text{real}}\chi_{A_i}(\gamma^{3m})=  \begin{cases}
                    1 & (n \not \equiv 0 \bmod 2),\\
                    2 & (n \equiv 0 \bmod 2, m \equiv 0 \bmod 2),\\
                    0 & (n \equiv 0 \bmod 2, m \not \equiv 0 \bmod 2).
                \end{cases}
        \end{equation}
\end{lemma}

\begin{proof}
    The assertion follows straightforwardly from Tables~\ref{tab:ch_real_zn_odd} and \ref{tab:ch_real_zn_even}.
\end{proof}

By applying Lemma~\ref{lem:chi_W_g_zn} to the formula in Proposition~\ref{prop:dim_formulas}, we obtain the following. 
\begin{lemma}\label{lem:dim2_zn}
For $\pi = \Z_n$, we have
    $d_2(\C\pi) = \begin{cases}
    \frac{1}{2} n + 1 & (n \equiv 0 \bmod 2),\\
         \frac{1}{2}n + \frac{1}{2} & (n \not \equiv 0 \bmod 2).
    \end{cases}$
\end{lemma}

\begin{prop}\label{prop:dim_zn}
    For $\pi = \Z_n$, we have
    \begin{align}
	&\dim \calA_\Theta^\odd(\C\pi)= \begin{cases}
        \frac{1}{12} n^2 + \frac{1}{2}n + 1 & (n \equiv 0 \bmod 2, n \equiv 0 \bmod 3),\\
        \frac{1}{12} n^2 + \frac{1}{2}n + \frac{2}{3} & (n \equiv 0 \bmod 2, n \not \equiv 0 \bmod 3),\\
         \frac{1}{12} n^2 + \frac{1}{2}n + \frac{3}{4} & (n \not \equiv 0 \bmod 2, n \equiv 0 \bmod 3),\\
          \frac{1}{12} n^2 + \frac{1}{2}n + \frac{5}{12} & (n \not \equiv 0 \bmod 2, n \not \equiv 0 \bmod 3),
    \end{cases}\label{eq:p_3(n)}\\
	&\dim \calA_\Theta^\odd(\ker\ve)= \begin{cases}
        \frac{1}{12} n^2  & (n \equiv 0 \bmod 2, n \equiv 0 \bmod 3),\\
        \frac{1}{12} n^2 - \frac{1}{3} & (n \equiv 0 \bmod 2, n \not \equiv 0 \bmod 3),\\
         \frac{1}{12} n^2 + \frac{1}{4} & (n \not \equiv 0 \bmod 2, n \equiv 0 \bmod 3),\\
          \frac{1}{12} n^2 - \frac{1}{12} & (n \not \equiv 0 \bmod 2, n \not \equiv 0 \bmod 3).
    \end{cases}\label{eq:p_3(n-3)}
\end{align}
\end{prop}

\begin{proof}
   The first assertion is a consequence of the formula in Proposition~\ref{prop:dim_formulas}, Lemma~\ref{lem:dim1_zn}, and Lemma~\ref{lem:dim2_zn}. For $\dim \calA_\Theta^\odd(\ker\ve)$, we use  Proposition~\ref{prop:graph-inv}-2  and  Lemma~\ref{lem:conj_z2_zn}.
\end{proof}

\begin{remark}
   The formulas \eqref{eq:p_3(n)} and \eqref{eq:p_3(n-3)} obtained in Proposition~\ref{prop:dim_zn} are known as those for $p_3(n)$ and $p_3(n-3)$ respectively (for example, see the sequence \href{https://oeis.org/A001399}{A001399} in OEIS (\cite{oeis})). Thus, we have given an alternative proof to Proposition~\ref{prop:OW_dim_cyclic} (\cite[Proposition~1.3]{OW}) from the viewpoint of character theory.
\end{remark}

\subsection{Binary dihedral group $D_{4p}^{\ast}$ with even $p$} \label{section:D4n_even}
Let $p$ be a positive integer. The \textit{binary dihedral group} $D_{4p}^{\ast}$ of order $4p$ admits the following finite presentation
\begin{equation}
    D_{4p}^{\ast} = \langle a, x \mid a^{2p}=1, x^2 = a^p, x^{-1} ax = a^{-1} \rangle.
\end{equation}
The group $D_{4p}^{\ast}$ gives rise to a \textit{prism manifold} $S^3/D_{4p}^{\ast}$.
Any element $g \in D_{4p}^{\ast}$ is uniquely written as $g = a^{k} x^{l}$ for some $0 \leq k \leq 2p-1$ and $0 \leq l \leq 1$. In the case that $\pi =  D_{4p}^{\ast}$, depending on whether $p$ is even or odd, there is a little difference in the computation of dimensions of $\calA_\Theta^\odd(\C\pi)$ and $\calA_\Theta^\odd(\mathrm{Ker}\,\ve)$. Thus, in what follows, we treat the cases of $p$ being even and odd separately.

The character table of $D_{4p}^\ast$ for $p$ even is given in Table~\ref{tab:ch_d4n_even}, which can be obtained from Proposition~\ref{prop:irrep_D4p_even}.

\begin{table}[h]
\centering
\renewcommand{\arraystretch}{1.2}
  \begin{tabular}{|c|c|c|c|c|c||c|}  \hline
    $\hat{\pi}$ & $C_{0,0}$ & $C_{k, 0}$ & $C_{p, 0}$ & $C_{\mathrm{even}, 1}$ & $C_{\mathrm{odd}, 1}$ & Frobenius--Schur\\ 
    size & $1$ & $2$ & $1$ & $p$ &$p$ & indicator \\ \hline
    $V^{(1)}_1$ & $1$ & $1$ & $1$ & $1$ & $1$ & $1$ \\
    $V^{(1)}_2$ & $1$ & $1$ & $1$ & $-1$ & $-1$ & $1$\\
    $V^{(1)}_3$ & $1$ & $(-1)^k$ & $1$ & $1$ & $-1$ &$1$ \\
     $V^{(1)}_4$ & $1$ & $(-1)^k$ & $1$ & $-1$ & $1$ &$1$\\
     $V_{\lambda}^{(2)}$ & $2$ & $2 \cos \left(\frac{2\pi k \lambda}{2p} \right)$ & $(-1)^{\lambda} \cdot 2$ & $0$ & $0$ & $(-1)^{\lambda}$\\ \hline
  \end{tabular}
\caption{The characters $\chi_{V^{(i)}_{\lambda}}(g)$ for $D_{4p}^{\ast}$ with even $p$. Here, $k=1,\ldots, p-1$, and $\lambda=1,\ldots, p-1$. }\label{tab:ch_d4n_even}
\end{table}

\subsubsection{Conjugacy classes of $D_{4p}^{\ast}$ with even $p$}\label{ss:conj_D4p}

For $\pi = D_{4p}^{\ast}$, we denote by $\hat{\pi}$ the set of conjugacy classes of $\pi$. Then, $|\hat{\pi}| = p+3$. 
More concretely, $\hat{\pi}$ consists of the following conjugacy classes:
\[\begin{split}
&C_{0,0} = [e] = \{e\},\quad C_{k, 0} = C_{2p-k,0} = [a^{k}] = \{ a^{k}, a^{2p-k}\} \quad (1 \leq k \leq p-1),\quad C_{p, 0} = [a^{p}] = \{ a^{p}\},\\
&C_{\mathrm{even}, 1} = [x]=\{x, a^2x, a^4x, \ldots, a^{2p-2}x\},\quad C_{\mathrm{odd}, 1} = [ax] = \{ax, a^3x, \ldots, a^{2p-1}x\}.
\end{split}\]
Here, the conjugacy classes of $\pi$ can be directly obtained by examining the values of irreducible characters.

\begin{lemma}\label{lem:conj_inv_D_4n_even}
Let $\pi = D_{4p}^{\ast}$ with $p$ even. For each class $[x] \in \hat{\pi}$, we have $[x^{-1}] = [x] \in \hat{\pi}$. In particular, $\dim(\C\hat{\pi})_{\Z_2}=\dim \C\hat{\pi}=p+3$.
\end{lemma}
\begin{proof}
	For $a^k$ and $a^k x$, their inverse elements are $a^{-k}$ and $x^{-1} a^{-k} = x^{-1} a^{2p-k} = a^{k -2p} x^{-1} = a^{k-p} x= a^{k+p} x$ respectively. As $x a^k x^{-1} = a^{-k}$, $[a^k] =[a^{-k}] \in \hat{\pi}$.  Since $p$ is even and positive,   $a (a^k x) a^{-1} = a^{k+2} x$ implies $[a^k x] =[a^{k+2} x] =\cdots =[a^{k+p}x] = [(a^kx)^{-1}] \in \hat{\pi}$.
\end{proof}

\begin{lemma}\label{lem:sq_cub_conj_cls_d4n_even}
    Let $\pi = D_{4p}^{\ast}$ with even $p$. For any conjugacy class $[g] = C_{i, j} \in \hat{\pi}$ we denote by $C^l_{i,j} \in \hat{\pi}$ the conjugacy class represented by $g^l$ for any integer $l$. Then, we have 
    \begin{equation}
    \begin{split}
            & C_{0,0}^2 = C_{0,0}, \quad C_{k,0}^2 = C_{2k,0},\quad C_{p,0}^2 = C_{0,0}, \quad C_{\mathrm{even},1}^2 = C_{p, 0}, \quad C_{\mathrm{odd},1}^2 = C_{p, 0},\\
        &C_{0,0}^3 = C_{0,0}, \quad C_{k,0}^3 = C_{3k,0},\quad C_{p,0}^3 = C_{p,0}, \quad C_{\mathrm{even},1}^3 = C_{\mathrm{even},1}, \quad C_{\mathrm{odd},1}^3 = C_{\mathrm{odd}, 1}.
        \end{split}
    \end{equation}
Here, $2k$ and $3k$ in the subscripts of $C_{2k,0}$ and $C_{3k,0}$ respectively are understood as the least non-negative remainders modulo $2p$.
\end{lemma}

\begin{proof}
The assertions follow from the finite presentation of $D_{4p}^{\ast}$. Namely, the equalities $C_{0,0}^2 = C_{0,0}$, $C_{k,0}^2 = C_{2k,0}$, $C_{0,0}^3 = C_{0,0}$ and $C_{k,0}^3 = C_{3k,0}$ are clear, and we have
\[ \begin{split}
    &C_{p,0}^2 = [(a^p)^2] =[a^{2p}]=[e] = C_{0,0},\qquad 
    C_{\mathrm{even}, 1}^2 = [x^2] = [a^p] = C_{p,0},\\
    &C_{\mathrm{odd}, 1}^2 = [(ax)^2] = [axax]=[x(x^{-1}ax)ax]=[xa^{-1}ax]=[x^2]=[a^p] = C_{p,0},\\
    &C_{p,0}^3 = [a^{3p}] = [a^p \cdot a^{2p}] = [a^p\cdot e]=[a^p] = C_{p,0},\\
    &C_{\mathrm{even},1}^3 = [x^3]= [x^2 \cdot x] = [a^p \cdot x] = C_{\mathrm{even}, 1},\\
    &C_{\mathrm{odd}, 1}^3 = [(ax)^3]=[(ax)^2 \cdot ax]=[a^p \cdot ax] = [a^{p+1} x] = C_{\mathrm{odd},1},
\end{split} \]
since $p$ is even and hence $p+1$ is odd.
\end{proof}

\begin{remark}\label{rem:d4n_even_c^2}
Let $k$ be an integer with $1\leq k\leq p-1$. Since $p$ is even, when $k=p/2$, we have $C_{2k,0} = C_{p,0}$ but otherwise we have $C_{2k,0}=C_{k'0}$ for some $k'$ with $1\leq k' \leq p-1$. Similarly, when $p \equiv 0 \bmod 3$, $C_{3k,0} = C_{0,0}$ or $C_{3k,0} = C_{p,0}$ hold when $k=2p/3$ or $k=p/3$ but otherwise $C_{3k,0} = C_{k',0}$ holds for some $k'$ with $1\leq k' \leq p-1$.
\end{remark}

To compute $\dim \calA_\Theta^\odd(\C\pi)$ we apply the formulas in Section \ref{section:4.1.1}. Related to Remark \ref{rem:d4n_even_c^2}, it is useful to prepare the following lemma, which will be used in the computation of $d_1(\C\pi)= \dim\,(\Sym^3 \C\pi)^{\pi\times \pi}$.
\begin{lemma}\label{lem:d4n_even_dim1_chiw_cubic}
Let $\pi = D_{4p}^{\ast}$ with even $p$. Then, we have the following.
\begin{equation}
    \sum_{(C(g),C(h)) \in \Delta_{\hat{\pi}}^{(3)}} \frac{|C(g)| |C(h)|}{|C(g^3)|} =  \begin{cases}
        8 p & (p \equiv 0 \bmod 3),\\
        4 p & (p \not \equiv 0 \bmod 3).
    \end{cases}
\end{equation}
\end{lemma}

\begin{proof}
  The proof is given in Appendix \ref{appendix:A}.
\end{proof}

\begin{lemma}\label{lem:dim1_D4n_even}
    When $\pi = D_{4p}^{\ast}$ with even $p$, we have
    $d_1(\C \pi)  = \begin{cases}
        \frac{1}{3} p^2 + \frac{5}{2}p + 3 & (p \equiv 0 \bmod 3),\\
        \frac{1}{3} p^2 + \frac{5}{2}p + \frac{8}{3} & (p \not \equiv 0 \bmod 3).
    \end{cases}$
\end{lemma}

\begin{proof}

By Remark \ref{rem:d4n_even_c^2}, we know that if $k=0,p/2$ then $|C_{2k,0}|=1$; otherwise $|C_{2k,0}|=2$. Taking care of this fact, one obtains the following from Lemma~\ref{lem:sq_cub_conj_cls_d4n_even} and 
Lemma~\ref{lem:d4n_even_dim1_chiw_cubic}:
\begin{equation}
\begin{split}
	&\frac{1}{6} \sum_{C(g) \in \hat{\pi}} \left( \frac{1}{|C(g)|} |\pi| + 3 \frac{|C (g) |}{|C(g^2)|} \right) = \frac{1}{3}p^2 + \frac{5}{2}p + \frac{7}{3},\\
	&\frac{1}{6|\pi|}  \sum_{(C(g), C(h)) \in \Delta_{\hat{\pi}}^{(3)}} 2 \frac{|C(g)| |C(h)|}{|C(g^3)|} = \begin{cases}
 	\frac{2}{3} & (p \equiv 0 \bmod 3),\\
 	\frac{1}{3}& (p \not \equiv 0 \bmod 3).
 \end{cases}
 \end{split}
\end{equation}
Substituting these two identities into the formula in Proposition~\ref{prop:dim_formulas} proves the assertion.
\end{proof}

\subsubsection{Sum of real characters for $D_{4p}^\ast$ with even $p$}

Next we compute $d_2(\C\pi)$.
Notice that  Table~\ref{tab:ch_d4n_even} implies that all distinct irreducible representations of $D_{4p}^{\ast}$ have real-valued character. Thus, we can apply the same formula as in \cite[\S 4.3]{OW} to compute $\dim \calA_\Theta^\odd(\C\pi)$. 

\begin{lemma}\label{lem:chi_W_g_d4n_even}
Let $\pi = D_{4p}^{\ast}$ with even $p$.
    For  $g \in D_{4p}^{\ast}$, we have the following.
        \begin{equation}
        \begin{split}
                &1.\quad    \sum_i \chi_{A_i}'(g) = \begin{cases}
                    2p + 2 & ([g]=C_{p,0}),\\
                    2  & ([g] = C_{k,0} \quad (k=0,2,4,\ldots, p-2)),\\
                    0 & (\text{otherwise}).
                \end{cases}\\
           &2. \quad  \sum_i \chi_{A_i}'(g^3)  = \begin{cases}
                    2p +2 & ([g] = C_{p,0}),\\
                    2p + 2 & ([g] = C_{p/3, 0}, p \equiv 0 \bmod 3),\\
                    2 & ([g] = C_{k,0}, p \equiv 0 \bmod 3\  (k=0,2,4,\ldots, p-2, k\neq p/3)),\\
                    2 & ([g] = C_{k,0}, p \not \equiv 0\bmod 3\  (k=0,2,4,\ldots, p-2)),\\
                    0 & (\text{otherwise}),
                \end{cases}
            \end{split}
        \end{equation}
        where we set $W= \C \pi$.
\end{lemma}

\begin{proof}
    The character table of $D_{4p}^{\ast}$ (Table~\ref{tab:ch_d4n_even}) implies that, for $0\leq k \leq p$ and $0\leq l \leq 1$, 
    \begin{equation}
        \sum_{i=1}^{p+3} \chi_{A_i}'(a^k x^l) = 1 +(-1)^l + (-1)^k + (-1)^{k+l} + \delta_{l,0} \sum_{\lambda=1}^{p-1} (-1)^{\lambda} 2 \cos \left(\frac{2\pi k\lambda}{2p} \right),
    \end{equation}
    where $\delta_{l,0}$ denotes the Kronecker delta.
    Then, the assertion 1  follows from these equations and Lemma~\ref{lem:sum_rt_unit}.
    The assertion 2 follows from the assertion 1 and Lemma~\ref{lem:sq_cub_conj_cls_d4n_even}.
\end{proof}

 By applying Lemma~\ref{lem:chi_W_g_d4n_even} to the formula in Proposition~\ref{prop:dim_formulas}, one directly gets the following.

\begin{lemma}\label{lem:dim2_D4n_even}
For $\pi = D_{4p}^{\ast}$ with even $p$, we have
$d_2(\C\pi) = \begin{cases}
        \frac{1}{3} p^2 + \frac{5}{2}p + 3 & (p \equiv 0 \bmod 3),\\
        \frac{1}{3} p^2 + \frac{5}{2}p + \frac{8}{3} & (p \not \equiv 0 \bmod 3).
    \end{cases}$
\end{lemma}

By Proposition~\ref{prop:graph-inv}-2, Lemma~\ref{lem:dim1_D4n_even}, Lemma~\ref{lem:dim2_D4n_even}, and Proposition~\ref{prop:dim_formulas}, the values of $\calA_\Theta^\odd(\C\pi)$ and $\calA_\Theta^\odd(\mathrm{Ker}\,\ve)$ for $\pi=D_{4p}^{\ast}$ with even $p$ is determined as follows.
\begin{prop}
    For $\pi = D_{4p}^{\ast}$ with even $p$, we have
    \begin{equation}
    \begin{split}
	&\dim \calA_\Theta^\odd(\C\pi)= \begin{cases}
        \frac{1}{3} p^2 + \frac{5}{2}p + 3 & (p \equiv 0 \bmod 3),\\
        \frac{1}{3} p^2 + \frac{5}{2}p + \frac{8}{3} & (p \not \equiv 0 \bmod 3),
    \end{cases}\\
	&\dim \calA_\Theta^\odd(\ker \ve)= \begin{cases}
        \frac{1}{3} p^2 + \frac{3}{2}p  & (p \equiv 0 \bmod 3),\\
        \frac{1}{3} p^2 + \frac{3}{2}p - \frac{1}{3} & (p \not \equiv 0 \bmod 3).
    \end{cases}
    \end{split}
\end{equation}
\end{prop}

\subsection{Binary dihedral group $D_{4p}^{\ast}$ with odd $p$} \label{section:D4n_odd}

The character table is given in Table~\ref{tab:ch_d4n_odd}, which can be obtained from Proposition~\ref{prop:irrep_D4p_odd}. 

\begin{table}[h]
\centering
\renewcommand{\arraystretch}{1.2}
  \begin{tabular}{|c|c|c|c|c|c||c|}  \hline
    $\hat{\pi}$ & $C_{0,0}$ & $C_{k, 0}$ & $C_{p, 0}$ & $C_{\mathrm{even}, 1}$ & $C_{\mathrm{odd}, 1}$ & Frobenius--Schur\\ 
    size & $1$ & $2$ & $1$ & $p$ &$p$ & indicator\\ \hline
    $V^{(1)}_1$ & $1$ & $1$ & $1$ & $1$ & $1$ &$1$\\
    $V^{(1)}_2$ & $1$ & $1$ & $1$ & $-1$ & $-1$ & $1$ \\
    $V^{(1)}_3$ & $1$ & $(-1)^k$ & $-1$ & $\sqrt{-1}$ & $-\sqrt{-1}$ &$0$ \\
     $V^{(1)}_4$ & $1$ & $(-1)^k$ & $-1$ & $-\sqrt{-1}$ & $\sqrt{-1}$ &$0$\\
     $V_{\lambda}^{(2)}$ & $2$ & $2 \cos \left(\frac{2\pi k \lambda}{2p} \right)$ & $(-1)^{\lambda} \cdot 2$ & $0$ & $0$ & $(-1)^{\lambda}$ \\ \hline
  \end{tabular}
\caption{The characters $\chi_{V^{(i)}_{\lambda}}(g)$ for $D_{4p}^{\ast}$ with odd $p$. Here, $k=1,\ldots, p-1$, and $\lambda=1,\ldots, p-1$. Note that the characters are non-real for $V_3^{(1)}$ and $V_4^{(2)}$.}\label{tab:ch_d4n_odd}
\end{table}

\subsubsection{Conjugacy classes of $D_{4p}^{\ast}$ with odd $p$}\label{ss:conj_D4p_odd}

For $\pi = D_{4p}^{\ast}$, we denote by $\hat{\pi}$ the set of conjugacy classes of $\pi$. Then, $|\hat{\pi}| = p+3$ since the number of equivalence classes of irreducible representations and one of the conjugacy classes of $\pi$ are the same. More concretely, $\hat{\pi}$ consists of the following conjugacy classes:
\[\begin{split}
&C_{0,0} = [e] = \{e\},\quad C_{k, 0} = C_{2p-k,0}= [a^{k}] = \{ a^{k}, a^{2p-k}\} \quad (1 \leq k \leq p-1),\quad C_{p, 0} = [a^{p}] = \{ a^{p}\},\\
&C_{\mathrm{even}, 1} = [x]=\{x, a^2x, a^4x, \ldots, a^{2p-2}x\},\quad 
C_{\mathrm{odd}, 1} = [ax] = \{ax, a^3x, \ldots, a^{2p-1}x\}.
\end{split}\]

\begin{lemma}\label{lem:conj_z2_D4n_odd}
Let $\pi = D_{4p}^{\ast}$ with odd $p$. Then, $[x^{-1}]=[x] \in \hat{\pi}$ for $x = a^k$ but $[x^{-1}] \neq [x] \in \hat{\pi}$ for $x = a^k x$. In particular, $\dim(\C\hat{\pi})_{\Z_2}=p+2$
\end{lemma}
\begin{proof}
	For $a^k$ and $a^k x$, their inverse elements are $a^{-k}$ and $x^{-1} a^{-k} = x^{-1} a^{2p-k} = a^{k -2p} x^{-1} = a^{k-p} x= a^{k+p} x$ respectively. Since $x a^k x^{-1} = a^{-k}$, $[a^k] =[a^{-k}] \in \hat{\pi}$. On the other hand, $a^k x$ and $a^{k+p}x$ have different parities in the power of $a$ since $p$ is odd. Therefore, $[(a^k x)^{-1}] \neq [a^k x] \in \hat{\pi}$. In other words, $\Z_2$ acts trivially on $C_{0,0}$, $C_{k,0}$ and $C_{p,0}$, whereas $\Z_2$ acts non trivially on $C_{\mathrm{even},1}, C_{\mathrm{odd},1}$ so that the $\Z_2$ orbit of them is given by $\{C_{\mathrm{even},1}, C_{\mathrm{odd},1}\}$. Hence, the dimension of the quotient space $(\C\hat{\pi})_{\Z_2}$ is calculated as $p+1 + 2\cdot \frac{1}{2} = p +2$.
\end{proof}

\begin{lemma}\label{lem:sq_cub_conj_cls_d4n_odd}
    Let $\pi = D_{4p}^{\ast}$ with odd $p$. For any conjugacy class $[g] = C_{i, j} \in \hat{\pi}$ we denote by $C^l_{i,j} \in \hat{\pi}$ the conjugacy class represented by $g^l$ for any integer $l$. Then, we have 
    \begin{equation}
    \begin{split}
            & C_{0,0}^2 = C_{0,0}, \quad C_{k,0}^2 = C_{2k,0},\quad C_{p,0}^2 = C_{0,0}, \quad C_{\mathrm{even},1}^2 = C_{p, 0}, \quad C_{\mathrm{odd},1}^2 = C_{p, 0},\\
            & C_{0,0}^3 = C_{0,0}, \quad C_{k,0}^3 = C_{3k,0},\quad C_{p,0}^3 = C_{p,0}, \quad C_{\mathrm{even},1}^3 = C_{\mathrm{odd},1}, \quad C_{\mathrm{odd},1}^3 = C_{\mathrm{even}, 1}.
    \end{split}
    \end{equation}
Here, $2k$ and $3k$ in the subscripts of $C_{2k,0}$ and $C_{3k,0}$ respectively are understood as the least non-negative remainders modulo $2p$.
\end{lemma}

\begin{proof}
Most part follows the same way as Lemma~\ref{lem:sq_cub_conj_cls_d4n_even}. It is enough to only show that $C_{\mathrm{even},1}^3 = C_{\mathrm{odd},1}$ and $C_{\mathrm{odd},1}^3 = C_{\mathrm{even}, 1}$. We have
\[ \begin{split}
    &C_{\mathrm{even},1}^3 = [x^3]= [x^2 \cdot x] = [a^p \cdot x] = C_{\mathrm{odd}, 1},\\
    &C_{\mathrm{odd}, 1}^3 = [(ax)^3]=[(ax)^2 \cdot ax]=[a^p \cdot ax] = [a^{p+1} x] = C_{\mathrm{even},1},
\end{split}\]
since $p$ is odd and hence $p+1$ is even.
\end{proof}

\begin{remark}\label{rem:d4n_odd_c^2}
Let $k$ be an integer with $1\leq k\leq p-1$. Since $p$ is odd, we have $C_{2k,0}=C_{k'0}$ for some $k'$ with $1\leq k' \leq p-1$. It differs from the case of $D_{4p}^{\ast}$ with even $p$ (cf. Remark \ref{rem:d4n_even_c^2}). When $p \equiv 0 \bmod 3$, $C_{3k,0} = C_{0,0}$ or $C_{3k,0} = C_{p,0}$ hold when $k=2p/3$ or $k=p/3$ respectively, but otherwise $C_{3k,0} = C_{k',0}$ holds for some $k'$ with $1\leq k' \leq p-1$.
\end{remark}

\begin{lemma}\label{lem:dim1_D4n_odd}
When $\pi=D_{4p}^*$ with odd $p$, we have
$    d_1(\C \pi)=\begin{cases}
        \frac{1}{3} p^2 + \frac{5}{2} p + \frac{5}{2} & (p \equiv 0 \bmod 3),\\
        \frac{1}{3}p^2 + \frac{5}{2}p + \frac{13}{6} & (p \not \equiv 0 \bmod 3).
    \end{cases}$
\end{lemma}

\begin{proof}
Notice that the statements in Lemma~\ref{lem:d4n_even_dim1_chiw_cubic} for $D_{4p}^{\ast}$ with even $p$ are also valid for odd $p$, although there is a little difference in the irreducible representations between them. Therefore, the computation itself goes along the same way as in the proof of Lemma~\ref{lem:dim1_D4n_even}. In fact, by Remark \ref{rem:d4n_odd_c^2}, for any $k$ with $1 \leq k \leq p-1$ we have $C_{k,0}^2 = C_{k'0}$ for some $k'$ with $1 \leq k' \leq p-1$, which makes the computation much simpler than the previous case. The computation is straightforward and so we omit the details.
\end{proof}

\subsubsection{Sum of real characters for $D_{4p}^\ast$ with odd $p$}

Next we compute $d_2(\C\pi)$. 
Notice that, different from the case of $p$ even, some irreducible representations have non-real valued characters when $p$ is odd. 

From Table~\ref{tab:ch_d4n_odd}, we know that some irreducible representations of $D_{4p}^{\ast}$ have non-real valued characters. Thus, we need the generalized formula considered in Proposition~\ref{prop:dim_formulas} to compute $\dim \calA_\Theta^\odd(\C\pi)$ instead of the original one given in \cite{OW}. 
Aimed at the computation of $d_2(\C\pi)$, we prepare a lemma that corresponds to Lemma~\ref{lem:chi_W_g_d4n_even}.

\begin{lemma}\label{lem:chi_W_g_d4n_odd}
Let $\pi = D_{4p}^{\ast}$ with odd $p$. For  $g \in D_{4p}^{\ast}$, we have the following.
        \begin{equation}
        \begin{split}
              &1. \quad  \sum_{\chi_{A_i}: \text{real}} \chi_{A_i}'(g) = \begin{cases}
                    2p & ([g]=C_{p,0}),\\
                    2  & ([g] = C_{k,0}\  (k=0,2,4,\ldots, p-1)),\\
                    0 & (\text{otherwise}).
                    \end{cases}\\
            &2. \quad \sum_{\chi_{A_i}: \text{real}}\chi_{A_i}'(g^3)  = \begin{cases}
                    2p & ([g] = C_{p,0}),\\
                    2p & ([g] = C_{p/3,0}, p \equiv 0 \bmod 3),\\
                    2 & ([g] = C_{k,0} (k=0,2,4\ldots, p-1)),\\
                    0 & (\text{otherwise}).
                \end{cases}
            \end{split}
        \end{equation}
\end{lemma}

\begin{proof}
    For $0\leq k \leq p$ and $0\leq l \leq 1$, we have 
        $\sum_{\chi_{A_i}: \text{real}} \chi_{A_i}'(a^k x^l) = 1 +(-1)^l + \delta_{l,0} \sum_{\lambda=1}^{p-1}(-1)^{\lambda} 2 \cos \left(\frac{2\pi k \lambda}{2p} \right)$.
    Then the assertion 1 follows from Lemma~\ref{lem:sum_rt_unit}.  
    By applying the transformation rules of $C(g^3)$ given in Lemma~\ref{lem:sq_cub_conj_cls_d4n_odd} to the assertion 1, we obtain the desired formula of the assertion 2.
\end{proof}
By using Lemma~\ref{lem:chi_W_g_d4n_odd} and Proposition~\ref{prop:dim_formulas}, we obtain the following.

\begin{lemma}\label{lem:dim2_D4n_odd}
When $\pi=D_{4p}^*$ with odd $p$, we have
$d_2(\C\pi) = \begin{cases}
        \frac{1}{3} p^2 + \frac{3}{2} p + \frac{3}{2}  & (p \equiv 0 \bmod 3),\\
        \frac{1}{3} p^2 + \frac{3}{2}p + \frac{7}{6} & (p \not \equiv 0 \bmod 3).
    \end{cases}$
\end{lemma}

By  Lemma~\ref{lem:dim1_D4n_odd}, Lemma~\ref{lem:dim2_D4n_odd}, Proposition~\ref{prop:dim_formulas}, and  Proposition~\ref{prop:graph-inv}-2, a direct computation leads the following dimension formula for $\pi=D_{4p}^{\ast}$ where $p$ is odd.
\begin{prop}\label{lem:dim_D4n_odd}
    For $\pi = D_{4p}^{\ast}$ with $p$ odd, we have
    \begin{equation}
    \begin{split}
	&\dim \calA_\Theta^\odd(\C\pi)= \begin{cases}
        \frac{1}{3} p^2 + 2p + 2 & (p \equiv 0 \bmod 3),\\
        \frac{1}{3} p^2 + 2p + \frac{5}{3} & (p \not \equiv 0 \bmod 3),
    \end{cases}\\
	&\dim \calA_\Theta^\odd(\ker \ve)= \begin{cases}
        \frac{1}{3} p^2 + p  & (p \equiv 0 \bmod 3),\\
        \frac{1}{3} p^2 + p - \frac{1}{3} & (p \not \equiv 0 \bmod 3).
    \end{cases}
    \end{split}
\end{equation}
\end{prop}

\begin{table}[h]
\centering
\begin{tabular}{|c|c|c|c|c|c|c|c|c|c|c|c|c|c|c|c|} \hline
$p$ & $1$ & $2$ & $3$ & $4$ & $5$ & $6$ & $7$ & $8$ & $9$ & $10$ & $11$ & $12$ & $13$ & $14$ & $15$ \\ \hline
$\dim \calA_\Theta^\odd(\mathbb{C}\pi)$ & $4$ & $9$ & $11$ & $18$ & $20$ & $30$ & $32$ & $44$ & $47$ & $61$ & $64$ & $81$ & $84$ & $103$ & $107$ \\
$\dim \calA_\Theta^\odd(\mathrm{Ker}\,\ve)$ & $1$ & $4$ & $6$ & $11$ & $13$ & $21$ & $23$ & $33$ & $36$ & $48$ & $51$ & $66$ & $69$ & $86$ & $90$ \\ \hline
\end{tabular}
\caption{The values of the dimensions of $\calA_\Theta^\odd(\C\pi)$ and $\calA_\Theta^\odd(\mathrm{Ker}\,\ve)$ for $\pi = D_{4p}^{\ast}$ and $p \leq 15$. }\label{tab:val_dims_d4p}
\end{table}

\subsection{Group $D_{2^{k+2}p}'$}\label{section:D_2k+2p}

The group $D_{2^{k+2}p}'$ of order $2^{k+2}p$, where $k \geq 0$ and $p \geq 3$ odd, admits the following finite presentation:
\begin{equation}
    D_{2^{k+2}p}' = \langle x, y \mid x^{2^{k+2}}=1, y^p =1, xy^{-1} = yx \rangle.
\end{equation}
Any element $g \in D_{2^{k+2}p}'$ is uniquely written as $g = x^{n} y^{l}$ for some $0 \leq n \leq 2^{k+2}-1$ and $0 \leq l \leq p-1$. The character table is given as Table~\ref{tab:ch_d2k+2p}, which can be obtained from Proposition~\ref{prop:irrep_D'}. 

\begin{table}[h]
\centering
  \begin{tabular}{|c|c|c|c||c|}  \hline
    $\hat{\pi}$ & $C_{2m,0}$ & $C_{2m, l}$ & $C_{2m+1, 0}$ & Frobenius--Schur\\ 
    size & $1$ & $2$ & $p$ & indicator \\ \hline
$V^{(1)}_{0}$ & $1$ & $1$ & $1$ & $1$\\ 
$V^{(1)}_{2^{k+1}}$ & $1$ & $1$ & $-1$  &$1$\\ 
$V^{(1)}_{j}$ & $\zeta^{2m j}_{2^{k+2}}$ & $\zeta^{2mj}_{2^{k+2}}$ & $\zeta^{(2m+1) j}_{2^{k+2}}$ &$0$ \\ 
$V^{(2)}_{s,0}$ & $2$ &  $2 \cos \left(\frac{2\pi sl}{p} \right) $ & $0$ &$1$\\ 
$V^{(2)}_{s,2^k}$ & $ (-1)^m 2$ &  $(-1)^m2 \cos \left(\frac{2\pi sl}{p} \right) $ & $0$ & $-1$ \\ 
$V^{(2)}_{s,t}$ & $2 \zeta_{2^{k+2}}^{2mt}$ &  $2 \zeta_{2^{k+2}}^{2mt} \cos \left(\frac{2\pi sl}{p} \right) $ & $0$ & $0$ \\ \hline
  \end{tabular}
\caption{The characters $\chi_{A_i}(g)$ for $D_{2^{k+2}p}'$. Here, $m=0,1,\ldots, 2^{k+1}-1$, 
$l=1,\ldots, \frac{p-1}{2}$,  
$s=1,2,\ldots, \frac{p-1}{2}$, $j=1,2,\ldots,2^{k+1}-1,2^{k+1}+1,\ldots,2^{k+2}-1$, and $t=1,2,\ldots 2^k-1,2^k+1,\ldots,2^{k+1}-1$. 
The characters are real for $V^{(1)}_{0}$, $V^{(1)}_{2^{k+1}}$, $V^{(2)}_{s,0}$, and $V^{(2)}_{s,2^k}$.}\label{tab:ch_d2k+2p}
\end{table}

\subsubsection{Conjugacy classes of $D_{2^{k+2}p}'$}

For $\pi = D_{2^{k+2}p}'$, the set of conjugacy classes $\hat{\pi}$ consists of the following conjugacy classes:
\[\begin{split}
&C_{2m,0} = [x^{2m}] = \{x^{2m}\} \quad (0 \leq m < 2^{k+1}),\\
&C_{2m, l} = C_{2m, p-l} = [x^{2m}y^l] = \{x^{2m}y^l, x^{2m}y^{p-l}\} \quad (0 \leq m < 2^{k+1}, 1\leq l \leq \frac{p-1}{2}),\\
&C_{2m+1, 0} = [x^{2m+1}] = \{x^{2m+1}, x^{2m+1}y, x^{2m+1}y^2, \ldots, x^{2m+1}y^{p-1}\} \quad (0 \leq m < 2^{k+1}).
\end{split}\]
These classes are obtained directly by computing the values of irreducible characters at every element of $D_{2^{k+2}p}'$.

\begin{lemma}\label{lem:conj_z2_D_2k+2p}
Let $\pi = D_{2^{k+2}p}'$ where $k \geq 0$ and $p \geq 3$ odd. Then, $[x^{-1}] \neq [x] \in \hat{\pi}$ for any $x$ but $x = e, x^{2^{k+1}}, y^l, x^{2^{k+1}}y^l$ $(1 \leq l \leq p-1)$. In particular, $\dim(\C\hat{\pi})_{\Z_2}=(2^k+1)\cdot (p+1)/2 + 2^k$.
\end{lemma}

\begin{proof}
    Since $(x^n y^l)^{-1}= y^{-l} x^{-n} = x^{2^{k+2} -n}y^{-(-1)^n l}$ for any $l =0, \ldots, p-1$, $[(x^n y^l)^{-1}]=[x^{2^{k+2} -n}y^{-(-1)^n l}]=[x^{n}y^{l}]$ holds only when $n=0$ or $n=2^{k+1}$. Therefore, $\Z_2$ acts trivially on $C_{0,0}, C_{2\cdot 2^k,0}, C_{0,l}, C_{2\cdot 2^k, l}$ $(1 \leq l \leq (p-1)/2)$ and non trivially on other conjugacy classes in $\hat{\pi}$ and hence one obtains the desired dimension formula for $\dim(\C\hat{\pi})_{\Z_2}$.
\end{proof}

\begin{lemma}\label{lem:sq_cub_conj_cls_d2k+2p}
    Let $\pi = D_{2^{k+2}p}'$  where $k \geq 0$ and $p \geq 3$ odd. For any conjugacy class $[g] = C_{i, j} \in \hat{\pi}$ we denote by $C^l_{i,j} \in \hat{\pi}$ the conjugacy class represented by $g^l$ for an integer $l$. Then, we have 
    \begin{equation}
    \begin{split}
    &C_{2m,0}^2 = C_{4m,0}, \quad C_{2m,l}^2 = C_{4m,2l},\quad C_{2m+1,0}^2 = C_{4m+2,0},\\
            &C_{2m,0}^3 = C_{6m,0}, \quad C_{2m,l}^3 = C_{6m,3l},\quad C_{2m+1,0}^3 = C_{6m+3,0}.
        \end{split}
    \end{equation}
Here, $4m, 4m+2, 6m, 6m+3$ and $2l, 3l$ in the subscripts of $C_{4m,0}, C_{4m,2l}, C_{4m+2,0}, C_{6m,0}, C_{6m,3l}, C_{6m+3, 0}$ are understood as the least non-negative remainders modulo $2^{k+2}$ and $p$ respectively.
\end{lemma}

\begin{proof}
The equalities $C_{2m,0}^2 = C_{4m,0}$ and $C_{2m, 0}^3 = C_{6m,0}$ are clear by definition. The other cases follow from the relations $y^l x^{2m} = x^{2m} y^l$ and $y^l x^{2m+1} = x^{2m+1} y^{-l}$. These relations are consequences of $yx = x y^{-1}$ and $y^{-1} x = y^{p-1} x = x y^{1-p} = xy$.	
\end{proof}

\begin{remark}\label{rem:d2k+2p_c^2}
Let $l$ be an integer with $1\leq l \leq (p-1)/2$. Since $p$ is odd, we have $C_{4m,2l}=C_{4m, l'}$ for some $l'$ with $1\leq l' \leq (p-1)/2$. When $p \equiv 0 \bmod 3$, $C_{6m,3l} = C_{6m,0}$ hold when $l=p/3$ and otherwise $C_{6m,3l} = C_{6m,l'}$ holds for some $l'$ with $1\leq l' \leq (p-1)/2$.
\end{remark}

With respect to Remark \ref{rem:d2k+2p_c^2}, we prepare the following lemma for the computation of $\dim\,(\Sym^3 W)^{\pi\times \pi}$.
\begin{lemma}\label{lem:d2k+2p_dim1_chiw_cubic}
Let $\pi = D_{2^{k+2}p}'$  where $k \geq 0$ and $p \geq 3$ odd. Then,  the following equation holds.
\begin{equation}
    \sum_{(C(g),C(h)) \in \Delta_{\hat{\pi}}^{(3)}} \frac{|C(g)| |C(h)|}{|C(g^3)|} =  \begin{cases}
     2^{k+3} p & (p \equiv 0 \bmod 3),\\
     2^{k+2} p & (p \not \equiv 0 \bmod 3).
    \end{cases}
\end{equation}
\begin{proof}
    For the proof, see Appendix \ref{appendix:A}.
\end{proof}
\end{lemma}

As in previous sections, we first compute $d_1(\C\pi)= \dim\,(\Sym^3 \C\pi)^{\pi\times \pi}$, which is done as in the case of $D_{4p}^{\ast}$ with odd $p$ by applying Lemma~\ref{lem:d2k+2p_dim1_chiw_cubic} to the formula in Proposition~\ref{prop:dim_formulas}.

\begin{lemma}\label{lem:dim1_D2k+2p}
Let $\pi = D_{2^{k+2}p}'$  where $k \geq 0$ and $p \geq 3$ odd. Then, we have the following.
\begin{equation}
   d_1(\C \pi) = \begin{cases}
           \frac{1}{3}2^{2k}p^2 + \frac{1}{2}2^{k}(2^{k+1}+3)p + \frac{1}{6}(2^{2k+3} + 3\cdot 2^k +4) & (p \equiv 0 \bmod 3),\\
           \frac{1}{3}2^{2k}p^2 + \frac{1}{2}2^{k}(2^{k+1}+3)p + \frac{1}{6}(2^{2k+3} + 3\cdot 2^k +2)& ( p \not \equiv 0 \bmod 3).
         \end{cases}
\end{equation}
\end{lemma}

\subsubsection{Sum of real characters for $D_{2^{k+2}p}'$}

Next we compute  $d_2(\C\pi)$.
The following lemma corresponds to Lemma~\ref{lem:chi_W_g_d4n_even} and Lemma~\ref{lem:chi_W_g_d4n_odd} by which we compute $d_2(\C\pi)$.

\begin{lemma}\label{lem:chi_W_g_d2k+2p}
Let $\pi = D_{2^{k+2}p}'$  where $k \geq 0$ and $p \geq 3$ odd. Then, for $g \in D_{2^{k+2}p}'$, we have the following.
        \begin{equation}
        \begin{split}
               &1. \quad \sum_{\chi_{A_i}: \text{real}} \chi_{A_i}'(g) = \begin{cases}
                    2p & ([g]=C_{2m,0} \quad (2m\not \equiv 0 \bmod 4 )),\\
                    2  & ([g] = C_{2m,l} \quad (2m \equiv 0 \bmod 4, 0\leq l \leq (p-1)/2 )),\\
                    0 & (\text{otherwise}).
                \end{cases}\\
           &2. \quad \sum_{\chi_{A_i}: \text{real}} \chi_{A_i}'(g^3) = \begin{cases}
                    2p & ([g]=C_{2m,0} \quad (2m\not \equiv 0 \bmod 4)),\\
                    2p & ([g]=C_{2m,p/3},\ p \equiv 0 \bmod 3 \quad (2m\not \equiv 0 \bmod 4 )),\\
                    2  & ([g] = C_{2m,l} \quad (2m \equiv 0 \bmod 4, 0\leq l \leq (p-1)/2 )),\\
                    0 & (\text{otherwise}).
                \end{cases}
                \end{split}
        \end{equation}
\end{lemma}

\begin{proof}
For $0\leq n \leq 2^{k+1}-1$ and $0\leq l \leq p-1$, by using Table~\ref{tab:ch_d2k+2p}, one obtains the assertion 1 directly
    \begin{equation}
    \begin{split}
        \sum_{\chi_{A_i}: \text{real}} \chi_{A_i}'(x^n y^l)& = 1 +(-1)^n + \delta_{[n],[0]} \sum_{s=1}^{\frac{p-1}{2}} 2\left( \cos \left(\frac{2\pi sl}{p}\right) -  (-1)^{n/2} \cos \left(\frac{2\pi sl}{p} \right)\right),
        \end{split}
    \end{equation}
    where $\delta_{[n],[0]}$ denotes the Kronecker delta on $\Z/2\Z=\{[0], [1]\}$. 
  By applying Lemma~\ref{lem:sum_rt_unit} to this equation,  we obtain the desired formula. 
  The assertion 2 is immediate from the assertion 1 and Lemma~\ref{lem:sq_cub_conj_cls_d2k+2p}.
\end{proof}

\begin{lemma}\label{lem:dim2_D2k+2p}
Let $\pi = D_{2^{k+2}p}'$  where $k \geq 0$ and $p \geq 3$ odd. Then, the following holds.
\begin{equation}
    d_2(\C\pi) = \begin{cases}
        \frac{1}{3} \, p^{2} + \frac{3}{2}  2^{k} p+ \frac{1}{2}2^{k} + 1 & (p \equiv 0 \bmod 3),\\
        \frac{1}{3} p^2 + \frac{3}{2}2^k p + \frac{1}{2}2^k + \frac{2}{3} & (p \not \equiv 0 \bmod 3).
    \end{cases}
\end{equation}
\end{lemma}

\begin{proof}
    The proof is given by applying Lemma~\ref{lem:chi_W_g_d2k+2p} to the formula in Proposition~\ref{prop:dim_formulas}.
\end{proof}

By Lemma~\ref{lem:dim1_D2k+2p}, Lemma~\ref{lem:dim2_D2k+2p}, Lemma~\ref{lem:conj_z2_D_2k+2p}, Proposition~\ref{prop:dim_formulas}, and Proposition~\ref{prop:graph-inv}-2, one obtains the following dimension formula for $\pi = D_{2^{k+2}p}'$ where $k \geq 0$ and $p \geq 3$ odd. 
\begin{prop}\label{prop:dim_d2k+2p}
    For $\pi = D_{2^{k+2}p}'$ where $k \geq 0$ and $p \geq 3$ odd, we have
    \begin{equation}
    \begin{split}
    &\dim \calA_\Theta^\odd(\C\pi)= \begin{cases}
        \frac{1}{6} \cdot 2^{2k} p^{2} + \frac{1}{2} \cdot 2^{2k} p + \frac{3}{2} \cdot 2^{k} p + \frac{1}{6} \, p^{2} + \frac{2}{3} \cdot 2^{2k} + \frac{1}{2} \cdot 2^{k} + \frac{5}{6} & (p \equiv 0 \bmod 3),\\
        \frac{1}{6} \cdot 2^{2k} p^{2} + \frac{1}{2} \cdot 2^{2k} p + \frac{3}{2} \cdot 2^{k} p + \frac{1}{6} \, p^{2} + \frac{2}{3} \cdot 2^{2k} + \frac{1}{2} \cdot 2^{k} + \frac{1}{2} & (p \not \equiv 0 \bmod 3),
    \end{cases}\\
	&\dim \calA_\Theta^\odd(\ker \ve)= \begin{cases}
       \frac{1}{6} \cdot 2^{2k} p^{2} + \frac{1}{2} \cdot 2^{2k} p +  2^{k} p + \frac{1}{6} \, p^{2} + \frac{2}{3} \cdot 2^{2k} -  2^{k} - \frac{1}{2} p + \frac{1}{3} & (p \equiv 0 \bmod 3),\\
        \frac{1}{6} \cdot 2^{2k} p^{2} + \frac{1}{2} \cdot 2^{2k} p +  2^{k} p + \frac{1}{6} \, p^{2} + \frac{2}{3} \cdot 2^{2k} -  2^{k} - \frac{1}{2} p& (p \not \equiv 0 \bmod 3).
    \end{cases}
    \end{split}
\end{equation}
\end{prop}

\subsection{Binary tetrahedral group $T^{\ast}$}\label{section:binary_tetra}

The \textit{binary tetrahedral group} $T^{\ast}$ of order $24$ is given as the following finitely presented group.
\begin{align*}
    \langle a, b \mid (ab)^2 = a^3 = b^3 \rangle = \langle x, y, z \mid x^2 = (xy)^2 = y^2, zxz^{-1} =y, zyz^{-1}=xy, z^3=1 \rangle.
\end{align*}
This case can be computed by using a mathematics software system as in the case of the binary icosahedral group $I^{\ast}$ (\cite{Oh}, \cite{OW}). The set of conjugacy classes of $T^{\ast}$ consists of the following $7$ classes represented by the elements
$e, z, x^2y, x^2, z^2, x^2 z, x^3 z^2$.
The character table of $T^{\ast}$ is given as Table~\ref{tab:ch_tetra}. To apply the formula in Section \ref{section:4.1.1}, we need to know the conjugacy classes of $g^2$ and $g^3$ in $T^{\ast}$. The table is given as Table~\ref{tab:conj_T}. We obtain the lists of the conjugacy classes and the character tables by using {\tt Sage} via {\tt CharacterTable} and {\tt ConjugacyClass} functions in the Sage-GAP interface (\cite{sagemath}).

\begin{table}[h]
\centering
  \begin{tabular}{|c|ccccccc||c|}  \hline
    $\hat{\pi}$ & $e$ & $z$ & $x^2y$ & $x^2$ & $z^2$ & $x^2 z$ & $x^3 z^2$ & Frobenius--Schur  \\ 
    size &  1 & $4$ & $6$ & $1$ & $4$ & $4$ & $4$  & indicator  \\ \hline
    $V_1$ & 1 & 1 & 1 & 1 & 1 & 1 & 1  & $1$ \\ 
    $V_2$ & 1 & $\zeta_{3}^2$ & 1 & 1 & $\zeta_{3}$ & $\zeta_3^2$ & $\zeta_{3}$ & $0$\\ 
    $V_3$ & 1 & $\zeta_{3}$ & 1 & 1 & $\zeta_3^2$ & $\zeta_{3}$ & $\zeta_3^2$ & $0$\\
    $V_4$ & 2 & $-1$ & 0 & $-2$ & $-1$ & 1 & 1 & $-1$\\ 
    $V_5$ & 2 & $-\zeta_{3}$ & 0 & $-2$ & $-\zeta_{3}^2$ & $\zeta_{3}$ & $\zeta_{3}^2$ & $0$ \\ 
    $V_6$ & 2 & $-\zeta_{3}^2$ & 0 & $-2$ & $-\zeta_{3}$ & $\zeta_{3}^2$ & $\zeta_{3}$  & $0$\\ 
    $V_7$ & 3 & 0 & $-1$ & 3 & 0 & 0 & 0 & $1$\\ 
    \hline
  \end{tabular}
\caption{The characters $\chi_{V_i}(g)$ for $T^{\ast}$ where $\zeta_3 = \frac{-1 + \sqrt{-3}}{2}$. }\label{tab:ch_tetra}
\end{table}

\begin{table}[h]
\centering
\begin{tabular}{|c||c|c|c|c|c|c|c|c|}\hline
$g$ & $e$ & $z$ & $x^2y$ & $x^2$ & $z^2$ & $x^2 z$ & $x^3 z^2$ \\\hline
$g^2$ & $e$ & $z^2$ & $x^{2}$ & $e$ & $z$ & $z^2$ & $z$ \\ \hline
$g^3$ & $e$ & $e$ & $x^2 y$ & $x^{2}$ & $e$ & $x^{2}$ & $x^{2}$\\ \hline
\end{tabular}
\caption{The conjugacy classes of $g^2$ and $g^3$ in $T^{\ast}$.}\label{tab:conj_T}
\end{table}

\begin{lemma}\label{lem:z2_act_hat_T}
Let $\pi = T^{\ast}$. We have $[g^{-1}] = [g] \in \hat{\pi}$ for $g=e, x^2, x^2y$ and $[g^{-1}] \neq  [g] \in \hat{\pi}$ for $g=z, z^2, x^2z, x^3 z^2$. In particular, $\dim(\C\hat{\pi})_{\Z_2} = 5$.
\end{lemma}
\begin{proof}
    The assertion can be directly checked by using {\tt ConjugacyClass} functions via {\tt Sage} with an elementary {\tt Python} programming implementation.
\end{proof}

\begin{prop}\label{prop:dim_t*}
When $\pi=T^{\ast}$, we have the following.
\begin{equation}
	 \dim \calA_\Theta^\odd(\C\pi)=15, \quad \dim \calA_\Theta^\odd(\mathrm{Ker}\,\ve)=10.
\end{equation}
\end{prop}

\begin{proof}
    The former can be checked directly via any mathematics software system such as {\tt Maxima} and {\tt Sage}, once we have the formula in Proposition~\ref{prop:dim_formulas}, Table  \ref{tab:ch_tetra} and Table~\ref{tab:conj_T} at hand. The latter follows from Proposition~\ref{prop:graph-inv}-2 and Lemma~\ref{lem:z2_act_hat_T}.
\end{proof}

\subsection{Group $T'_{8\cdot 3^k}$}

The group $T'_{8\cdot 3^k}$ with $k \geq 1$ is given as the finitely presented group
\begin{align*}
    T_{8 \cdot 3^k}' &= \langle x, y, z \mid x^2 = (xy)^2 = y^2, zxz^{-1} =y, zyz^{-1}=xy, z^{3^k}=1 \rangle.
\end{align*}
Note that, when $k=1$, $T'_{8\cdot 3}=T^{\ast}$. Any element $g \in  T_{8 \cdot 3^k}'$ has the following unique expression as
\begin{equation}
w z^l \quad (w \in \{e, x, y, x^2=y^2=(xy)^2, xy, yx, x^3, y^3\}, l=0,1,\ldots, 3^{k}-1).
\end{equation}
The character table is given as Table~\ref{tab:ch_t83k}, which can be obtained from Proposition~\ref{prop:irrep_T'}.

\begin{table}[h]
\centering
  \begin{tabular}{|c|c|c|c|c|c|c|c||c|}  \hline
    $\hat{\pi}$ & $C_{1,3m}$ & $C_{2, 3m}$ & $C_{3, 3m}$ & $C_{4, 3m+1}$ & $C_{5, 3m+1}$ & $C_{6, 3m+2}$ & $C_{7, 3m+2}$ & Frobenius--Schur\\ 
    size & $1$ & $1$ & $6$ & $4$ & $4$ & $4$ &$4$ & indicator \\ \hline
$V^{(1)}_{0}$ & \multicolumn{3}{|c|}{$1$} & \multicolumn{2}{|c|}{$1$}  & \multicolumn{2}{|c||}{$1$} & $1$\\  \hline 
$V^{(1)}_{\lambda}$ & \multicolumn{3}{|c|}{$\zeta_{3^k}^{3m\lambda}$} & \multicolumn{2}{|c|}{$\zeta_{3^k}^{(3m+1)\lambda}$}  & \multicolumn{2}{|c||}{$\zeta_{3^k}^{(3m+2)\lambda}$} & $0$\\  \hline 
$V^{(2)}_{0}$ & $2 $ & $-2 $ & $0$ & $- 1$ & $1$ & $-1$ & $1$ &$-1$\\ \hline
$V^{(2)}_{\lambda}$ & $2 \zeta_{3^k}^{3m\lambda}$ & $-2 \zeta_{3^k}^{3m\lambda}$ & $0$ & $- \zeta_{3^k}^{(3m+1)\lambda}$ & $\zeta_{3^k}^{(3m+1)\lambda}$ & $-\zeta_{3^k}^{(3m+2)\lambda}$ & $\zeta_{3^k}^{(3m+2)\lambda}$ &$0$ \\ \hline
$V^{(3)}_0$ & \multicolumn{2}{|c|}{$3$} & $-1$ & \multicolumn{4}{|c||}{$0$} &$1$ \\ \hline
$V^{(3)}_{\lambda}$ & \multicolumn{2}{|c|}{$3 \zeta_{3^k}^{3m\lambda}$} & $-\zeta_{3^k}^{3m\lambda}$ & \multicolumn{4}{|c||}{$0$} &$0$ \\ \hline
  \end{tabular}
\caption{The characters $\chi_{V^{(i)}_{\lambda}}(g)$ for $T_{8 \cdot 3^k}'$ where $m = 0,1,\ldots, 3^{k-1}-1$ and, $\lambda=1, 2,\ldots, 3^k -1$ when $i=1,2$ and $\lambda=1,2,\ldots, 3^{k-1}-1$ when $i=3$. The characters are real for $V_0^{(1)},V_0^{(2)},V_0^{(3)}$.}\label{tab:ch_t83k}
\end{table}

\subsubsection{Conjugacy classes of $T'_{8\cdot 3^k}$}

From the explicit matrix expressions of the irreducible representations of $T_{8\cdot 3^k}'$, we can directly determine the set of conjugacy classes. 
The set of conjugacy classes of $T'_{8\cdot 3^k}$ consists of the following $7\cdot 3^{k-1}$ classes:
\[\begin{split}
&C_{1, 3m} = [e\cdot z^{3m}] = \{e \cdot z^{3m}\},\\
&C_{2, 3m} = [x^2 \cdot z^{3m}] =\left\{x^{2} \cdot z^{3m} \right\},\\
&C_{3, 3m} = [x^2y \cdot z^{3m}] = \left\{x\cdot z^{3m},y\cdot z^{3m},(xy)\cdot z^{3m}, (yx)\cdot z^{3m}, (x^3)\cdot z^{3m}, (y^3=x^2y)\cdot z^{3m}\right\},\\
&C_{4, 3m+1} = [z\cdot z^{3m}]=\left\{z\cdot z^{3m}, (yxz)\cdot z^{3m}, (x^3 z)\cdot z^{3m}, (y^3 z)\cdot z^{3m}\right\},\\
&C_{5, 3m+1} = [(xz)\cdot z^{3m}] = \left\{(xz)\cdot z^{3m}, (yz)\cdot z^{3m}, (x^2z)\cdot z^{3m}, (xyz)\cdot z^{3m}\right\},\\
&C_{6, 3m+2} = [z^2\cdot z^{3m}] =\left\{z^2\cdot z^{3m}, (xz^2)\cdot z^{3m}, (yz^2)\cdot z^{3m}, (xyz^2)\cdot z^{3m}\right\},\\
&C_{7, 3m+2} = [(x^3z^2)\cdot z^{3m}] =\left\{(x^2z^2)\cdot z^{3m}, (yxz^2)\cdot z^{3m}, (x^3z^2)\cdot z^{3m}, (y^3z^2)\cdot z^{3m}\right\},
\end{split}\]
where $m = 0, 1, \ldots, 3^{k-1}-1$.

\begin{lemma}\label{lem:sq_cub_conj_cls_t83k}
    Let $\pi = T_{8\cdot 3^k}'$. For any conjugacy class $C_{i, j}=[g] \in \hat{\pi}$ represented by $g \in T_{8 \cdot 3^k}'$, we denote by $C^l_{i,j} \in \hat{\pi}$ the conjugacy class represented by $g^l$ for any integer $l$. Then, we have 
    \begin{equation}
        \begin{split}
            &C_{1,3m}^2 = C_{1,6m}, \quad C_{2,3m}^2 = C_{1,6m},\quad C_{3,3m}^2 = C_{2,6m}, \\
            &C_{4,3m+1}^2 = C_{6, 6m+5}, \quad C_{5,3m+1}^2 = C_{6, 6m+5},\quad C_{6,3m+2}^2 = C_{4, 6m+4},\quad C_{7,3m+2}^2 = C_{4, 6m+4}, \quad\text{and}\\
            &C_{1,3m}^3 = C_{1,9m}, \quad C_{2,3m}^3 = C_{2,9m}, \quad C_{3,3m}^3 = C_{3,9m}, \\
            &C_{4,3m+1}^3 = C_{1, 9m+3}, \quad C_{5,3m+1}^3 = C_{2, 9m+3},\quad C_{6,3m+2}^3 = C_{1, 9m+6},\quad C_{7,3m+2}^3 = C_{2, 9m+6}.
        \end{split}
    \end{equation}
\end{lemma}

\begin{proof}
    The assertion can be checked directly using the relations $x^2 z = z x^2$, $zx = yz$, $zy = xy z$, and $z xy = x z$.
\end{proof}

\begin{lemma}\label{lem:conj_z2_t83k}
Let $\pi = T_{8\cdot 3^k}'$. We have
	$\dim(\C\hat{\pi})_{\Z_2} = \displaystyle\frac{1}{2}\cdot 3^{k} + 2 \cdot 3^{k-1} + \frac{3}{2}$.
\end{lemma}

\begin{proof}
    It suffices to know how $\Z_2$ acts on the set $\hat{\pi}$ of conjugacy classes in $\pi$. By Lemma~\ref{lem:z2_act_hat_T}, the group  $\Z_2$ acts on the classes $C_{1,0}=C(e), C_{2,0}=C(x^2), C_{3,0}=C(x^2 y)$ trivially. For other classes, one obtains directly that, for non-zero integer $m=1,\ldots, 3^{k-1}-1$, 
    \begin{equation}
        \Z_2 \cdot C_{1,3m} = \{C_{1,3m}, C_{1, 3^k - 3m}\}, \quad \Z_2 \cdot C_{2,3m} = \{C_{2,3m}, C_{2, 3^k - 3m}\}, \quad \Z_2 \cdot C_{3,3m} = \{C_{3,3m}, C_{3, 3^k - 3m}\},
    \end{equation}
    and, for any integer $m =0,1,\ldots, 3^{k-1}-1$, 
    \begin{equation}
        \Z_2 \cdot C_{4,3m+1} = \{C_{4,3m+1}, C_{6, (3^k - 3m-3) +2}\}, \quad  \Z_2 \cdot C_{5,3m+1} = \{C_{5,3m+1}, C_{7, (3^k - 3m -3)+2}\}.
    \end{equation}
    To check the last equality, it may be useful to use the relation $x^3 y= (xy)^3 = yx$.
\end{proof}

\begin{lemma} \label{lem:t83k_dim1_chiw_cubic}
    Let $\pi = T_{8\cdot 3^k}'$ with $k \geq 2$. The following holds.
    \begin{equation}
        \sum_{(C(g),C(h)) \in \Delta_{\hat{\pi}}^{(3)}} \frac{|C(g)||C(h)|}{|C(g^3)|} = 2^3 \cdot 3^{k+2}.
    \end{equation}
\end{lemma}
\begin{proof}
    The proof is given in Appendix \ref{appendix:A}.
\end{proof}

By substituting the result in Lemma~\ref{lem:t83k_dim1_chiw_cubic} to the formula in Proposition~\ref{prop:dim_formulas}, we get the following.
\begin{lemma}\label{dim1_t83k}
    When $\pi = T_{8 \cdot 3^k}'$ with $k \geq 2$, we have
    $d_1(\C \pi) = 38 \cdot 3^{2k-3} + 2 \cdot 3^k + 3$.
\end{lemma}

\subsubsection{Sum of real characters for $T_{8\cdot 3^k}'$}

Next, we move on to the computation of $d_2(\C\pi)$.

\begin{lemma}\label{lem:sum_real_char_t83k}
Let $\pi = T_{8\cdot 3^k}'$. Then, for any integer $m$  with $0 \leq m \leq 3^{k-1}-1$, we have the following.
        \begin{equation}
        \begin{split}
            &1. \quad  \sum_{\chi_{A_i}: \text{real}} \chi_{A_i}'(g) = \begin{cases}
            	6 & (g \in C_{2, 3m}),\\
                2 & (g \in C_{1,3m} ,C_{4,3m+1}, C_{6, 3m+2}),\\
                0 & (g \in  C_{3,3m}, C_{5,3m+1}, C_{7, 3m+2}).
                \end{cases}\\
       &2. \quad \sum_{\chi_{A_i}: \text{real}} \chi_{A_i}'(g^3) = \begin{cases}
                6 & (g \in C_{2, 3m},  C_{5,3m+1}, C_{7, 3m+2}),\\
                2 & (g \in C_{1,3m} ,C_{4,3m+1}, C_{6, 3m+2}),\\
                0 & (g \in  C_{3,3m}).
                \end{cases}
            \end{split}
        \end{equation}
\end{lemma}

\begin{proof}
    The assertion 1 is clear from the real character table Table~\ref{tab:ch_t83k}. 
    By Lemma~\ref{lem:sq_cub_conj_cls_t83k},  we know that $C_{1,3m}^3, C_{4, 3m+1}^3$, and $C_{6, 3m+2}^3$ are of the form $C_{1, 0}$, etc. Thus, the assertion 2 follows from this fact and assertion 1.
\end{proof}

\begin{lemma}\label{lem:dim2_t83k}
     When $\pi = T_{8 \cdot 3^k}'$ with $k \geq 2$, we have
     $d_2(\C\pi) = 2 \cdot 3^k + 3$.
\end{lemma}

\begin{proof}

By Lemma~\ref{lem:sum_real_char_t83k} and character  Table~\ref{tab:ch_t83k}, the proof reduces to direct computation using the formula in Proposition~\ref{prop:dim_formulas}.
\end{proof}

\begin{prop}\label{prop:dim_t83k}
When $\pi = T_{8\cdot 3^k}'$ with $k \geq 2$, we have the following.
\begin{equation}
	\dim \calA_\Theta^\odd(\C\pi) = 19 \cdot 3^{2k-3} + 2\cdot 3^k + 3, \quad \dim \calA_\Theta^\odd(\mathrm{Ker}\,\ve) = 19 \cdot 3^{2k - 3} + \frac{5}{2} \cdot 3^{k-1} + \frac{3}{2}.
\end{equation}
\end{prop}

\begin{proof}
    These are direct consequences of  Lemma~\ref{dim1_t83k}, Lemma~\ref{lem:dim2_t83k}, Lemma~\ref{lem:conj_z2_t83k}, the formula in Proposition~\ref{prop:dim_formulas}, and Proposition~\ref{prop:graph-inv}-2.
\end{proof}

\begin{table}[h]
\centering
\begin{tabular}{|c|c|c|c|c|c|c|c|c|c|} \hline
$k$ & $1$ & $2$ & $3$ & $4$ & $5$ & $6$ & $7$ & $8$ & $9$ \\ \hline
$\dim \calA_\Theta^\odd(\C\pi)$ & $15$ & $78$ & $570$ & $4782$ & $42042$ & $375438$ & $3370170$ & $30305262$ & $272668602$ \\
$\dim \calA_\Theta^\odd(\mathrm{Ker}\,\ve)$ & $10$ & $66$ & $537$ & $4686$ & $41757$ & $374586$ & $3367617$ &  $30297606$ & $272645637$ \\ \hline
\end{tabular}
\caption{The values of the dimensions of $\calA_\Theta^\odd(\C\pi)$ and $\calA_\Theta^\odd(\mathrm{Ker}\,\ve)$ for $\pi = T_{8\cdot 3^k}'$ and $k \leq 9$. }\label{tab:val_dims_t83k}
\end{table}

\subsection{Binary octahedral group $O^{\ast}$} \label{section:O_ast}
The \textit{binary octahedral group} $O^{\ast}$ of order $48$ admits the presentation $\langle a,b \mid (ab)^2=a^3=b^4 \rangle$.

By using {\tt ConjugacyClass} functions in the Sage-GAP interface (\cite{sagemath}), one can compute that the set of conjugacy classes of $O^{\ast}$ consists of $8$ classes represented respectively by the elements $e, ab, a^2, b^2, a^3, b, a, a^2b$.
The character table of $O^{\ast}$ is given as Table~\ref{tab:ch_octa}. As in Section \ref{section:binary_tetra}, the character table is also computed by using {\tt Sage} via {\tt CharacterTable} functions in the Sage-GAP interface. 

\begin{table}[h]
\centering
  \begin{tabular}{|c|cccccccc||c|}  \hline
    $\hat{\pi}$ & $e$ & $ab$ & $a^2$ & $b^2$ & $a^3$ & $b$ & $a$ & $a^2b$ & Frobenius--Schur \\ 
    size &  1 & $12$ & $8$ & $6$ & $1$ & $6$ & $8$  & $6$  & indicator\\ \hline
    $A_1$ & 1 & 1 & 1 & 1 & 1 & 1 & 1 & 1 & $1$\\ 
    $A_2$ & 1 & $-1$ & 1 & 1 & 1 & $-1$ & 1 & $-1$  &$1$\\ 
    $A_3$ & 2 & 0 & $-1$ & 2 & 2 & 0 & $-1$ & 0  & $1$\\ 
    $A_4$ & 2 & 0 & $-1$ & 0 & $-2$ & $-\sqrt{2}$ & 1 & $\sqrt{2}$ &$-1$ \\ 
    $A_5$ & 2 & 0 & $-1$ & 0 & $-2$ & $\sqrt{2}$ & 1 & $-\sqrt{2}$  &$-1$\\ 
    $A_6$ & 3 & 1 & 0 & $-1$ & 3 & $-1$ & 0 & $-1$ &$1$\\ 
    $A_7$ & 3 & $-1$ & 0 & $-1$ & 3 & 1 & 0 & 1 & $1$\\ 
    $A_8$ & 4 & 0 & 1 & 0 & $-4$ & 0 & $-1$ & 0 &$-1$\\ \hline
  \end{tabular}
\caption{The characters $\chi_{A_i}(g)$ for $O^{\ast}$. }\label{tab:ch_octa}
\end{table}

To compute the dimensions $\dim \calA_\Theta^\odd(\C\pi)$ and $\dim \calA_\Theta^\odd(\mathrm{Ker}\,\ve)$ via the formula in Section \ref{section:4.1.1}, we must determine the conjugacy classes of $g^2$ and $g^3$ in $O^{\ast}$ and $\dim(\C\hat{\pi})_{\Z_2}$. The former can be computed directly as Table~\ref{tab:conj_O}, and the latter is given in the following lemma, which can be shown by {\tt ConjugacyClass} functions combined with an elementary {\tt Sage} code implementation. 

\begin{lemma}\label{lem:involution-inv-O*}
Let $\pi = O^{\ast}$. For each class $[x] \in \hat{\pi}$, we have $[x^{-1}] = [x] \in \hat{\pi}$.	In particular, $\dim(\C\hat{\pi})_{\Z_2}=\dim \C\hat{\pi}=8$.
\end{lemma}

\begin{table}
\centering
\begin{tabular}{|c||c|c|c|c|c|c|c|c|c|}\hline
$g$ & $e$ & $ab$ & $a^2$ & $b^2$ & $a^3$ & $b$ & $a$ & $a^2 b $ \\\hline
$g^2$ & $e$ & $a^{3}$ & $a^{2}$ & $a^{3}$ & $e$ & $b^2$ & $a^{2}$ & $b^{2}$ \\ \hline
$g^3$ & $e$ & $ab$ & $e$ & $b^{2}$ & $a^{3}$ & $a^{2} b$ & $a^{3}$ & $b$ \\ \hline
\end{tabular}
\caption{The conjugacy classes of $g^2$ and $g^3$ in $O^{\ast}$.}\label{tab:conj_O}
\end{table}

\begin{prop}\label{prop:ex2}
When $\pi=O^{\ast}$, we have the following.
\begin{equation}
	 \dim \calA_\Theta^\odd(\C\pi)=35, \quad \dim \calA_\Theta^\odd(\mathrm{Ker}\,\ve)=27.
\end{equation}
\end{prop}

\begin{proof}
The former can be checked directly via any mathematics software system such as {\tt Maxima} and {\tt Sage} combined with the formula in Section \ref{section:4.1.1} and Table  \ref{tab:ch_octa} and Table~\ref{tab:conj_O}. The latter follows from Proposition~\ref{prop:graph-inv}-2 and Lemma~\ref{lem:involution-inv-O*}.
\end{proof}

\subsection{Binary icosahedral group $I^{\ast}$}
The \textit{binary icosahedral group} $I^{\ast}$ of order $120$ admits the presentation
$\langle a, b \mid (ab)^2 = a^3=b^5 \rangle$.
It is known that $I^{\ast}$ is isomorphic to $\mathrm{SL}_2(\F_5)$.
Let $\hat{\pi}$ denote the set of conjugacy classes of $\pi=I^{\ast}$. As in Section \ref{section:binary_tetra} and \ref{section:O_ast}, by using {\tt ConjugacyClass} functions in {\tt Sage},  we know that  $\hat{\pi}$ has 9 elements, represented respectively by the elements
$
e,
a^{3},
(a^{2}b^{2})^{2} a,
a b a^{2}b, 
a, 
(a^{2}b^{2})^{2}, 
a^{2} b^{2}, 
a^{2} b^{2}a, 
b$.

\begin{lemma}[{\cite[Lemma~4.1]{OW}}]\label{lem:involution-inv}
For $\pi=I^{\ast}=\mathrm{SL}_2(\F_5)$, $\hat{\pi}$ is invariant under taking the inverse. Namely, for each class $[x]\in\hat{\pi}$, we have $[x^{-1}]=[x]$. In particular, $\dim(\C\hat{\pi})_{\Z_2}=\dim \C\hat{\pi}=9$.
\end{lemma}
\begin{proof}
    As in Section \ref{section:binary_tetra} and \ref{section:O_ast},  we can also check the assertion by using computer program such as {\tt ConjugacyClass} functions provided in {\tt Sage}.
\end{proof}

There are 9 distinct irreducible representations $A_i$ ($i=1,2,\ldots,9$) of the group $\pi$ whose character is given as in Table~\ref{tab:ch}, and any irreducible representation of $\pi$ over $\C$ is isomorphic to one of them. The  tables Table~\ref{tab:ch} and \ref{tab:conj} are computed by using {\tt Sage} with {\tt CharacterTable} functions therein.

\begin{table}[h]
\centering
  \begin{tabular}{|c|ccccccccc||c|}  \hline
    $\hat{\pi}$ & $e$ & $a^{3}$ & $(a^{2}b^{2})^{2} a$ &  $a b a^{2}b$ &  $a$ & $(a^{2}b^{2})^{2}$ &  $a^{2} b^{2}$ &  $a^{2} b^{2}a$ & $b$ & Frobenius--Schur \\ 
    size & 1 & 1 & 30 & 20 & 20 & 12 & 12 & 12 & 12 & indicator\\ \hline
    $A_1$ & 1 & 1 & 1 & 1 & 1 & 1 & 1 & 1 & 1 &$1$\\ 
    $A_2$ & 2 & $-2$ & 0 & $-1$ & 1 & $-\phi^*$ & $-\phi$ & $\phi^*$ & $\phi$ &$-1$\\ 
    $A_3$ & 2 & $-2$ & 0 & $-1$ & 1 & $-\phi$ & $-\phi^*$ & $\phi$ & $\phi^*$ &$-1$\\ 
    $A_4$ & 3 & 3 & $-1$ & 0 & 0 & $\phi$ & $\phi^*$ & $\phi$ & $\phi^*$ & $1$\\ 
    $A_5$ & 3 & 3 & $-1$ & 0 & 0 & $\phi^*$ & $\phi$ & $\phi^*$ & $\phi$ &$1$\\ 
    $A_6$ & 4 & 4 & 0 & 1 & 1 & $-1$ & $-1$ & $-1$ & $-1$ & $1$\\ 
    $A_7$ & 4 & $-4$ & 0 & 1 & $-1$ & $-1$ & $-1$ & 1 & 1 &$-1$ \\ 
    $A_8$ & 5 & 5 & 1 & $-1$ & $-1$ & 0 & 0 & 0 & 0 & $1$\\ 
    $A_9$ & 6 & $-6$ & 0 & 0 & 0 & 1 & 1 & $-1$ & $-1$ &$-1$\\ \hline
  \end{tabular}
\caption{The characters $\chi_{A_i}(g)$ for $I^{\ast}$ where we set $\phi=\frac{1+\sqrt{5}}{2}$ and  $\phi^*=\frac{1-\sqrt{5}}{2}$. }\label{tab:ch}
\end{table}

\begin{table}[h]
\centering
\begin{tabular}{|c||c|c|c|c|c|c|c|c|c|}\hline
$g$ & $e$ & $a^{3}$ & $(a^{2}b^{2})^{2} a$ &  $a b a^{2}b$ &  $a$ & $(a^{2}b^{2})^{2}$ &  $a^{2} b^{2}$ &  $a^{2} b^{2}a$ & $b$\\\hline
$g^2$ & $e$ & $e$ & $a^{3}$ & $a b a^{2} b$ & $aba^{2} b$ & $a^{2} b^{2}$ & $(a^{2} b^{2})^{2}$ & $a^{2} b^{2}$ & $(a^{2} b^{2})^{2}$ \\ \hline
$g^3$ & $e$ & $a^{3}$ & $(a^{2}b^{2})^{2} a$ & $e$ & $a^{3}$ & $a^{2} b^{2}$ & $(a^{2} b^{2})^{2}$ & $b$ & $a^{2} b^{2} a$\\ \hline
\end{tabular}
\caption{The conjugacy classes of $g^2$ and $g^3$ in $I^{\ast}$.}\label{tab:conj}
\end{table}

Using the formula as in Section \ref{section:4.1.1} together with Table~\ref{tab:ch} and Table~\ref{tab:conj}, the second author and Ohta establish the following dimension formula, which can also be checked directly via {\tt Sage} combined with the formula in Section \ref{section:4.1.1} and Table  \ref{tab:ch}.
\begin{prop}[{\cite{OW}, \cite{Oh}}]\label{prop:dim_I*}
When $\pi=I^{\ast} = \mathrm{SL}_2(\F_5)$, we have the following.
\begin{equation}
	 \dim \calA_\Theta^\odd(\C\pi)=65, \quad \dim \calA_\Theta^\odd(\mathrm{Ker}\,\ve)=56.
\end{equation}
\end{prop}

\subsection{Direct products with $\Z_m$: Proof of Theorem~\ref{thm:dim-formula}}
According to the list given in Theorem~\ref{thm:classification}, the remaining case is that $\pi$ is given by a direct product of any of the groups $D_{4p}^{\ast}, D_{2^{k+2}p}', T^{\ast}, T_{8\cdot 3^k}', O^{\ast}, I^{\ast}$  with a cyclic group of relatively prime order.

\begin{lemma}\label{lem:dim_cpi_z2}
Let $m$ be an odd positive integer and $\pi = \Z_m \times \pi'$ where $\pi'$ is a finite group. We denote by $\hat{\pi}_0'$ the set of fixed points of $\hat{\pi}'$ under the $\Z_2$-action, i.e., $\hat{\pi}_0'=\{[x] \in \hat{\pi}'\mid [x^{-1}]= [x]\}$, and set $\hat{\pi}_1'= \hat{\pi}' \setminus \hat{\pi}_0'$. Then, we have
\begin{equation}
	\dim(\C \hat{\pi})_{\Z_2} = \frac{m+1}{2} |\hat{\pi}_0'| + \frac{m}{2} |\hat{\pi}_1'|.
\end{equation} 
\end{lemma}

\begin{proof}
Since $m$ is odd, the equality $x^{-1} = x \in \Z_m$ holds if and only if $x = e \in \Z_m$. Therefore, $[(x_1, x_2)^{-1}] = [(x_1, x_2)] \in \hat{\pi}$ holds if and only if $x_1 = e \in \Z_m$ and $[x_2] \in \hat{\pi}_0'$. It implies that 
\begin{equation}
	 \dim(\C \hat{\pi})_{\Z_2} = |[e] \times \hat{\pi}_0'| + \frac{1}{2} \sum_{x \in \Z_m; x \neq e} | [x] \times \hat{\pi}_0'| + \frac{1}{2} \sum_{x \in \Z_m} | [x] \times \hat{\pi}_1'| 
     = \frac{m+1}{2} |\hat{\pi}_0'| + \frac{m}{2} |\hat{\pi}_1'|.
\end{equation}
\end{proof}

\begin{lemma}\label{lem:dim_cpi_z2_sherical}
    Let $\pi$ be the fundamental group of a spherical $3$-manifold. Then, the dimension $\dim(\C \hat{\pi})_{\Z_2}$ is given as follows.
    \begin{enumerate}
        \item[\rm (a)] $\dim(\C \hat{\pi})_{\Z_2}=p_2(n)=1 + \floor{\frac{n}{2}}$ when $\pi = \Z_n$. Here, $p_2(n)$ is the number of partitions of $n$ into at most $2$ and $\floor{\frac{n}{2}}$ denotes the greatest integer less than or equal to $\frac{n}{2}$.
        \item[\rm (b)]  $(1)$ $\dim(\C \hat{\pi})_{\Z_2}=\frac{1}{2}mp+ \frac{3}{2}m + \frac{1}{2}p + \frac{3}{2}$ when $\pi = \Z_m \times D_{4p}^{\ast}$, $m \geq 1$, $p>0$ even, and $(m, 2p)=1$.
        
        $(2)$ $\dim(\C \hat{\pi})_{\Z_2}=\frac{1}{2} \, m p + \frac{3}{2} \, m + \frac{1}{2} \, p + \frac{1}{2}$ when $\pi = \Z_m \times D_{4p}^{\ast}$, $m \geq 1$, $p>0$ odd, and $(m, 2p)=1$.	
        \item[\rm (c)] $\dim(\C \hat{\pi})_{\Z_2}=\frac{1}{2} \cdot 2^{k} m p + \frac{3}{2} \cdot 2^{k} m + \frac{1}{2} \, p + \frac{1}{2}$ when $\pi = \Z_m \times D_{2^{k+2}p}'$, $m\geq 1$, $k\geq 0$, $p \geq 3$ odd, and $(m, 2p)=1$.
        \item[\rm (d)]  $\dim(\C \hat{\pi})_{\Z_2}=\frac{7}{2}m + \frac{3}{2}$ when $\pi = \Z_m \times T^{\ast}$, $m\geq 1$ and $(m, 6)=1$.
        \item[\rm (e)] $\dim(\C \hat{\pi})_{\Z_2}=\frac{7}{6} \cdot 3^{k} m + \frac{3}{2}$ when $\pi = \Z_m \times T_{8\cdot 3^k}'$, $m\geq 1$, $(m, 6)=1$, and $k \geq 2$.
        \item[\rm (f)] $\dim(\C \hat{\pi})_{\Z_2}=4m + 4$ when $\pi=\Z_m \times O^{\ast}$, $m\geq 1$ and $(m,6)=1$.
        \item[\rm (g)] $\dim(\C \hat{\pi})_{\Z_2}=\frac{9}{2}m + \frac{9}{2}$ when $\pi = \Z_m \times I^{\ast}$, $m\geq 1$ and $(m,30)=1$.
    \end{enumerate}
\end{lemma}

\begin{proof}
    The assertions are consequences of Lemma~\ref{lem:dim_cpi_z2} and the arguments in the proofs of Lemma~\ref{lem:conj_z2_zn}, Lemma~\ref{lem:conj_inv_D_4n_even}, Lemma~\ref{lem:conj_z2_D4n_odd}, Lemma~\ref{lem:conj_z2_D_2k+2p}, Lemma~\ref{lem:conj_z2_t83k}, Lemma~\ref{lem:involution-inv-O*}, and Lemma~\ref{lem:involution-inv}.
\end{proof}
We are now in the position to give a proof of Theorem~\ref{thm:dim-formula} stated in the introduction, which determines the dimensions $\dim \calA_\Theta^\odd(\C\pi)$ and $\dim \calA_\Theta^\odd(\mathrm{Ker}\,\ve)$ for spherical $3$-manifold groups.

\begin{proof}[Proof of Theorem~\ref{thm:dim-formula}]
When $\pi = \Z_n$, it follows from Proposition~\ref{prop:OW_dim_cyclic} and Proposition~\ref{prop:dim_zn}. Thus, we suppose that $\pi$ is given as  $\pi = \Z_m \times \pi'$ for some finite group $\pi'$ as in the above list. Recall that there is a bijection between the set  $\hat{\pi}$ of conjugacy classes of $\pi$ and the direct product $\hat{\Z}_m \times \hat{\pi}'$ of those of $\Z_m$ and $\pi'$. Thus, we have
\begin{equation}
    \begin{split}
        \sum_{C(q) \in \hat{\pi}} \biggl(\frac{1}{ |C(q)|} |\pi|^2+3 \frac{|C(q)|}{|C(q^2)|} |\pi|\biggr) &= \sum_{C(x) \in \hat{\Z}_m} \sum_{C(g) \in \hat{\pi}'} \biggl( \frac{1}{|C(x)| |C(g)|} m^2 |\pi'|^2+3 \frac{|C(x)||C(g)|}{|C(x^2)||C(g^2)|} m|\pi'|\biggr)\\
        &=\sum_{C(g) \in \hat{\pi}'} \biggl(\frac{1}{ |C(g)|}m^3 |\pi'|^2+3 \frac{|C(g)|}{|C(g^2)|} m^2|\pi'|\biggr),\\
        \sum_{(C(q), C(r)) \in \Delta^{(3)}_{\hat{\pi}} } \frac{|C(q)| |C(r)|}{|C(q^3)|} &= \sum_{(C(x), C(y)) \in \Delta^{(3)}_{\hat{\Z}_m}} \sum_{(C(g), C(h)) \in \Delta^{(3)}_{\hat{\pi}'} } \frac{|C(x)||C(g)| |C(y)||C(h)|}{|C(x^3)||C(g^3)||C(y^3)||C(h^3)|}\\
        &=|\Delta^{(3)}_{\hat{\Z}_m}| \sum_{(C(g), C(h)) \in \Delta^{(3)}_{\hat{\pi}'}} \frac{|C(g)| |C(h)|}{|C(g^3)|}, 
    \end{split}
\end{equation}
where $|\Delta^{(3)}_{\hat{\Z}_m}|=m$ for $m\not \equiv 0 \bmod 3$, and $|\Delta^{(3)}_{\hat{\Z}_m}|=3m$ for $m \equiv 0 \bmod 3$ by Lemma~\ref{lem:zn_dim1_chiw_cubic}. Hence, by substituting these into the formula to compute $d_1(\C\pi)$ in Proposition~\ref{prop:dim_formulas}, one obtains
\begin{equation}\label{eq:modified_ow_formula_m}
    \begin{split}
  d_1(\C \pi) 
  &= \displaystyle \frac{1}{6|\pi'|}\Biggl( \sum_{C(g) \in \hat{\pi}'} \biggl(\frac{m^2}{ |C(g)|} |\pi'|^2+3m \frac{|C(g)|}{|C(g^2)|} |\pi'|\biggl) + a\sum_{(C(g),C(h)) \in \Delta_{\hat{\pi}'}^{(3)}} \frac{|C(g)| |C(h)|}{|C(g^3)|}\Biggr),
\end{split}
\end{equation}
where $a=2$ for $m\not \equiv 0 \bmod 3$, and $a=6$ for $m \equiv 0 \bmod 3$.
Similarly, we have
\begin{equation}\label{eq:modified_ow_formula_tau_m}
d_2(\C\pi) = \frac{1}{6 |\pi'|} \sum_{C(g) \in \hat{\pi}'} |C(g)| \Bigl\{\Bigl(
\sum_{\chi_{A_i'}: \text{real}}\chi_{A_i'}'(g)\Bigr)^3+3\frac{|\pi'|}{|C(g)|} \Bigl(\sum_{\chi_{A_j'}: \text{real}}\chi_{A_j'}'(g)\Bigr)+2\sum_{\chi_{A_i'}: \text{real}}\chi_{A_i'}'(g^3)\Bigr\},
\end{equation}
where we denote the distinct irreducible representations of $\pi'$ by  $A_i'$. Note here that we use the assumption that $m$ is odd and apply Lemma~\ref{lem:chi_W_g_zn}. Then, by using the modified formulas \eqref{eq:modified_ow_formula_m} and \eqref{eq:modified_ow_formula_tau_m}, we obtain the dimension formula for $\dim \calA_\Theta^\odd(\C\pi)$ by applying lemmas obtained in the previous sections. Once we have obtained the formula for $\dim \calA_\Theta^\odd(\C\pi)$, that for $\dim \calA_\Theta^\odd(\mathrm{Ker}\,\ve)$ are computed directly from Proposition~\ref{prop:graph-inv}-2 and Lemma~\ref{lem:dim_cpi_z2_sherical}. 
\end{proof}

\begin{remark}
    For $\pi=\Z_m \times D_{4p}^{\ast}, \Z_m \times D_{2^{k+2}p}'$, the resulting dimension formulas when $p\equiv 0 \bmod 3, m \not \equiv 0 \bmod 3$ and those when $p \not \equiv 0 \bmod 3, m \equiv 0 \bmod 3$ happen to coincide although their computations are different.
\end{remark}

\appendix
\section{Proofs of technical Lemmas} \label{appendix:A}

\subsection{Cyclic group $\Z_n$}
\begin{proof}[Proof of Lemma~\ref{lem:zn_dim1_chiw_cubic}]
Since $\Z_n$ is abelian, $|C(g)|=1$ for any $C(g) \in \hat{\pi}$. Thus,
\begin{equation}\label{eq:cube_sum_zn}
     \sum_{(C(g),C(h)) \in \Delta_{\hat{\pi}}^{(3)}} \frac{|C(g)| |C(h)|}{|C(g^3)|}   = \# \{ (m_1, m_2) \in \Z_n \times \Z_n \mid 3 m_1 \equiv 3 m_2 \bmod n \}.
\end{equation}
When $n \not \equiv 0 \bmod 3$, the pairs $(m_1, m_2) \in \Z_n^2$ with $3 m_1 \equiv 3 m_2 \bmod n$ are given by $m_1 = m_2 \in \Z_n$. Thus, the RHS of \eqref{eq:cube_sum_zn} is equal to $n$. 

When $n \equiv 0 \bmod 3$, the condition $3m_1 \equiv 3 m_2 \bmod n$ is equivalent to $m_1 \equiv m_2 \bmod n/3$. Therefore, the RHS of \eqref{eq:cube_sum_zn} is equal to the number of intersections of $\{0,1,\ldots, n-1\}^2$ with $m_2 = m_1 \pm \frac{n}{3}k$ $(k=0,1,2)$, which equals to $3n$ as desired.
\end{proof}

\subsection{Binary dihedral group $D_{4p}^{\ast}$}
\begin{proof}[Proof of Lemma~\ref{lem:d4n_even_dim1_chiw_cubic}]
    By Lemma~\ref{lem:sq_cub_conj_cls_d4n_even},  $C_{k,0}^3$ and $C_{l,1}^3$ are not equal to each other for any $k, l$. Thus, the set $\Delta_{\hat{\pi}}^{(3)}$ is disjoint union of the following subsets $D$ and $E$:
    \[
    D = \{(C_{k,0},C_{l,0})\mid 3k\equiv 3l \bmod {2p}\}, \quad E=\{(C_{\mathrm{odd},1},C_{\mathrm{odd},1}),(C_{\mathrm{even},1},C_{\mathrm{even},1})\}.
    \]
    Therefore, it is enough to show the following claim.
\begin{claim}
Let $\pi = D_{4p}^{\ast}$ with even $p$. Then, the following holds.
\begin{enumerate}
    \item $\displaystyle\sum_{(C(g), C(h)) \in D } \frac{|C(g)| |C(h)|}{|C(g^3)|}=  
        \begin{cases}
        6p & (p \equiv 0 \bmod 3),\\
        2 p & (p \not \equiv 0 \bmod 3).
        \end{cases}$
\item $\displaystyle
	\sum_{(C(g), C(h)) \in E } \frac{|C(g)| |C(h)|}{|C(g^3)|} = 2p$.
\end{enumerate}
\end{claim}

\begin{proof}[Proof of Claim]
	For the assertion 1, let us suppose that $p \not \equiv 0 \bmod 3$. In this case, $D=\{(C_{k,0},C_{l,0})\mid k\equiv l \bmod {2p}\} = \{(C_{k,0},C_{k,0})\mid 1\leq k \leq p\}$ holds and hence, by using Lemma~\ref{lem:sq_cub_conj_cls_d4n_even} and the character Table~\ref{tab:ch_d4n_even}, we have
\begin{equation}
\sum_{(C(g), C(h)) \in D } \frac{|C(g)| |C(h)|}{|C(g^3)|} = |C_{0,0}|^2 + \sum_{k=1}^{p-1} \frac{|C_{k,0}|^2}{2}  + |C_{p,0}|^2 
		 = 2p.
\end{equation}

    Next, suppose that $p  \equiv 0 \bmod 3$ and set $m=p/3$. Since $|C_{0,0}|=|C_{p,0}|=1$ and $|C_{k,0}|=2$ $(k=1,2,\ldots, p-1)$ hold, the sum decomposes into 
    \begin{equation}\label{eq:A_d4n_1}
        \sum_{(C(g), C(h))\in A} |C(g)| \cdot |C(h)| + \sum_{(C(g),C(h)) \in B} \frac{|C(g)| \cdot |C(h)|}{2},\qquad\text{where}
    \end{equation}
    \[
    A \coloneqq  \{ (C_{k_1,0}, C_{k_2,0})  \mid 3k_1 \equiv3 k_2 \equiv 0, p \bmod 2p \}, \quad B \coloneqq  \{ (C_{k_1,0}, C_{k_2,0})  \mid 3k_1 \equiv3 k_2 \not \equiv 0, p \bmod 2p \}.
    \]
    By Lemma~\ref{lem:sq_cub_conj_cls_d4n_even} and noting Remark \ref{rem:d4n_even_c^2}, the subsets $A$ and $B$ are equal to the following sets:
    \begin{equation}
        \begin{split}
          A \cong 
            & \{(0,0), (m, m),(2m,2m), (3m,3m)=(p,p)\} \\
            &\cup \{(0,2m), (m,3m)=(m,p) \} \cup (\text{similar set with $k_1$ and $k_2$ interchanged}),\\
		B \cong 
        &\{(k,k) \mid 1 \leq k \leq p-1, k \neq m, 2m\}\\
        &\cup \{(k, k+m) \mid 1 \leq k \leq 2m-1, k\neq m \} \cup \{(k, k+2m) \mid 1 \leq k \leq m-1\}\\
        & \cup (\text{similar set with $k_1$ and $k_2$ interchanged}).
		\end{split}
	\end{equation}
    Note that $|B|=3(p-3)$ and $|C(g)| \cdot |C(h)|=4$ for any $(C(g),C(h)) \in B$.
    Therefore, by character Table~\ref{tab:ch_d4n_even}, we have
    \begin{equation} \label{eq:chiw_33_d4neven1}
        \sum_{(C(g),C(h)) \in A} |C(g)| \cdot |C(h)| = 18,\qquad
        \sum_{(C(g), C(h)) \in B} \frac{|C(g)| \cdot |C(h)|}{2} = 2 |B| = 6p-18.
    \end{equation}
The sum of the two values in \eqref{eq:chiw_33_d4neven1} gives $6p$ as desired in the assertion 1.
The assertion 2 follows immediately by Lemma~\ref{lem:sq_cub_conj_cls_d4n_even} and Table~\ref{tab:ch_d4n_even}.
\end{proof}
By adding the identities in Claim 1 and 2, we get the desired formula.
\end{proof}

\subsection{Group $D_{2^{k+2}p}'$}
\begin{proof}[Proof of Lemma~\ref{lem:d2k+2p_dim1_chiw_cubic}]
The proof is similar to Lemma~\ref{lem:d4n_even_dim1_chiw_cubic}. 
By Lemma~\ref{lem:sq_cub_conj_cls_d2k+2p},  $C_{2m,l}^3$ and $C_{2m'+1,0}^3$ are not equal to each other for any $m, m', l$. Thus, $\Delta_{\hat{\pi}}^{(3)}=D\sqcup E$, where
    \begin{align*}
    D &= \{(C_{2m_1,l_1},C_{2m_2,l_2})\mid 3m_1\equiv 3m_2 \bmod {2^{k+1}}, 3l_1 \equiv 3l_2 \bmod p\} \\
    &= \{(C_{2m,l_1},C_{2m,l_2})\mid 0 \leq m \leq 2^{k+1} -1, 3l_1 \equiv 3l_2 \bmod p\}\\
    E& =\{(C_{2m_1+1,0},C_{2m_2+1,0})\mid 3m_1\equiv 3m_2 \bmod {2^{k+1}}\}\\
    &= \{(C_{2m+1,0},C_{2m+1,0})\mid 0 \leq m \leq 2^{k+1} -1\}
    \end{align*}
    where we use the fact that $3m_1 \equiv 3m_2 \bmod 2^{k+1}$ is equivalent to $m_1 = m_2$ for $m_1, m_2$ with $0 \leq m_1, m_2 \leq 2^{k+1}-1$. Then, the proof boils down to the following claim.
\begin{claim}
Let $\pi = D_{2^{k+2}p}'$. Then,  the following equations hold.
\begin{enumerate}
\item $\displaystyle
   \sum_{(C(g), C(h)) \in D } \frac{|C(g)| |C(h)|}{|C(g^3)|}  = 
   \begin{cases}
       3\cdot 2^{k+1}p & (p \equiv 0 \bmod 3),\\
       2^{k+1}p & (p \not \equiv 0 \bmod 3).
   \end{cases}$
\item $\displaystyle
	 \sum_{(C(g), C(h)) \in E } \frac{|C(g)| |C(h)|}{|C(g^3)|}   = 2^{k+1}p$.
\end{enumerate}
\end{claim}

\begin{proof}[Proof of Claim]
For the assertion 1, suppose that $p \not \equiv 0 \bmod 3$. By Lemma~\ref{lem:sq_cub_conj_cls_d2k+2p} and Remark \ref{rem:d2k+2p_c^2}, for any $l$ $(l=1,2,\ldots, (p-1)/2)$, $C_{2m,l}^3 = C_{6m, 3l} = C_{6m, l'}$ for some $l'$ $(1\leq l'\leq (p-1)/2)$. Therefore, we can compute the sum easily by applying the character Table~\ref{tab:ch_d2k+2p} and Lemma~\ref{lem:sq_cub_conj_cls_d2k+2p} as follows:
\begin{equation}
    \sum_{(C(g), C(h)) \in D } \frac{|C(g)| |C(h)|}{|C(g^3)|}
         = \sum_{m =0}^{2^{k+1}-1} \frac{|C_{2m, 0}| \cdot |C_{2m,0}|}{|C_{2m, 0}^3|} + \sum_{l=1 }^{(p-1)/2} \sum_{m =0}^{2^{k+1}-1} \frac{|C_{2m, l}| \cdot |C_{2m,l}|}{|C_{2m, l}^3|} =  2^{k+1} p.
\end{equation}

 Next, suppose that $p \equiv 0 \bmod 3$. Since $|C_{2m,0}|=1$ and $|C_{2m,l}|=2$ $(l=1,2,\ldots, (p-1)/2)$ hold, the sum decomposes into 
    \begin{equation}\label{eq:A_d2k=2p_1}
    \begin{split}
     & \sum_{(C(g), C(h))\in A} |C(g)| \cdot |C(h)| + \sum_{(C(g),C(h)) \in B} \frac{|C(g)| \cdot |C(h)|}{2},\qquad\text{where}
     \end{split}
     \end{equation}
     \begin{equation}
    \begin{split}
    A &= \{(C_{2m,l_1},C_{2m,l_2})\mid 0 \leq m \leq 2^{k+1} -1, 3l_1 \equiv 3l_2 \equiv 0 \bmod p\},\\
    B &= \{(C_{2m,l_1},C_{2m,l_2})\mid 0 \leq m \leq 2^{k+1} -1, 3l_1 \equiv 3l_2 \not \equiv 0 \bmod p\}.
    \end{split}
    \end{equation}
    
    By Lemma~\ref{lem:sq_cub_conj_cls_d2k+2p} and Remark \ref{rem:d2k+2p_c^2}, the subsets $A$ and $B$ are equal to the following sets
    \begin{equation}
        \begin{split}
            A \cong  & \{(2m,2m) \mid 0 \leq m \leq 2^{k+1} -1\}  \times \{ (l_1, l_2) \in \{0,1,\ldots, (p-1)/2\}^2 \mid 3l_1 \equiv3 l_2 \equiv 0 \bmod p \}\\
            =&\{(2m,2m) \mid 0 \leq m \leq 2^{k+1} -1\}  \times \biggl( \{(0,0), (p/3, p/3)\} \cup \{(0,p/3), (p/3,0)\}\biggr),\\
            B \cong &  \{(2m,2m) \mid 0 \leq m \leq 2^{k+1} -1\}  \times \{ (l_1, l_2) \in \{0,1,\ldots, (p-1)/2\}^2 \mid 3l_1 \equiv3 l_2 \not \equiv 0 \bmod p \}\\
            =&  \{(2m,2m) \mid 0 \leq m \leq 2^{k+1} -1\} \\
            &\times  \biggl(\{(l,l) \mid 1 \leq l \leq (p-1)/2, l \neq p/3\} \cup \{(l, l+p/3) \mid 1 \leq l \leq (p-3)/6\} \\
		&\cup \{ (l, p/3-l) \mid 1\leq l \leq (p/3-1)/2\} \cup \{(l, 2p/3-l) \mid (p+3)/6 \leq l \leq p/3 - 1\}\\
        & \cup (\text{similar set with $l_1$ and $l_2$ interchanged})\bigg).
        \end{split}
    \end{equation}
    Here, note that $|B| = 2^{k+1} \frac{3(p-3)}{2}$ and $|C(g)||C(h)| = 4$ for any $(C(g), C(h)) \in B$. Thus, by character Table~\ref{tab:ch_d2k+2p}, we have 
    \begin{equation} \label{eq:chiw_33_d2k+2_1}
        \sum_{(C(g), C(h))\in A} |C(g)| \cdot |C(h)| = 3^2 \cdot 2^{k+1},\qquad
        \sum_{(C(g), C(h))\in B} \frac{|C(g)| \cdot |C(h)|}{2} = 3(p-3)\cdot 2^{k+1}.
    \end{equation}
The sum of the values in \eqref{eq:chiw_33_d2k+2_1} gives the desired formula in the assertion 1. The assertion 2 is immediate by Lemma~\ref{lem:sq_cub_conj_cls_d2k+2p} and the character Table~\ref{tab:ch_d2k+2p}.
\end{proof}
By adding the equations in Claim 1 and 2, we get the desired formula.
\end{proof}

\subsection{Group $T'_{8 \cdot 3^k}$}

\begin{proof}[Proof of Lemma~\ref{lem:t83k_dim1_chiw_cubic}]
    By Lemma~\ref{lem:sq_cub_conj_cls_t83k}, $C_{i,j}^3$ and $C_{i',j'}^3$ are not equal to each other for $i \neq i'$. Therefore, it is enough to show that
    \begin{equation}
         \sum_{i=1}^7 \sum_{j,j'=0; 3j \equiv 3j' \bmod 3^k}^{3^{k-1}-1} \frac{|C_{i,j}| \cdot |C_{i,j'}|}{|C_{i,j}^3|}  = 2^3 \cdot 3^{k+2}.
    \end{equation}
   
    For simplicity of computation, we divide the sum into the on-diagonal part and the off-diagonal part with respect to indices $(j,j')$. The on-diagonal part is computed by using Lemma~\ref{lem:sq_cub_conj_cls_t83k} and Table~\ref{tab:ch_t83k} as follows.
    \begin{equation}\label{eq:on_diag_T83k_dim1}
        \begin{split}
            &  \sum_{i=1}^7 \sum_{j=0}^{3^{k-1}-1} \frac{|C_{i,j}|^2}{|C_{i,j}^3|}=  2^3 \cdot 3^{k+1}.
        \end{split}
    \end{equation}
    We then compute the off-diagonal part with respect to $(j,j')$. By Lemma~\ref{lem:sq_cub_conj_cls_t83k}, the sum with respect to $(j,j')$ runs over all indices satisfying  $|j'-j| = 3^{k-2}$ or $|j'-j| = 2 \cdot 3^{k-2}$. More concretely, such indices $(j,j')$ belong to the set
    \begin{equation}
    \begin{split}
       & \{ (j, j+3^{k-2}) \mid j =0, 1, \ldots, 2 \cdot 3^{k-2}-1 \}\cup \{(j, j + 2\cdot 3^{k-2}) \mid j=0, 1, \ldots, 3^{k-2}-1\} \\
        \cup &  \{ (j,j - 3^{k-2} ) \mid j =3^{k-2}, 3^{k-2} + 1, \ldots, 3 \cdot 3^{k-2}-1 \}\\
        \cup &\{(j, j - 2\cdot 3^{k-2}) \mid j=2\cdot 3^{k-2}, 2\cdot 3^{k-2} + 1, \ldots, 3\cdot 3^{k-2}-1\}.
        \end{split}
    \end{equation}
    Its cardinality is $2 \cdot 3^{k-1}$. Again by Lemma~\ref{lem:sq_cub_conj_cls_t83k}, for any such pair $(j,j')$, we have $|C_{i,j}^3| = |C_{i,j'}^3|= 1 \quad (i\neq 3)$, $|C_{3,j}^3|=|C_{3,j'}^3| = 6$.
    Therefore, by Table~\ref{tab:ch_t83k}, we can compute the off-diagonal part as follows:
    \begin{equation}\label{eq:off_diag_T83k_dim1}
        \begin{split}
            &  \sum_{i} \sum_{j\neq j'; j\equiv j' \bmod 3^{k-1}}\frac{|C_{i,j}| \cdot |C_{i,j'}|}{|C_{i,j}^3|} =  (2 \cdot 3^{k-1}) \Bigl( \frac{|C_{3,0}|\cdot |C_{3,3^{k-2}}|}{6} + \sum_{i \neq 3} |C_{i,0}|\cdot |C_{i,3^{k-2}}|   \Bigr) =  2^4 \cdot 3^{k+1}.
        \end{split}
    \end{equation}
    Hence, by adding on-diagonal part \eqref{eq:on_diag_T83k_dim1} with off-diagonal part \eqref{eq:off_diag_T83k_dim1},  we obtain the desired formula.
\end{proof}

\section{Irreducible representations of some spherical 3-manifold groups} \label{appendix:B}

\subsection{Binary dihedral group $D_{4p}^\ast$ with even $p$} \label{ss:D4p_even}

\begin{prop}\label{prop:irrep_D4p_even}
    Let $p$ be an even positive integer. The distinct irreducible representations of $D_{4p}^{\ast}$ consist of $4$ one-dimensional representations
\begin{equation}
\begin{split}
    &\rho_{V^{(1)}_1} (a) = 1,  \rho_{V^{(1)}_1} (x) = 1, \quad  \rho_{V^{(1)}_2}(a)  = 1,  \rho_{V^{(1)}_2} (x) = -1,\\
    &\rho_{V^{(1)}_3} (a) = -1,  \rho_{V^{(1)}_3} (x) = 1,\quad  \rho_{V^{(1)}_4} (a) = -1,  \rho_{V^{(1)}_4} (x) = -1,
    \end{split}
\end{equation}
and $p-1$ two-dimensional representations
\begin{equation}
    \rho_{V^{(2)}_{\lambda}}(a) = \begin{pmatrix}
        \zeta_{2p}^{\lambda} & 0 \\
        0 & \zeta_{2p}^{-\lambda}
    \end{pmatrix},
    \quad 
    \rho_{V^{(2)}_{\lambda}}(x) = \begin{pmatrix}
        0 & (\sqrt{-1})^{\lambda} \\
       (\sqrt{-1})^{\lambda} & 0
    \end{pmatrix}
    \quad ({\lambda}=1, \ldots, p-1).
\end{equation}
\end{prop}
By direct computation, we can show that these are indeed representations of $D_{4p}^{\ast}$. The irreducibility of two-dimensional representations can be checked directly as in \cite[\S 5.3]{Se77}. Namely, since $\rho_{V_{{\lambda}}^{(2)}}(a)$ is a diagonal matrix with distinct diagonal entries, one dimensional subspaces which can be stable are coordinate axes of $\C^2$ but they are not stable under $\rho_{V_{{\lambda}}^{(2)}}(x)$. Direct computation shows that these representations are pairwise non-isomorphic. These irreducible representations given above are the only irreducible ones since $4\cdot 1^2 + (p-1)\cdot 2^2 = 4p=|D_{4p}^{\ast}|$ (cf. \cite[\S 2.4. Remarks (1)]{Se77}).

\subsection{Binary dihedral group $D_{4p}^\ast$ with odd $p$} \label{ss:D4p_odd}

\begin{prop}\label{prop:irrep_D4p_odd}
Let $p$ be an odd positive integer. The distinct irreducible representations of $D_{4p}^{\ast}$ consist of $4$ one-dimensional representations
\begin{equation}
\begin{split}
    &\rho_{V^{(1)}_1} (a) = 1, \rho_{V^{(1)}_1}(x) = 1, \quad \rho_{V^{(1)}_2}(a) = 1, \rho_{V^{(1)}_2}(x) = -1,\\
    &\rho_{V^{(1)}_3}(a) = -1, \rho_{V^{(1)}_3}(x) = \sqrt{-1},\quad \rho_{V^{(1)}_4}(a) = -1, \rho_{V^{(1)}_4}(x) = -\sqrt{-1},
    \end{split}
\end{equation}
and $p-1$ two-dimensional representations
\begin{equation}
    \rho_{V^{(2)}_{\lambda}}(a) = \begin{pmatrix}
        \zeta_{2p}^{\lambda} & 0 \\
        0 & \zeta_{2p}^{-\lambda}
    \end{pmatrix},
    \quad 
    \rho_{V^{(2)}_{\lambda}}(x) = \begin{pmatrix}
        0 & (\sqrt{-1})^{\lambda} \\
       (\sqrt{-1})^{\lambda} & 0
    \end{pmatrix}
    \quad (\lambda=1, \ldots, p-1).
\end{equation}
\end{prop}
As in Section \ref{ss:D4p_even}, we can directly check that these representations are indeed representations of irreducible, pairwise non-isomorphic, and the only irreducible representations of $D_{4p}^{\ast}$ with odd $p$.

\subsection{Group $D'_{2^{k+2}p}$} \label{ss:D'}

\begin{prop}\label{prop:irrep_D'}
Let $k \geq 0$ and $p \geq 3$ be odd. The distinct irreducible representations of $D_{2^{k+2}p}'$ consist of $2^{k+2}$ one-dimensional irreducible representations
\begin{equation}
     \rho_{V^{(1)}_j} (x) = \zeta_{2^{k+2}}^j, \quad \rho_{V^{(1)}_j}(y) = 1 \quad (j=0,1, \ldots, 2^{k+2}-1),
\end{equation}
and $(p-1)/2 \cdot 2^{k+1}$ two-dimensional irreducible representations 
\begin{equation}
    \rho_{V^{(2)}_{s,t}}(x) = \begin{pmatrix}
        0 & \zeta_{2^{k+2}}^t\\
        \zeta_{2^{k+2}}^t & 0 
    \end{pmatrix},
    \quad 
    \rho_{V^{(2)}_{s,t}}(y) = \begin{pmatrix}
        \zeta_{p}^s & 0 \\
        0 & \zeta_{p}^{-s}  
    \end{pmatrix} \quad (0<s\leq  \frac{p-1}{2}, 0\leq t < 2^{k+1}).
\end{equation}
\end{prop}
One can directly check that these are representations of $D_{2^{k+2}p}'$. As in Section \ref{ss:D4p_even} and \ref{ss:D4p_odd}, direct computation proves that these representations are irreducible and pairwise nonisomorphic. Since $2^{k+2}\cdot 1^2 + (\frac{p-1}{2}\cdot 2^{k+1})\cdot 2^2= 2^{k+2}p=|D_{2^{k+2}p}'|$, these irreducible representations are the only irreducible ones. 

\subsection{Group $T'_{8\cdot 3^k}$} \label{ss:T'}

\begin{prop}\label{prop:irrep_T'}
Let $k$ be a positive integer. The distinct irreducible representations of $T'_{8\cdot 3^k}$ consist of $3^k$ one-dimensional representations
\begin{equation}
    \rho_{V^{(1)}_{\lambda}}(x) = \rho_{V^{(1)}_{\lambda}}(y) =1, \quad \rho_{V^{(1)}_{\lambda}}(z) = \zeta_{3^k}^{\lambda} \quad (\lambda \in \{0,1,\ldots, 3^k-1\}),
\end{equation}
$3^k$ two-dimensional representations 
\begin{align*}
    & \rho_{V^{(2)}_{\lambda}}(x) = \left(\begin{array}{rr}
\sqrt{-1} & 0 \\
0 & -\sqrt{-1}
\end{array}\right), \quad 
\rho_{V^{(2)}_{\lambda}}(y) = \left(\begin{array}{rr}
0 & 1 \\
-1 & 0
\end{array}\right),\\
& \rho_{V^{(2)}_{\lambda}}(z) = \left(\begin{array}{rr}
-\frac{1}{2} \sqrt{-1} - \frac{1}{2} & -\frac{1}{2} \sqrt{-1} - \frac{1}{2} \\
-\frac{1}{2}\sqrt{-1} + \frac{1}{2} & \frac{1}{2} \sqrt{-1} - \frac{1}{2}
\end{array}\right)\otimes \zeta_{3^k}^{\lambda} \quad (\lambda \in \{0,1,\ldots, 3^k -1\}),
\end{align*}
and $3^{k-1}$ three-dimensional representations 
\begin{equation}
    \rho_{V^{(3)}_{\lambda}}(x) = \left(\begin{array}{rrr}
1 & 0 & 0 \\
0 & -1 & 0 \\
0 & 0 & -1
\end{array}\right), \quad
\rho_{V^{(3)}_{\lambda}}(y) = \left(\begin{array}{rrr}
-1 & 0 & 0 \\
0 & 1 & 0 \\
0 & 0 & -1
\end{array}\right), \quad
\rho_{V^{(3)}_{\lambda}}(z) = \left(\begin{array}{rrr}
0 & 0 & \zeta_{3^k}^{\lambda} \\
\zeta_{3^k}^{\lambda} & 0 & 0 \\
0 & \zeta_{3^k}^{\lambda} & 0
\end{array}\right),
\end{equation}
where $\lambda \in \{0,1, \ldots, 3^{k-1}-1\}$. 
\end{prop}
The irreducibility of these representations can be checked in the same way as in Section \ref{ss:D4p_even}, \ref{ss:D4p_odd}, and \ref{ss:D'}. Direct computation confirms that none of these irreducible representations are isomorphic. The above representations are the only irreducible representations of $T'_{8\cdot 3^k}$ (up to isomorphism) since $3^k \cdot 1 + 3^k \cdot 2^2 + 3^{k-1} \cdot 3^2 = 8 \cdot 3^k = |T_{8 \cdot 3^k}'|$.

\section{$\Z\pi$-homology equivalence vs. 3-manifold over $K(\pi,1)$}\label{s:3-mfd-over-Kpi1}
A {\it 3-manifold over $K(\pi,1)$} is a pair $(M,f)$ of a 3-manifold $M$ and a continuous map $f\colon M\to K(\pi,1)$. We say that such pairs $(M,f)$ and $(M',f')$ are {\it $\pi$-diffeomorphic} if there is a diffeomorphism $g\colon M\to M'$ such that $f\simeq f'\circ g$ (we follow the terminology of \cite{HW}). Such objects are special cases of those studied in Turaev's HQFT (\cite{Tu,LT} etc.). 
We consider the case where $N$ is a closed orientable irreducible 3-manifold. In this case, $\Z\pi$-homology equivalences to $N$ are related to 3-manifolds over $K(\pi,1)$ as follows.

Since $\pi_2N=0$, we may consider $K(\pi,1)$ as obtained from $N$ by attaching cells of dimensions $\geq 4$ to kill $\pi_k$ for $k\geq 3$. Thus we may consider $N$ as the 3-skeleton $K(\pi,1)^{(3)}$ of $K(\pi,1)$. Let $\iota_N\colon N\to K(\pi,1)$ denote the inclusion. The following proposition is an analogue of \cite[Example~2.10]{GL} for closed orientable irreducible 3-manifolds.

\begin{prop}\label{prop:f_f'}
Let $N$, $M$ be closed orientable irreducible 3-manifolds. Let $\pi=\pi_1N$ and let $f,f'\colon M\to N$ be $\Z\pi$-homology equivalences. Let $\overline{f},\overline{f'}\colon M\to K(\pi,1)$ be the maps defined by the compositions $\iota_N\circ f,\iota_N\circ f'$, respectively. Then the following conditions (i)--(iii) are equivalent.
\[ {\rm{(i)}}\,\,f\simeq f',\qquad 
{\rm{(ii)}}\,\, \overline{f}\simeq\overline{f'},\qquad
{\rm{(iii)}}\,\, f_*=f_*'\in \Hom(\pi_1M,\pi).
\]
\end{prop}
\begin{proof}
We prove (i)$\Rightarrow$(iii)$\Rightarrow$(ii)$\Rightarrow$(i). (i)$\Rightarrow$(iii) is obvious. To prove (iii)$\Rightarrow$(ii), we assume (iii), namely, $f_*=f_*'\colon \pi_1M\to \pi$. Then the maps
$\phi_f,\phi_{f'}\colon K(\pi_1M,1)\to K(\pi,1)$
induced by $f,f'$, respectively, are homotopic. We may assume that they are extensions of $f,f'$ such that $\iota_N\circ f\simeq \phi_f\circ \iota_M$ and $\iota_N\circ f'\simeq \phi_{f'}\circ \iota_M$ in the following diagram.
\[ \xymatrix{
    M \ar@/^/[rr]^-{f}\ar@/_/[rr]_-{f'} \ar[d]_-{\iota_M} & & N \ar[d]^-{\iota_N}\\
    K(\pi_1M,1) \ar@/^/[rr]^-{\phi_f}\ar@/_/[rr]_-{\phi_{f'}} & & K(\pi,1)
}\]
Then (ii) follows since $\overline{f}\simeq\phi_f\circ \iota_M\simeq \phi_{f'}\circ\iota_M\simeq\overline{f'}$.

Next we assume (ii) and prove (i). It is known that if $N$ is closed, orientable and irreducible, but not spherical, then it is $K(\pi,1)$ (e.g., \cite[Theorem~1.7.3 and \S{3.2}]{AFW}), and the equivalence (i)$\Leftrightarrow$(ii) is obvious. Thus we need only to consider the case where $N$ is spherical, and we assume so. 
We take a cell decomposition of $M$ with one 0-cell and one 3-cell. Let $M^{(2)}$ denote the 2-skeleton of $M$. By cellular approximation, the homotopy $M^{(2)}\times I\to K(\pi,1)$ between the restrictions $\overline{f}|_{M^{(2)}}=f|_{M^{(2)}}$ and $\overline{f'}|_{M^{(2)}}=f'|_{M^{(2)}}$ can be homotoped into the 3-skeleton $K(\pi,1)^{(3)}=N$, and the resulting homotopy gives that between $f|_{M^{(2)}}$ and $f'|_{M^{(2)}}$ in $N$, hence by Hopf's degree theorem and the homotopy extension property for the cofibration $M^{(2)}\to M$, we have $f'\simeq f\# \psi$ for some $\psi\colon S^3\to N$. Now we recall the following.
\begin{itemize}
\item $f$ and $f'$ are degree one. The degree $\Omega_3(N)\to \Z$ is additive with respect to connected sum.
\item The universal covering $S^3\to N$ induces an isomorphism $\pi_3(S^3)\to \pi_3(N)$. This follows from the long exact sequence of homotopy groups for the $\pi$-bundle $S^3\to N$.
Hence $\deg\,{\psi}$ is a multiple of the degree $|\pi|$ of the universal covering $S^3\to N$.
\end{itemize}
We have $1=\deg\,{f'}=\deg\,{f}+\deg\,{\psi}=1+m|\pi|$,
which implies $m=0$, and $f'\simeq f$. This proves (i).
\end{proof}

When $N$ is a lens space, for which $\pi$ is abelian, the condition (iii) is equivalent to $f_*=f_*'\in\Hom(H_1(M),H_1(N))=H^1(M;\pi)$. 

\begin{cor}\label{cor:diffeomorphic}
Let $N$ be as in Proposition~\ref{prop:f_f'}.
Let $f\colon M\to N$ and $f'\colon M'\to N$ be two $\Z\pi$-homology equivalences from irreducible 3-manifolds. Then the following conditions are equivalent.
\begin{enumerate}
\item $f$ and $f'$ are diffeomorphically equivalent.
\item $(M,\overline{f})$ and $(M',\overline{f'})$ are $\pi$-diffeomorphic. 
\end{enumerate}
Hence $\calH(N)$ is canonically embedded into the set $\calH_3(\pi)$ of $\pi$-diffeomorphism classes of irreducible closed 3-manifolds over $K(\pi,1)$.
\end{cor}
\begin{proof}
That 1 implies 2 is obvious. For the converse, we assume that $(M,\overline{f})$ and $(M',\overline{f'})$ are $\pi$-diffeomorphic and prove that $f'\circ g\simeq f$ for some diffeomorphism $g\colon M\to M'$. By the assumption, we know that $\overline{f'\circ g}=\overline{f'}\circ g\simeq \overline{f}$ for a $\pi$-diffeomorphism $g\colon M\to M'$. Then by Proposition~\ref{prop:f_f'}, we have $f'\circ g\simeq f$. 
\end{proof}

Corollary~\ref{cor:diffeomorphic} shows that a $\Z\pi$-homology equivalence to $N$ is roughly a 3-manifold $M$ equipped with a structure of a principal $\pi$-bundle (or a normal $\pi$-covering) over it satisfying some homological triviality property of the total space $\widetilde{M}$: there is a $\pi$-equivariant map $\widetilde{M}\to \widetilde{N}$ that induces an isomorphism $H_*(\widetilde{M})\cong H_*(\widetilde{N})$, where $\widetilde{N}$ is $S^3$ or $\R^3$, depending on whether $N$ is spherical or $K(\pi,1)$. 

When $\pi$ is a finite abelian group, the invariant $\widehat{Z}_\Theta^\odd$ can be considered as that of pairs $(M,\phi)$ for some classes $\phi\in H^1(M;\pi)$, according to Corollary~\ref{cor:diffeomorphic}. If $\widehat{Z}_\Theta^\odd$ could be defined over the whole of $\calA_\Theta^\odd(\C\pi)$, it gives a refinement of the 2-loop part of the Kontsevich configuration space invariant in \cite{Kon} in the sense that the latter is recovered by applying the augmentation map $\ve$ to the decorations of 2-loop graphs. Refinements of quantum invariants that might be relevant are given or discussed in \cite{LT} and \cite[\S{11.5}]{Ohts2} etc.

\bibliographystyle{amsalpha}
\bibliography{refs}

\ 

\noindent
Hisatoshi Kodani kodani@imi.kyushu-u.ac.jp \\ 
Institute of Mathematics for Industry,
Kyushu University,
744, Motooka, Nishi-ku, Fukuoka, 819-0395,
Japan

\

\noindent
Tadayuki Watanabe tadayuki.watanabe@math.kyoto-u.ac.jp \\
Department of Mathematics, Faculty of Science, Kyoto University, 
Kitashirakawa, Oiwake-cho, Sakyo-ku, Kyoto 606-8502, Japan

\end{document}